\documentclass{amsart}

\usepackage{gensymb,epic,eepic,float,fullpage}
\usepackage{longtable}
\usepackage[usenames,dvipsnames]{color}
\usepackage{siunitx}
\usepackage{empheq}
\usepackage[mathscr]{euscript}
\usepackage[T1]{fontenc}
\usepackage{comment}
\usepackage{caption}
\usepackage{enumerate,enumitem}
\usepackage{amsmath,amsfonts,amssymb,amsthm}
\usepackage{mathtools}
\usepackage[colorinlistoftodos]{todonotes}
\usepackage{cancel}
\usepackage{stmaryrd}
\usepackage{url}
\usepackage{tikz}
\usepackage[outline]{contour}
\usetikzlibrary{graphs,patterns,decorations.markings,arrows.meta,matrix}
\usetikzlibrary{calc,decorations.pathmorphing,decorations.pathreplacing,shapes}
\usepackage{graphicx}

\usepackage{hyperref}

\usepackage{bm}
\usepackage{bbm}
\usepackage{mathrsfs}

\colorlet{lgray}{white!85!black}
\colorlet{dgray}{white!45!black}
\colorlet{lred}{white!85!red}
\colorlet{dred}{white!35!red}
\colorlet{lgreen}{white!60!green}
\colorlet{dgreen}{black!30!green}
\colorlet{lpurple}{white!60!purple}
\colorlet{lblue}{white!60!blue}
\definecolor{green}{rgb}{0.1,0.8,0.1}
\definecolor{yellow}{rgb}{1.0,0.85,0.25}
\definecolor{purple}{rgb}{1.0, 0, 1.0}
\definecolor{blue}{rgb}{0, 0, 1.0}

 \renewcommand{\tikz}[2]{
\begin{tikzpicture}[scale=#1,baseline=(current bounding box.center),>=stealth]
#2
\end{tikzpicture}}

\newcommand{\tikzbase}[3]{
\begin{tikzpicture}[scale=#1,baseline={([yshift=#2]current bounding box.center)},>=stealth]
#3
\end{tikzpicture}}

\newcommand\restr[2]{{% we make the whole thing an ordinary symbol
  \left.\kern-\nulldelimiterspace % automatically resize the bar with \right
  #1 % the function
  \vphantom{\big|} % pretend it's a little taller at normal size
  \right|_{#2} % this is the delimitert
 }}

\tikzstyle{fused}=[lgray, line width=3pt, arrows={-Stealth[scale=1.1,length=10, width=10,dgray]}]
\tikzstyle{fused*}=[lgray, line width=3pt]

\newtheorem{prop}{Proposition}[section]
\newtheorem*{prop*}{Proposition}
\newtheorem{theo}[prop]{Theorem}
\newtheorem*{theo*}{Theorem}
\newtheorem{theoAlph}{Theorem}

\newtheorem{lem}[prop]{Lemma}
\newtheorem{cor}[prop]{Corollary}
\theoremstyle{definition}

\newtheorem{rem}[prop]{Remark}

\numberwithin{equation}{section}

\renewcommand{\(}{\left (}
\renewcommand{\)}{\right )}

\newcommand{\bA}{\bm A}
\newcommand{\bB}{\bm B}
\newcommand{\bC}{\bm C}

\newcommand{\bI}{\bm I}
\newcommand{\bJ}{\bm J}
\newcommand{\bK}{\bm K}
\newcommand{\bL}{\bm L}
\newcommand{\bR}{\bm R}
\newcommand{\bP}{\bm P}
\newcommand{\bQ}{\bm Q}
\newcommand{\bX}{\bm X}
\newcommand{\bY}{\bm Y}

\newcommand{\tmho}{\widetilde{\mho}}

\newcommand \calA{\mathcal A}
\newcommand \calB{\mathcal B}
\newcommand \calC{\mathcal C}

\newcommand \F{\mathbb F}
\newcommand \C{\mathbb C}
\newcommand \B{\mathbb B}

\newcommand \Z{\mathbb Z}
\newcommand \Y{\mathbb Y}

\newcommand{\be}{\bm e}
\newcommand{\x}{\mathbf x}

\renewcommand{\sl}{\mathfrak{sl}}
\newcommand{\ve}{\varepsilon}

\DeclareMathOperator{\Sym}{Sym}

\title{Representation theoretic interpretation and interpolation properties of inhomogeneous spin $q$-Whittaker polynomials}

\author{Sergei Korotkikh}

\begin{document}
\maketitle

\begin{abstract} We establish new properties of inhomogeneous spin $q$-Whittaker polynomials, which are symmetric polynomials generalizing $t=0$ Macdonald polynomials. We show that these polynomials are defined in terms of a vertex model, whose weights come not from an $R$-matrix, as is often the case, but from other intertwining operators of $U'_q(\widehat{\sl}_2)$-modules. Using this construction, we are able to prove a Cauchy-type identity for inhomogeneous spin $q$-Whittaker polynomials in full generality. Moreover, we are able to characterize spin $q$-Whittaker polynomials in terms of vanishing at certain points, and we find interpolation analogues of $q$-Whittaker and elementary symmetric polynomials.
\end{abstract}

\tableofcontents

\section{Introduction}

\subsection{Overview}

\emph{Inhomogeneous spin $q$-Whittaker polynomials} are symmetric polynomials generalizing the classical $q$-Whittaker functions (specialization of Macdonald symmetric functions at $t=0$) by adding two sequences of  parameters $\calA=(a_0, a_1, \dots)$ and $\calB=(b_0, b_1, \dots)$. In full generality they were recently constructed in \cite{BK21} using integrable lattice models, and some particular degenerations were described earlier in \cite{BW17}, \cite{MP20}. 

One of the main features of the spin $q$-Whittaker polynomials can be summarized as follows: while the new parameters $\calA, \calB$ are added in a non-trivial way, the resulting functions share several defining algebraic-combinatorial properties with the classical $q$-Whittaker functions. Namely, there exist spin $q$-Whittaker analogues of the branching rule and the (skew) Cauchy identity. Moreover, there also exists an analogue of the dual Cauchy identity, which involves spin $q$-Whittaker polynomials and \emph{spin Hall-Littlewood functions}: a generalization of Hall-Littlewood functions constructed from integrable lattice models in \cite{Bor14}, \cite{BP16}. Other known properties of spin $q$-Whittaker polynomials include formulae for explicit action of certain first order difference operators \cite{MP20}, and applications to various stochastic models from \emph{integrable probability} \cite{MP20}, \cite{K21}. 

The above properties were proved using an explicit construction of inhomogeneous spin $q$-Whittaker polynomials in terms of an integrable vertex model, called the \emph{$q$-Hahn vertex model}. In particular, the Cauchy and dual Cauchy identities follow directly from appropriate modifications of the Yang-Baxter equation using \emph{zipper} or \emph{train} argument. This idea is not new: vertex model constructions and Yang-Baxter equations were used to study numerous other special functions, examples can be found in  \cite{Tsi05}, \cite{ZJ09}, \cite{BBS09}, \cite{Kor11}, \cite{Bor14}, \cite{BP16}, \cite{BW18}, \cite{BSW19}, \cite{ABW21}. However, there is an important difference distinguishing the $q$-Hahn vertex model among the other integrable vertex models. Usually, integrable vertex models originate from matrix coefficients of $R$-matrices acting on specific representations of quantum groups, and the Yang-Baxter equation for vertex models is a reformulation of the corresponding equation for quantum groups. But the general $q$-Hahn vertex model cannot be presented in such a fashion, and in \cite{BK21} somewhat different combinatorial methods were used to show integrability of this model and deduce properties of inhomogeneous spin $q$-Whittaker polynomials.

In this work we establish several new properties of the spin $q$-Whittaker polynomials that can be grouped in two clusters. The first group is focused around the integrability of the $q$-Hahn model. It turns out that, while the $q$-Hahn vertex model does not come from an $R$-matrix of a quantum group, it still has a quantum group interpretation, where instead of the $R$-matrix we consider different intertwining operators between tensor products of certain representations of $U'_q(\widehat{\frak{sl}}_2)$. Unlike the $R$-matrix, these intertwining operators change the tensor factors instead of just permuting them. The integrability and the (modified) Yang-Baxter equations for the vertex model follow immediately from this observation, and using this new point of view we can finish the proof of the Cauchy identity for the spin $q$-Whittaker polynomials, which was partially done in \cite{BK21} but was not completed in full generality.

The second group of our new results reveals an unexpected connection of spin $q$-Whittaker polynomials with another area of the theory of symmetric functions: \emph{interpolation symmetric functions}. These are inhomogeneous deformations of classical symmetric functions which can be characterized by vanishing at specific points. The main known classes of interpolation symmetric functions are \emph{factorial Schur functions} and \emph{interpolation Macdonald functions}\footnote{It is worth noting that for the supersymmetric functions there is also the class of \emph{interpolation $Q$-Schur functions} \cite{Iv03}.}. The former ones are well-studied and their interpolation properties have appeared in the contexts of Capelli identities, multiplicity free spaces and asymptotic representation theory, see \cite{HU91}, \cite{S94},  \cite{OO96a}. The interpolation Macdonald functions are more complicated and still somewhat mysterious, however a number of nice combinatorial properties is known about them and their degenerations, see \cite{S96},  \cite{Ok96a}, \cite{Ok96b}, \cite{KS96}, \cite{Olsh17}, \cite{Cue17}, \cite{Pet20}. Moreover, there exist more general type $BC_n$ interpolation Macdonald polynomials, which are Laurent symmetric polynomials with connections to Koornwinder polynomials and applications to multivariate analogues of $q$-hypergeometric transformations, see \cite{Ok96c}, \cite{Rai01}, \cite{Rai04}, \cite{Krn14}.

In \cite{Ok97} it was shown that, under a certain constraint, all  interpolation symmetric polynomials of interest generally fall in one of three classes: factorial monomial functions (which were considered trivial), factorial Schur functions and (type $BC_n$) interpolation Macdonald functions. However, we show that inhomogeneous spin $q$-Whittaker polynomials also can be characterized by vanishing at specific points. Moreover, by specializing inhomogeneous spin $q$-Whittaker polynomials, one can obtain new symmetric functions which should fill the place of \emph{interpolation $q$-Whittaker} and \emph{interpolation elementary} functions. These new interpolation symmetric functions depend on a countable family of parameters (similarly to the factorial Schur functions and unlike the interpolation Macdonald functions) and both do not naturally arise via the interpolation Macdonald functions. Finally, we are able to find and prove a parallel version of the classification from \cite{Ok97}, which identifies our new interpolation functions as unique functions satisfying a constraint similar to the one from \cite{Ok97}.

Overall, our new results give two descriptions of inhomogeneous spin $q$-Whittaker polynomials: an improved constructions in terms of the vertex model from \cite{BK21} and a characterization in terms of vanishing at specific points. Combination of these two properties is rare: so far and to the best of our knowledge the only other symmetric functions simultaneously having both vertex model description and a vanishing property were factorial Schur functions\footnote{See \cite{BMN14} for a vertex model description of factorial Schur functions, leading to the vanishing property.}\footnote{There is also a degeneration of interpolation Macdonald polynomials with a vertex model construction, namely, interpolation Hall-Littlewood functions from \cite{Cue17}, \cite{Pet20}. However, interpolation Hall-Littlewood functions do not in fact have interpolation properties, so we do not include them.}. We still don't know if this situation is an artifact specific, for some reason, to the Schur and $q$-Whittaker levels, or it might indicate existence of  yet uncovered even more general symmetric functions with similar properties, but both possibilities look exciting for us. 

Below we briefly state the main results of this work.

\subsection{$q$-Hahn vertex model and intertwining operators} Inhomogeneous spin $q$-Whittaker polynomials $\F_{\lambda}(x_1, \dots, x_n\mid \calA, \calB)$ are defined as partition functions of a vertex model consisting of a square grid of vertices with weights
$$
\tikz{1.3}{
	\draw[fused] (-1,0) -- (1,0);
	\draw[fused] (0,-1) -- (0,1);
	\node[below] at (-1,-0.1) {\tiny \shortstack[c]{$I$}};\node[below] at (1,-0.1) {\tiny \shortstack[c]{$L$}};
	\node[below] at (0,-1) {\tiny $J$};\node[above] at (0,1) {\tiny $K$};
}=\delta_{I+J=K+L}\delta_{J\geq L} (x_i/b_j)^{L}\frac{(a_{i+j}/x_i;q)_{L}(x_i/b_j;q)_{J-L}}{(a_{i+j}/b_j;q)_{J}}\frac{(q;q)_{J}}{(q;q)_{L}(q;q)_{J-L}}.
$$
Here $i,j$ are coordinates of the vertex, $(x;q)_k$ denote the $q$-Pochhammer symbol recalled in Section \ref{notation} below, $a_i, b_i$ are parameters from $\calA,\calB$, and $I,J,K,L\in\mathbb Z_{\geq 0}$ are integers, representing a configuration of edges around the vertex. Note that these weights are orthogonality weights for $q$-Hahn polynomials, so they are traditionally called the $q$-Hahn weights. The exact definition of the spin $q$-Whittaker polynomials $\F_{\lambda}(x_1, \dots, x_n\mid \calA, \calB)$ via such weights is given in Section \ref{sqW-sect} of the text. 

Our first main result states that $q$-Hahn vertex weights are in fact matrix coefficients of an isomorphism between two generically irreducible representations of the affine quantum algebra $U'_q(\widehat{\frak {sl}}_2)$. The representations in question are tensor products of the evaluation modules $V(s)_z$, which are induced from the irreducible $U_q(\frak{sl}_2)$-module $V(s)$ with highest weight $s$ along the evaluation map $\mathrm{ev}_z:U'_q(\widehat{\frak{sl}_2})\to U_q(\frak{sl}_2)$. As in the classical $\frak{sl}_2$ case, representations $V(s)_z$ have a natural countable basis $|v_I\rangle$ consisting of weight vectors.

\begin{theoAlph}[{Theorem \ref{basicintertwiner} in the text}] For generic parameters $a_1,a_2,b_1,b_2$ the representations $V(a_1/b_1)_{a_1b_1}\otimes V(a_2/b_2)_{a_2b_2}$ and $V(a_1/b_2)_{a_1b_2}\otimes V(a_2/b_1)_{a_2b_1}$ are irreducible and isomorphic, with the isomorphism $W^b$ given explicitly by
$$
W^b:V(a_1/b_1)_{a_1b_1}\otimes V(a_2/b_2)_{a_2b_2}\to V(a_1/b_2)_{a_1b_2}\otimes V(a_2/b_1)_{a_2b_1},
$$
$$
\langle v_{K}\otimes v_{L}| W^b|v_{I}\otimes v_{J}\rangle = \delta_{I+ J=K+L}\delta_{J\geq L}\frac{(b_2/a_1)^L}{(b_1/a_1)^J}
(b_1/b_2)^{2L}\frac{(a_2^2/b_1^2;q)_{L} (b_1^2/b_2^2;q)_{J-L}}{(a_2^2/b_2^2;q)_{J}}\frac{(q;q)_{J}}{(q;q)_{L}(q;q)_{J-L}}.
$$
\end{theoAlph}

The fact that the representations $V(a_1/b_1)_{a_1b_1}\otimes V(a_2/b_2)_{a_2b_2}$ and $V(a_1/b_2)_{a_1b_2}\otimes V(a_2/b_1)_{a_2b_1}$ are isomorphic is not new: one can show that for generic parameters both $V(a_1/b_1)_{a_1b_1}\otimes V(a_2/b_2)_{a_2b_2}$ and $V(a_1/b_2)_{a_1b_2}\otimes V(a_2/b_1)_{a_2b_1}$ are irreducible and have the same highest $l$-weight $\frac{(a_1^{-1}-a_1u)(a_2^{-1}-a_2u)}{(b_1^{-1}-b_1u)(b_2^{-1}-b_2u)}$ ($l$-weights are rational functions in $u$, which extend the usual weights to $U'_q(\frak{sl}_2)$-modules, \emph{cf.} \cite{MY12}). The nontriviality of our result is the explicit form for matrix coefficients of this isomorphism, which have a surprisingly simple form identical to the $q$-Hahn weights. Our result can also be extended to the higher rank case, see Theorem \ref{basicintertwiner} in the main text for the $U'_q(\widehat{\frak{sl}}_n)$ version.

We can use the expression for the isomorphism $W^b$ above to deduce two other results. First, we can obtain an explicit expression for the fully fused $R$-matrix $R:V(a_1/b_1)_{a_1b_1}\otimes V(a_2/b_2)_{a_2b_2}\to V(a_2/b_2)_{a_2b_2}\otimes V(a_1/b_1)_{a_1b_1}$, reproducing in a new way the results from \cite{Man14}, \cite{BM16}, see Proposition \ref{Rexpression} in the text. Another consequence is the following Cauchy identity for spin $q$-Whittaker polynomials $\F_{\lambda}(x_1, \dots, x_n\mid\calA,\calB)$:
\begin{theoAlph}[{Theorem \ref{cauchyTheo} in the text}] The following equality of formal power series in $x,y,q$ holds:
\begin{equation*}
\sum_{\lambda}\F_{\lambda}(x_1, \dots, x_n\mid\calA,\calB)\F^*_{\lambda}(y_1, \dots, y_m\mid \overline{\calB}, \overline{\calA})=\prod_{i=1}^n\prod_{j=1}^m\frac{(a_iy_j;q)_\infty(x_i/b_j;q)_\infty}{(x_iy_j;q)_\infty(a_i/b_j;q)_\infty},
\end{equation*}
where the sum is over all partitions $\lambda$.
\end{theoAlph}
Here $\F^*_{\lambda}(y_1, \dots, x_n\mid\calB,\calA)$ is an explicit renormalization of spin $q$-Whittaker polynomials, with parameters $a_i,b_i$ swapped and inverted. The proof of the Cauchy identity is based on an algebraic identity, obtained by choosing specific parameters $s_i, s_i', z_i, z_i'$ and constructing an isomorphism $V(s_1)_{z_1}\otimes V(s_2)_{z_2}\otimes V(s_3)_{z_3}\cong V(s'_1)_{z'_1}\otimes V(s'_2)_{z'_2}\otimes V(s'_3)_{z'_3}$ of irreducible representations in two different ways, using our explicit expressions for isomorphisms $W^b$ and $R$.

\subsection{Vanishing properties and interpolation characterization} For this group of results the starting point is the following vanishing property of inhomogeneous spin $q$-Whittaker polynomials $\F_\lambda(x_1, \dots, x_n\mid\calA,\calB)$: 
$$
\F_{\lambda}(a_1q^{\mu_1-\mu_2}, a_2q^{\mu_2-\mu_3},\dots, a_nq^{\mu_n}\mid \calA,\calB)=0,\qquad \text{unless}\ \lambda\subseteq\mu.
$$
This vanishing property is new, and it leads to an alternative characterization of the inhomogeneous spin $q$-Whittaker polynomials. Let $\x^n_{\calA}(\mu)=(a_1q^{\mu_1-\mu_2}, a_2q^{\mu_2-\mu_3},\dots, a_nq^{\mu_n})$, and let $\mathcal G^n_{0}\subset \mathcal G^n_{1}\subset \mathcal G^n_{2}\subset \dots$ denote a certain deformation of the natural filtration of the algebra of symmetric functions in $n$ variables, which is defined before Theorem \ref{charTheo} in the text.
\begin{theoAlph}[{Theorem \ref{charTheo} in the text}] For each partition $\lambda\in\Y^n$ the spin $q$-Whittaker polynomial $\F_\lambda(x_1, \dots, x_n\mid \calA,\calB)$ is uniquely characterized, up to a scalar, by the following properties:
\begin{enumerate}
\item $\F_\lambda(x_1, \dots, x_n\mid\calA,\calB)\in\mathcal G^{n}_{|\lambda|}$.
\item For any partition $\mu$ such that $|\mu|\leq |\lambda|$ and $\mu\neq\lambda$ we have $\F_\lambda(\x^n_{\calA}(\mu)\mid\calA,\calB)=0$.
\item $\F_\lambda(\x^n_{\calA}(\lambda)\mid\calA,\calB)\neq 0$.
\end{enumerate}
\end{theoAlph}
For the normalization we use the values $\F_\lambda(\x^n_{\calA}(\lambda)\mid\calA,\calB)$ from (3) can be found explicitly using the vertex model description and they are products of $q$-Pochhammer symbols, see Proposition \ref{vanishing} in the text. Note that in the characterization above the parameters $a_i$ from $\calA$ determine the interpolation points. The dependence on the second family $\calB$ is in fact hidden in the subspaces $\mathcal G^n_{k}$. 

These vanishing and characterization properties closely resemble the defining properties of interpolation symmetric polynomials, which are families of symmetric polynomials  $F_{\lambda}(x_1, \dots, x_n)$ satisfying
\begin{enumerate}[label=(\roman*)]
\item $\deg F_\lambda(x_1, \dots, x_n)\leq |\lambda|$;
\item $F_{\lambda}(\mho(\mu))=0$ unless $\lambda\subset\mu$;
\item $F_{\lambda}(\mho(\lambda))\neq 0$
\end{enumerate}
for some family $\mho(\mu)$ of $n$-dimensional points enumerated by partitions $\mu$. Property (ii), called the \emph{vanishing property}, is the non-trivial one: it overdefines the functions $f_{\lambda}$ and  makes the existence of these functions for a given family $\mho$ exceptional. The factorial Schur functions and the interpolation Macdonald functions are examples of such functions $F_\lambda$ for certain choices of $\mho$, in both these examples $\mho$  is of the form
$$
\mho(\mu)=(f_1(\mu_1), \dots, f_n(\mu_n))
$$
for some functions $f_1, \dots, f_n$.

The spin $q$-Whittaker functions $\F_\lambda(x_1, \dots, x_n\mid \calA,\calB)$ almost satisfy the properties of interpolation symmetric functions $F_\lambda$, but in the first property we have the function spaces $\mathcal G^n_k$ instead of the natural degree filtration of the algebra of symmetric polynomials. However, when the parameters $b_i$ from $\calB$ tend to $\infty$, the subspaces $\mathcal G^n_{k}$ degenerate to this natural filtration, and in this way we can find two degenerations of the spin $q$-Whittaker functions, which satisfy properties (i)-(iii) for certain choices of $\mho$. Namely, when $b_i\to\infty$ we get functions $\widetilde{\F}_\lambda(x_1, \dots, x_n\mid\calA,\infty)$, which satisfy the interpolation properties (i)-(iii) with 
$$
\mho(\mu)=(a_1q^{\mu_1-\mu_2}, a_2q^{\mu_2-\mu_3},\dots, a_nq^{\mu_n}).
$$
We call these functions \emph{interpolation $q$-Whittaker polynomials}, as their top homogeneous components coincide with the usual $q$-Whittaker functions.

The other degeneration is obtained by setting $b_i=\infty$ and considering the following limit regime:
$$
x_i=e^{\ve r_i}, \quad a_i=e^{\ve c_i}, \quad b_i\to\infty,\quad q=e^{\ve},\quad \ve\to 0.
$$
As a result we obtain symmetric polynomials $\F^{el}(r_1, \dots, r_n\mid \calC, \infty)$ in $r_1, \dots, r_n$, depending on a family of parameters $\calC=(c_1, c_2, \dots)$ and satisfying an interpolation property with 
$$
\mho(\mu)=(c_1+\mu_1-\mu_2, c_2+\mu_2-\mu_3,\dots, c_n+\mu_n).
$$
Moreover, the top homogeneous component of $\F^{el}(r_1, \dots, r_n\mid \calC, \infty)$ coincides with the elementary symmetric polynomial $e_{\lambda'}(r_1, \dots, r_n)$ for the conjugated partition $\lambda'$, so we call this second degeneration \emph{elementary interpolation polynomials}.

As a final result of our work we show that these two degenerations, interpolation $q$-Whittaker polynomials and elementary interpolation polynomials, exhaust all interpolation symmetric polynomials for $\mho(\mu)$ of the form
$$
\mho(\mu)=(f_1(\mu_1-\mu_2), \dots, f_n(\mu_n)).
$$
A precise result is stated in Theorem \ref{perfectTheorem} of the text, below we give its brief reformulation:
\begin{theoAlph}[{Theorem \ref{perfectTheorem} in the text}]
Assume that $n\geq 3$ and $\mho(\mu)$ is a collection of interpolation points of the form
$$
\mho(\mu)=(f_1(\mu_1-\mu_2), \dots, f_n(\mu_n)).
$$
Then the interpolation polynomials $F_\lambda$ satisfying the properties (i)-(iii) exist only if either
$$
\mho(\mu)=(c+a_1q^{\mu_1-\mu_2}, c+a_2q^{\mu_2-\mu_3},\dots, c+a_nq^{\mu_n})
$$
for some $c,q, a_1, a_2, \dots, a_n$, or if
$$
\mho(\mu)=(c_1+d(\mu_1-\mu_2), c_2+d(\mu_2-\mu_3),\dots, c_n+d\mu_n)
$$
for some $d, c_1, c_2, \dots, c _n$. In these cases they coincide with the interpolation $q$-Whittaker polynomials and the elementary interpolation polynomials, respectively.
\end{theoAlph}

\subsection{Layout of the paper.} 
In Section \ref{quantum-sect} we recall definitions and necessary properties of the affine quantum algebra $U'_q(\widehat{\sl}_{n+1})$ and its representations $V(s)_z$. In Section \ref{sect-intertw} we introduce explicit expressions for intertwining operators between tensor products $V(s_1)_{z_1}\otimes V(s_2)_{z_2}$, and use them to deduce equations needed for the later parts. Section \ref{exchange-sect} is devoted to row operators and exchange relations, which are a convenient technical tool for working with the vertex model description of the spin $q$-Whittaker polynomials. In Section \ref{sqW-sect} we collect all results about spin $q$-Whittaker polynomials, proving the Cauchy identity and the vanishing property. In the same section we also describe degenerations of spin $q$-Whittaker functions: interpolation $q$-Whittaker and elementary functions. In Section \ref{interpolationSect} we discuss the general classification for interpolation symmetric polynomials, and prove the version covering interpolation $q$-Whittaker and elementary functions.

\subsection{Notation}\label{notation}Throughout the work we treat $q^{\frac12}$ either as a fixed transcendental complex number such that $|q^{\frac12}|<1$, or as a formal variable. In both cases we set $q=(q^{\frac12})^2$ and
$$
[n]_q:=\frac{q^{\frac{n}{2}}-q^{-\frac n2}}{q^{\frac12}-q^{-\frac12}},\qquad [n]_q!:=\prod_{i=1}^n[i]_q,\qquad {\ n+m\ \brack m}_q:=\frac{[n+m]_q!}{[n]_q![m]_q!},
$$ 
where $m,n\in\Z_{\geq 0}$. We define the $q$-Pochhammer symbol by
$$
(x;q)_n:=\begin{cases}
\prod_{i=1}^{n}(1-xq^{i-1}), \quad &n\geq 0,\\
\prod_{i=1}^{-n}(1-xq^{-i})^{-1},\quad &n<0.
\end{cases}
$$
Let $(x;q)_{\infty}:=\lim_{n\to \infty}(x;q)_n$, which is well-defined both when $q$ is a number such that $|q|<1$ and when $q$ is a formal variable. Note that we can express the $q$-binomial coefficients in the following ways
$$
{\ n+m\ \brack m}_q=q^{-\frac{nm}{2}}\frac{(q;q)_{n+m}}{(q;q)_n(q;q)_m}=q^{-\frac{nm}{2}}\frac{(q^{n+1};q)_m}{(q;q)_m}=q^{-\frac{nm}{2}}\frac{(q^{m+1};q)_n}{(q;q)_n},
$$
allowing us to extend the definition to the case when either $n$ or $m$ are negative, but not both simultaneously.

A \emph{partition} $\lambda$ is an infinite sequence $(\lambda_1,\lambda_2,\lambda_3,\dots)$ of nonnegative integers such that 
$$
\lambda_1\geq\lambda_2\geq\lambda_3\geq\dots\geq 0
$$
and all but finitely many of $\lambda_i$ are equal to $0$. Sometimes we omit a tail consisting of zeroes, writing $(\lambda_1,\dots,\lambda_n)$ instead of $(\lambda_1,\dots,\lambda_n,0,0,\dots)$. The coordinates $\lambda_i$ are called the parts of the partition $\lambda$, the number of nonzero parts $\lambda_i$ is the \emph{length} $l(\lambda)$ of $\lambda$ and the \emph{weight} is defined by $|\lambda|:=\lambda_1+\lambda_2+\dots$. We use $\Y^n$ to denote the set of partitions of length at most $n$. For a pair of partitions $\lambda,\mu$, we write $\mu\subset\lambda$ if and only if for any $i\geq 1$ we have $\mu_i\leq\lambda_i$. Furthermore, we say the partition $\lambda$ \emph{interlaces} the partition $\mu$, and write $\lambda\succ\mu$, if $\lambda_i\geq\mu_i\geq\lambda_{i+1}$ for all $i\geq 1$. Note that in this case we also have $0\leq l(\lambda)-l(\mu)\leq 1$.

Whenever we claim that a statement holds for generic complex parameters $a_1,\dots, a_n$, we mean that there exists a countable collection of non-constant polynomials $F_i\in\mathbb C[x_1, \dots, x_n]$ such that the statement holds for all $(a_1, \dots, a_n)\in \C^n$ satisfying $F_i(a_1, \dots, a_n)\neq 0$ for all $i$. In particular, the set of such $(a_1, \dots, a_n)$ is dense in $\mathbb C^n$.

\subsection{Acknowledgements} The author is grateful to Alexei Borodin for his constant support and fruitful suggestions during this project, and to Matteo Mucciconi for his significant help with initiating this project. The author also would like to thank Pavel Etingof and Evgeny Mukhin for their clarifications regarding affine quantum groups and Grigori Olshanski for his explanation of the context around interpolation symmetric functions.

%%%%%%%%%%%%

\section{Quantum affine algebra $U_q'(\widehat{\sl}_{n+1})$ and representations $V(s)_z$} \label{quantum-sect}

In this section we describe the necessary background on the quantum affine algebra $U_q'(\widehat\sl_{n+1})$ and define the representations $V(s)_z$, which play the main role in Section \ref{sect-intertw}. Throughout this section $q^{\frac12}$ is assumed to be a transcendental complex number, though all results are still valid when $q^{\frac12}$ is a formal variable. References for the material in this section are  \cite[Chapter 12]{CP94a}, \cite{CP94b}, and \cite{MY12}.

\subsection{Quantum affine algebra $U'_q(\widehat{\sl}_{n+1} )$} Let $\mathfrak{h}$ denote the Cartan subalgebra of $\mathfrak{sl}_{n+1}$, which we identify with its dual $\mathfrak{h}^\vee$ via the Killing form $(\cdot, \cdot )$. Let $\alpha_i$ denote the simple roots of $\mathfrak{sl}_{n+1}$, and $C_{i,j}$ denote the corresponding Cartan matrix. We enumerate the nodes $I=\{1, \dots, n\}$ of the Dynkin diagram in a way such that $C_{i,i}=2$ and $C_{i,i-1}=C_{i-1,i}=-1$, while $C_{i,j}=0$ for $|i-j|>1$. We also assume that $(\alpha_i, \alpha_i)=2$, so the simple coroots $\alpha_i^{\vee}$ coincide with the roots $\alpha_i$. We use $P$ ($P^+$ and $P^-$) to denote the $\Z$-span ($\Z_{\geq 0}$-span and $\Z_{\leq 0}$-span correspondingly) of the simple roots $\alpha_i$, while $Q:=\{\mu\in\mathfrak h \mid (\mu, \alpha_i)\in\Z\}$ denotes the lattice of integral weights.

Let $\theta=\alpha_1+\alpha_2+\dots+\alpha_n$ be the maximal positive root, $\widehat{I}:=\{0, 1, \dots, n\}$ and $\widehat{C}_{i,j}$ be the extended Cartan matrix
$$
\widehat{C}_{0,0}=(-\theta,-\theta)=2, \qquad \widehat{C}_{0,i}=\widehat{C}_{i,0}=-(\alpha_i, \theta), \qquad \widehat C_{i,j}=C_{i,j}=(\alpha_i,\alpha_j), \qquad i,j\in I.
$$
The \emph{quantum affine algebra} $U'_q(\widehat{\sl}_{n+1})$\footnote{The prime in $U_q'(\widehat\sl_{n+1})$ reflects that our definition does not include the additional generator $d$, which is sometimes used in the context of affine algebras.} is the unital associative algebra over $\mathbb C$ with generators $\{k_i^{\pm 1}\}_{i\in\widehat I}$, $\{x_i^{\pm}\}_{i\in\widehat I}$ and relations
$$
k_ik_i^{-1}=k_i^{-1}k_i=1, \qquad k_ik_j=k_jk_i,
$$ 
\begin{equation}
\label{slrel}
k_ix_j^{\pm}=q^{\pm\frac12  \widehat{C}_{i,j}} x^{\pm}_jk_i,\qquad [x_i^+,x_j^-]=\delta_{i,j}\frac{k_i-k_i^{-1}}{q^{\frac12}-q^{-\frac12}},
\end{equation}
$$
\sum_{r=0}^{1-\widehat{C}_{ij}}(-1)^r{1-\widehat{C}_{ij}\brack r}_q(x^\pm_i)^rx^{\pm}_j(x^{\pm}_{i})^{1-\widehat{C}_{i,j}-r}=0 \quad \text{for\ } i\neq j.
$$
The last relation is called the (quantum) \emph{Serre relation}. The algebra $U'_q(\widehat{\sl}_{n+1})$ has a Hopf algebra structure which is given by
$$
\Delta k_i=k_i\otimes k_i, \qquad \Delta x_i^+=x_i^+\otimes k_i+1\otimes x_i^+,\qquad \Delta x_i^-=x_i^-\otimes 1 + k_i^{-1}\otimes x_i^{-},
$$
$$
S(x^+_i)=-x_i^+k_i^{-1}, \qquad S(x^-_i)=-k_ix_i^-,\qquad S(k_i^{\pm 1})=k_i^{\mp 1},
$$
$$
\epsilon(k_i)=1,\qquad \epsilon(x_i^+)=\epsilon(x_i^-)=0.
$$

There is another presentation of $U'_q(\widehat{\sl}_{n+1})$, introduced by Drinfeld in \cite{Dri88}. In this presentation the algebra is generated by $\{k_i^{\pm 1}\}_{i\in I}$, $\{h_{i,m}\}_{i\in I, m\in \mathbb Z\backslash\{0\}}$, $\{x^{\pm}_{i,m}\}_{i\in I, m\in\mathbb Z}$, $c^{\pm 1}$, with the relations
$$
[c^{\pm1}, U_q'(\widehat{\sl}_{n+1})]=0,\qquad k_ik_i^{-1}=k_i^{-1}k_i=1, \qquad k_ih_{j,m}=h_{j,m}k_i, \qquad k_ix^{\pm}_{j,m}=q^{\pm\frac 12 C_{i,j}}x^{\pm}_{j,m}k_i, 
$$ 
$$
 [h_{i,m}, h_{j,l}]=\delta_{m,-l}\frac{1}{m}[m C_{i,j}]_q\frac{c^m-c^{-m}}{q^{\frac12}-q^{-\frac12}}, \qquad [h_{i,m},x^{\pm}_{j,l}]=\pm\frac{1}{m}[mC_{i,j}]_q\,c^{-(m\mp |m|)/2}\,x^{\pm}_{j,m+l}, 
$$
\begin{equation}\label{newrealisation}
x^{\pm}_{i,m+1}x^{\pm}_{j,l}-q^{\pm\frac12 C_{i,j}}x^{\pm}_{i,l}x^{\pm}_{j,m+1}=q^{\pm\frac12 C_{i,j}}x^{\pm}_{i,m}x^{\pm}_{j,l+1}-x^{\pm}_{i,l+1}x^{\pm}_{j,m},
\end{equation}
$$
\qquad [x_{i,m}^+,x_{j,l}^-]=\delta_{i,j}\frac{c^{m}\phi^+_{i,m+l}-c^{l}\phi^-_{i,m+l}}{q^{\frac12}-q^{-\frac12}},
$$
$$
\sum_{\pi\in S_{1-C_{i,j}}}\sum_{r=0}^{1-{C}_{ij}}(-1)^r{1-C_{ij}\brack r}_qx^\pm_{i,m_{\pi(1)}}\dots x^\pm_{i,m_{\pi(r)}}x^{\pm}_{j,l}x^\pm_{i,m_{\pi(r+1)}}\dots x^\pm_{i,m_{\pi(1-C_{i,j})}}=0.
$$
Here $\phi^{\pm}_{i,m}$ are defined as coefficients in the formal power series
$$
\phi_i^{\pm}(u)=\sum_{m\geq 0}\phi^{\pm}_{i,\pm m}u^{\pm m}=k_i^{\pm 1}\exp\left(\pm(q^{\frac12}-q^{-\frac12})\sum_{m>0} h_{i,\pm m} u^{\pm m}\right)
$$
and we set $\phi^{\pm}_{i,\pm m}=0$ if $m<0$. The last relation in \eqref{newrealisation} is taken for every pair of distinct integers $(i,j)\in I^2$ and for every integer sequence $m_1, \dots, m_{1-C_{i,j}}$, while the first summation is over all permutations $\pi$ of $1-C_{i,j}$ elements.

The two presentations of $U_q'(\widehat{\sl}_{n+1})$ above are related as follows: $\{k_i^{\pm 1}\}_{i\in I}$ denote the same elements in both presentations, $x^{\pm}_{i,0}=x^{\pm}_i$ for $i\in I$ and $c=k_0k_1k_2\dots k_n$. The generators $x^{\pm}_0$ are given in terms of the generators from Drinfeld's presentation using the following expressions, \emph{cf.} \cite{Dri88}, \cite{Jin96}:
\begin{equation}
\label{e0expression}
x^+_0=[x_n^-, [x^-_{n-1},\dots [x_2^-,x_{1,1}^-]_{q^{-\frac12}}\dots]_{q^{-\frac12}}]_{q^{-\frac12}}\ c k^{-1}_1k^{-1}_2\dots k^{-1}_n,
\end{equation}
\begin{equation}
\label{f0expression}
x^-_0=c^{-1} k_1\dots k_n\ [\dots [[x_{1,-1}^+,x_{2}^+]_{q^{\frac12}}, x_3^+]_{q^{\frac12}}, \dots x_n^+]_{q^{\frac12}},
\end{equation}
where we use the notation $[A,B]_u=AB-uBA$.

The subalgebra of $U_q'(\widehat{\sl}_{n+1})$ generated by $\{k^{\pm 1}_i\}_{i\in I}$, $\{x^{\pm}_i\}_{i\in I}$ is a Hopf subalgebra isomorphic to the quantized enveloping algebra $U_q(\sl_{n+1})$, thus any $U'_q(\widehat{\sl}_{n+1})$-module has also a structure of a $U_q(\sl_{n+1})$-module. We can define a $P$-grading on $U_q'(\widehat{\sl}_{n+1})$ and $U_q({\sl}_{n+1})$ by setting
$$
\deg( x_{i,m}^{\pm})=\pm\alpha_i, \qquad \deg( x^{\pm}_0)=\mp\theta, \qquad \deg(k_{i}^{\pm 1})=\deg(h_{i,n})=\deg(c)=0.
$$

Let $U^{+}$ (resp. $U^-$ and $U^0$) be the subalgebra of $U_q(\sl_{n+1})$ generated by $\{x_i^{+}\}_{i\in I}$ (resp. $\{x_i^{-}\}_{i\in I}$ for $U^-$ and $\{k^{\pm1}_i\}_{i\in I}$ for $U^0$). Similarly, let $\widehat{U}^{+}$ (resp. $\widehat{U}^{-}$,$\widehat{U}^0$) be the subalgebra of $U'_q(\widehat{\sl}_{n+1})$ generated by $\{x_{i,m}^{+}\}_{i\in I, m\in \mathbb Z}$ (resp. $\{x_{i,m}^{-}\}_{i\in I, m\in \mathbb Z}$ and $\{k^{\pm1}_i, h_{i,m}\}_{i\in I, m\in \mathbb Z\backslash\{0\}}$). The following factorizations are known \cite{CP94a}:
$$
U_q(\sl_{n+1})=U^-.U^0.U^+, \qquad U_q'(\widehat{\sl}_{n+1})=\widehat{U}^-.\widehat{U}^0.\widehat{U}^+,
$$
where we use $U_1.U_2:=\{u_1u_2\mid u_1\in U_1, u_2\in U_2\}$.

There exists an involution of $U_q'(\widehat{\sl}_{n+1})$, which is denoted by $\widehat{\omega}$ and is defined by
$$
\widehat{\omega}(x_{i,m}^{\pm})=-x_{i,-m}^{\mp}, \qquad \widehat{\omega}(h_{i,m})=-h_{i,-m},\qquad \widehat{\omega}(k_{i}^{\pm 1})=k_i^{\mp 1}, \qquad \widehat{\omega}(c^{\pm 1})=c^{\mp 1}.
$$
One can check, using \eqref{e0expression}, \eqref{f0expression} that $\widehat\omega(x_0^\pm)=-q^{\mp\frac{n+1}{2}}x_0^{\mp}$. Moreover, $\widehat{\omega}$ is a coalgebra anti-automorphism:
$$
\Delta\circ \widehat\omega=(\widehat\omega\otimes\widehat\omega)\circ\Delta'.
$$
For a $U'_q(\widehat{\sl}_{n+1})$-module $V$ we use $V^{\omega}$ to denote the pullback of $V$ through $\omega$.

Note that we have not listed the expressions for the coproduct of the Drinfeld generators. In fact, explicit formulae for $\Delta x^\pm_{i,m}$ and $\Delta h_{i,m}$ are not known, but partial expressions are available in the case of the \emph{quantum loop algebra} $U_q(L{\sl}_{n+1})=U_q'(\widehat{\sl}_{n+1})/(c-1)$, see \cite[Proposition 1.2]{Cha01} and references therein.

\begin{rem} In the definitions of the quantum algebras we choose to use $q^{\frac 12}$ instead of $q$. Our reasoning behind this choice will be clear in Section \ref{sqW-sect}, where our $q$ will match the parameter $q$ of the $q$-Whittaker and Macdonald symmetric functions. 
\end{rem}

\subsection{Highest $l$-weight modules} To study representations of $U_q(\sl_{n+1})$ and $U_q'(\widehat{\sl}_{n+1})$ it is convenient to reintroduce the notions of weights, weight spaces and highest weight representations in the quantum affine setting. For a $U_q(\sl_{n+1})$-module $V$ we say that $v\in V$ is a weight vector if
$$
k_i.v=\rho_i v, \qquad i\in I
$$
for an $n$-tuple $\rho=(\rho_1, \rho_2, \dots, \rho_n)\in(\mathbb C^\times)^n$ called the \emph{weight} of $v$. Treating $\mathbb C^\times=\mathbb C\backslash\{0\}$ as an abelian group, we get an abelian group structure on the set of such weights: $\rho\tau=(\rho_1\tau_1, \dots, \rho_n\tau_n)$. We define an abelian group map $q^{\frac{\bullet}{2}}:Q\to(\mathbb C^\times)^n$ by setting $q^{\frac{\mu}{2}}:=(q^{\frac12 (\mu,\alpha_1)}, \dots, q^{\frac12 (\mu,\alpha_n)})$. Since we have assumed that $q^{\frac12}$ is transcendental this map is injective.

We call a $U_q(\sl_{n+1})$-module $V$ a \emph{weight module} if 
$$
V=\bigoplus_{\rho\in(\mathbb C^\times)^n}V_{\rho},\qquad V_\rho:=\{v\in V \mid k_i.v=\rho_iv\}.
$$
The spaces $V_\rho$ are called the \emph{weight spaces} of $V$, and $\rho$ is called a weight of $V$ if $V_\rho\neq 0$. We say that an $U_q(\sl_{n+1})$-module $V$ is in the category $\mathcal O$ if 
\begin{itemize}
\item $V$ is a weight module and all its weight spaces are finite-dimensional;
\item All weights of $V$ are contained in $\bigcup_{i=1}^{m}\{\rho_i q^{-\frac{\mu}{2}}\mid\mu\in P^+\}$ for some weights $\rho_1, \dots, \rho_m\in(\mathbb C^{\times})^n$.
\end{itemize}
Similarly to the classical situation, for each weight $\rho$ there exists a unique irreducible $U_q(\sl_{n+1})$-module $V(\rho)\in \mathcal O$ with the highest weight $\rho$.

To extend the formalism above to $U_q'(\widehat{\sl}_{n+1})$ we assume for convenience that $c$ always acts by $1$; one can check that this assumption is not restrictive as long as we work with simple $U_q'(\widehat{\sl}_{n+1})$-modules in $\mathcal O$, \emph{cf.} \cite[Proposition 3.2]{MY12}. For a $U'_q(\widehat{\sl}_{n+1})$-module $V$ we call $v\in V$ an \emph{$l$-weight vector} if \begin{samepage}
$$
\phi^{\pm}_{i,m}.v=\gamma^{\pm}_{i,m}v, \qquad i\in I, m\in\Z,$$
where $\gamma=(\gamma_{i,m}^\pm)_{i\in I, m\in \pm\mathbb{Z}_{\geq 0}}$ is a collection of complex numbers satisfying $\gamma_{i,0}^+\gamma_{i,0}^-=1$ for every $i\in I$.\end{samepage} Such collections of numbers $\gamma$ are called \emph{$l$-weights}\footnote{As a warning, $l$-weights are not as nicely behaved as the usual weights. For example, if $v$ is an $l$-weight vector $x^{\pm}_i.v$ might fail to be an $l$-weight vector.}. We say that an $l$-weight $\gamma$ is \emph{rational} if for some rational functions $(f_1(u), \dots, f_n(u))$ the expansions of $f_i(u)$ around $0$ and $\infty$ are
$$
f_i(u)=\sum_{m\geq 0} \gamma_{i,m}^+ u^m, \qquad f_i(u)=\sum_{m\geq 0} \gamma_{i,-m}^- u^{-m}. 
$$
In this situation we denote this $l$-weight by ${\bm f}=(f_1(u), \dots, f_n(u))$. Note that rational functions $f_i(u)$ define a rational $l$-weight as long as  $f_i(u)$ are regular at $0,\infty$ and $f_i(0)f_i(\infty)=1$. For rational $l$-weights $\bm{f}, \bm{g}$ set $\bm{fg}:=(f_1(u)g_1(u), \dots, f_n(u)g_n(u))$. 

We call an $l$-weight vector $v$ \emph{singular} if $x^+_{i,m}.v=0$ for any $i,m$.  An $U'_q(\widehat{\sl}_{n+1})$-module $V$ is called \emph{a highest $l$-weight module} when $V=U'_q(\widehat{\sl}_{n+1}).v=\widehat{U}^-.v$ for a singular $l$-weight vector $v$; the $l$-weight of $v$ is called the \emph{highest $l$-weight} of $V$. The following statement summarizes the information about irreducible highest $l$-weight modules with rational highest $l$-weights:

\begin{prop}[\cite{MY12}]\label{Lmod} Let $\bm f, \bm g$ be rational $l$-weights.
\begin{enumerate}
\item There exists a unique up to isomorphism irreducible representation of $U_q'(\widehat{\sl}_{n+1})/(c-1)$ with the highest $l$-weight ${\bm f}$, which we denote by $L(\bm f)$.
\item $L(\bm f)\in\mathcal O$ as a $U_q(\sl_{n+1})$-module.
\item If $L(\bm f)\otimes L(\bm g)$ is an irreducible $U_q'(\widehat{\sl}_{n+1})$-module then $L(\bm f)\otimes L(\bm g)\cong L(\bm{fg})\cong L(\bm g)\otimes L(\bm f)$.
\end{enumerate}
\end{prop}

\subsection{The representations $V(s)_z$} For a pair of complex numbers $s,z\in\C^\times$ let $V(s)_z$ denote the $U_q'(\widehat{\sl}_{n+1})$-module $L(\bm{f})$ with $f_1(u)=s^{-1}\frac{1-szu}{1-s^{-1}zu}$ and $f_r(u)=1$ for $r>1$. For $a,b\in\C^\times$ we also set $V^a_b:=V(a/b)_{ab}$. These representations can be explicitly described, and to do it we use the following notation. By a \emph{composition} $\bI$ we mean an $n$-tuple of nonnegative integers $(I_1, \dots, I_n)\in\mathbb Z_{\geq 0}^n$ and for any composition $\bI$ we set $|\bI|=I_1+I_2+\dots+I_n$. Define $\be^i=(e_1^i, e_2^i,\dots, e_n^i)$ as the composition with $e^i_j=\delta_{i,j}$, and set $\bm 0:=(0, \dots, 0)$.

\begin{prop}\label{explicitAction} The representation $V(s)_z$ is infinite-dimensional and has basis $\{v_{\bI}\}_{\bI\in\Z^n_{\geq 0}}$ if $s\neq\pm q^{-\frac{m}{2}}$ for any $m\in\mathbb Z_{\geq 0}$, and is finite-dimensional with basis $\{v_{\bI}\}_{\bI\in\Z^n_{\geq 0}:|\bI|\leq m}$ if $s=\pm q^{-\frac{m}{2}}$ for $m\in\Z_{\geq0}$. The action of $U'_q(\widehat{\sl}_{n+1})$ is given explicitly by
\begin{align}
k_1. v_{\bI}&=s^{-1}q^{\frac12(-|\bI|-I_1)}\ v_{\bI}, \quad &k_r.v_{\bI}&=q^{\frac12(I_{r-1}-I_r)}\ v_{\bI},\quad &k_0.v_{\bI}&= sq^{\frac12(|\bI|+I_n)}\ v_{\bI},\label{slaction} \\
x_1^+.v_{\bI}&=s^{-1}q^{\frac12(-|\bI|+1)}[I_1]_q\ v_{\bI-\be^1},\quad     &x_r^+. v_{\bI}&=[I_{r}]_q\ v_{\bI-\be^r+\be^{r-1}}, \quad &x_0^+. v_{\bI}&=z\frac{1-s^2q^{|\bI|}}{q^{\frac12}-q^{-\frac12}}\ v_{\bI+\be^{n}},\nonumber\\
x_1^-.v_{\bI}&=\frac{1-s^2q^{|\bI|}}{q^{\frac12}-q^{-\frac12}}\ v_{\bI+\be^1},\quad  &x_r^-. v_{\bI}&=[I_{r-1}]_q\ v_{\bI+\be^r-\be^{r-1}}, \quad &x^-_0.v_{\bI}&=(zs)^{-1}q^{\frac12(-|\bI|+1)}[I_n]_q\ v_{\bI-\be^n}\nonumber,
\end{align}
where $r=2, \dots, n$. Moreover, viewed as a $U_q(\sl_{n+1})$-module, $V(s)_z$ is irreducible with the highest weight $(s^{-1}, 1, \dots, 1)$.
\end{prop}
For the finite-dimensional case the claim can be deduced from \cite[Proposition 5.1]{MTZ03}, and the infinite-dimensional situation can be reached using analytic continuation. Below we provide another way to verify Proposition \ref{explicitAction}.
\begin{proof} We first check that \eqref{slaction} gives a well-defined $U'_q(\widehat{\sl}_{n+1})$-module, and then verify that the resulting module is indeed $V(s)_z$.

Let $V'$ be a vector space with a basis $\{v_{\bI}\}_{\bI\in\Z^n_{\geq 0}}$, and consider $V'$ as a $U_q'(\widehat{\sl}_{n+1})$-module with the action defined by \eqref{slaction}. To check that this action is well-defined we need to verify the defining relations \eqref{slrel}, which can be done by a direct computation. For instance,
\begin{multline*}
[x_1^+, x_1^-].v_{\bI}= \left(\frac{1-s^2q^{|\bI|}}{q^{\frac12}-q^{-\frac12}} s^{-1}q^{-\frac12|\bI|}\frac{q^{\frac{I_1+1}{2}}-q^{-\frac{I_1+1}{2}}}{q^{\frac12}-q^{-\frac12}}- s^{-1}q^{\frac12(-|\bI|+1)}\frac{q^{\frac{I_1}{2}}-q^{-\frac{I_1}{2}}}{q^{\frac12}-q^{-\frac12}}\frac{1-s^2q^{|\bI|-1}}{q^{\frac12}-q^{-\frac12}}\right)v_{\bI}\\
=\frac{s^{-1}q^{-\frac{|\bI|}{2}-\frac{I_1}{2}}-sq^{\frac{|\bI|}{2}+\frac{I_1}{2}}}{q^{\frac12}-q^{-\frac12}}v_{\bI}=\frac{k_1-k_1^{-1}}{q^{\frac12}-q^{-\frac12}}.v_{\bI},
\end{multline*}
\begin{equation*}
((x_1^+)^2x_2^+-[2]_qx_1^+x_2^+x_1^++x_2^+(x_1^+)^2).v_{\bI}=s^{-2}q^{-|\bI|+\frac32}[I_2]_q[I_1]_q \left([I_1+1]_q-[2]_q[I_1]_q+[I_1-1]_q\right)v_{\bI}=0.
\end{equation*}
Other relations  (including the Serre relations for $U'_q(\widehat{\sl}_2)$, when $\widehat{C}_{0,1}=-2$) are either trivial or similar to the two above, so we omit their verification.

Note that when $s=\pm q^{-\frac{m}{2}}$ for $m\in\mathbb Z_{\geq 0}$  the module $V'$ defined above has a submodule spanned by $\{v_{\bI}\}_{\bI\in\Z^n_{\geq 0}:|\bI|\leq m}$. Indeed, the only generators $x_i^{\pm}$ which send $v_{\bI}$ to $v_{\bJ}$ with $|\bJ|>|\bI|$ are $x_0^+$ and $x_1^-$, and we have $x_0^+v_{\bI}=0, x_1^-v_{\bI}=0$ when $s=\pm q^{-\frac{m}{2}}$ and $|\bI|=m$. In this situation, let $V''\subset V'$ denote this submodule spanned by $\{v_{\bI}\}_{\bI\in\Z^n_{\geq 0}:|\bI|\leq m}$, otherwise, when $s\neq \pm q^{-\frac{m}{2}}$ for any $m\in\mathbb Z_{\geq 0}$, set $V''=V'$. Looking at the action of $x_i^{\pm}$ for $i=1, \dots, n$, one can see that $V''=U_q(\sl_{n+1}).v_{\bm 0}$ and $v_{\bm0}$ is the only $U_q(\sl_{n+1})$-singular vector of $V''$, hence $V''$ is an irreducible $U_q(\sl_{n+1})$-module with the highest weight $(s^{-1},1,\dots,1)$.

To finish the proof we only need to show that $V''\cong V(s)_z$. To do it we refer to \cite[Proposition 5.5]{MY12}, which claims that a $U_q'(\widehat{\sl}_{n+1})$-module which is irreducible as a $U_q(\sl_{n+1})$-module and has the highest $U_q(\sl_{n+1})$-weight $(s^{-1}, 1, \dots, 1)$ is isomorphic to $L(\bm g)$ with $g_1(u)=s^{-1}\frac{1-sz'u}{1-s^{-1}z'u}$, $g_2(u)=g_3(u)=\dots=1$ for some $z'\in\mathbb C\backslash\{0\}$\footnote{Such a $U_q'(\widehat{\sl}_{n+1})$-module is called a Kirillov-Reshetikhin module corresponding to the first fundamental weight.} . So $V'' \cong L(\bm g)$, and we just need to show that $z=z'$. To do it we can compute the action of $x^-_{1,1}$ on $v_{\bm 0}$: Note that for any $A\in U_q'(\widehat{\sl}_{n+1})$ such that $[A,x_i^+]=0$, $k_iAk_i^{-1}=q^{\frac12}A$ we have $[x_i^+, [x_i^-, A]_{q^{-\frac12}}k_i^{-1}k_{i-1}^{-1}]_{q^{\frac12}}=Ak_{i-1}^{-1}$. Using this relation as an inductive step, we see from \eqref{e0expression} that
$$
[x_{i}^+,[x_{i+1}^+,\dots[x_n^+,x_0^+]_{q^{\frac12}}\dots]_{q^{\frac12}}]_{q^{\frac12}}=[x_{i-1}^-,\dots[x_2^-,x_{1,1}^-]_{q^{-\frac12}}\dots]_{q^{-\frac12}} ck_1^{-1}\dots k_{i-1}^{-1}.
$$
In particular, 
$$
x_{1,1}^-.v_{\bm 0}=[x_{2}^+,[x_{3}^+,\dots[x_n^+,x_0^+]_{q^{\frac12}}\dots]_{q^{\frac12}}]_{q^{\frac12}}k_1.v_{\bm 0}=s^{-1}x_2^+\dots x_n^+x_0^+.v_{\bm 0}=z\frac{s^{-1}-s}{q^{\frac12}-q^{-\frac12}}v_{\be^1},
$$
so
$\phi^+_{1,1}.v_{\bm0}=(q^{\frac12}-q^{-\frac12})[x_1^+, x_{1,1}^-].v_{\bm 0}=z(s^{-2}-1).v_{\bm 0}$. On the other hand, since $g_1(u)=s^{-1}+z'(s^{-2}-1)u+\overline o(u)$, we have $\phi^+_{1,1}.v_{\bm0}=z'(s^{-2}-1).v_{\bm 0}$. Hence $z=z'$ and $V''\cong V(s)_z$.
\end{proof}

\begin{rem} The representation $V(s)_z$ can also be viewed as an evaluation module induced from an analytically-continued symmetric tensor power of the standard representation of $U_q(\sl_{n+1})$. More rigorously, to obtain $V(s)_z$ one can consider the extension $U_q(\sl_{n+1})\subset U_q(\mathfrak{gl}_{n+1})$, define a $U_q(\mathfrak{gl}_{n+1})$-module $V(s)$ and then pull it along the evaluation map $\mathrm {ev_z}:U'_q(\widehat{\sl}_{n+1})\to U_q(\mathfrak{gl}_{n+1})$ introduced in \cite{Jim86}.
\end{rem}

The following fact will be crucial in the following section.

\begin{prop}
\label{irred}
The $U_q'(\widehat{\sl}_{n+1})$-module $V(s_1)_{z_1}\otimes V(s_2)_{z_2}\otimes\dots\otimes V(s_L)_{z_L}$ is irreducible for a fixed $L\in\Z_{\geq 1}$ and generic complex parameters $z_1, \dots, z_L, s_1, \dots, s_L$.
\end{prop}
\begin{proof}
We refer to a much stronger result \cite[Theorem 4]{Cha01} for the case when all $V(s_i)_{z_i}$ are finite-dimensional, that is, when $s_i\in q^{-\frac{1}{2}\mathbb Z_{\geq 0}}$ for each $i$. 
\begin{theo}[{\cite[Theorem 4]{Cha01}}] \label{irredtheoref}Let $m_1, \dots, m_L\in\mathbb Z_{\geq 0}, z_1,\dots, z_L\in\C\backslash\{0\}$. If for any pair $i,j$ such that $i\neq j$ we have $z_i/z_j\not\in q^{\frac12 \mathbb Z}$, then $V(q^{-\frac{m_1}{2}})_{z_1}\otimes V(q^{-\frac{m_2}{2}})_{z_2}\otimes\dots\otimes V(q^{-\frac{m_L}{2}})_{z_L}$ is a highest $l$-weight module. 
\end{theo}
Below we explain how \cite[Theorem 4]{Cha01} implies Proposition \ref{irred} by first explaining the notation from \cite{Cha01}, then showing irreducibility of  $V(q^{-\frac{m_1}{2}})_{z_1}\otimes\dots\otimes V(q^{-\frac{m_L}{2}})_{z_L}$ for generic $z_1, \dots, z_L$ and then lifting the restriction $s_i\in q^{-\frac12{\mathbb Z_{\geq 0}}}$.

We start with matching the notation used in \cite{Cha01} with our notation: 
\begin{itemize}
\item Our $q^{\frac12}$ corresponds to $q$ in \cite{Cha01}, to avoid confusion we keep using our $q^{\frac12}$ even when describing notation from \cite{Cha01}. 
\item In \cite{Cha01} the irreducible finite-dimensional highest $l$-weight $U'_q(\widehat{\sl}_{n+1})$-modules are parametrized by $n$-tuples $\bm \pi=(\pi_1(u), \dots, \pi_n(u))$ of polynomials $\pi_i(u)$, instead of rational $l$-weights $\bm f$ like in our work. More precisely, the functions $\pi_i(u)$ corresponding to the representation $V({\bm f})$ are given by
$$
\pi_i(u)=\begin{cases}
\prod_{j\geq 1} \frac{f_i(q^{j}u)}{f_i(0)}, \quad |q|<1,\\
\prod_{j\geq 0} \frac{f_i(0)}{f_i(q^{-j}u)},\quad |q|>1,
\end{cases}
$$
and it turns out that for the finite-dimensional representation $V(q^{-\frac{m}{2}})_z$ these functions $\pi_i$ are polynomials (regardless of conditions $|q|>1$ or $|q|<1$):
$$
\pi_1(u)=\prod_{i=1}^m(1-q^{\frac{m}{2}-i+1}zu), \qquad \pi_r(u)=1, \quad r\geq 2.
$$
The $n$-tuple of polynomials $(\pi_1(u), \dots, \pi_n(u))$ above is denoted by $\bm \pi^1_{m,zq^{1/2}}$ in \cite{Cha01}, and our representations $V(q^{-\frac{m}{2}})_z$ correspond to $V(\bm \pi^1_{m,zq^{1/2}})$ in \cite{Cha01}.

\item In \cite{Cha01} the results are proved over $\mathbb C(q^{\frac12})$, while here we fix $q^{\frac 12}$ as a transcendental complex number. Since we work generically and the action of $U_q'(\widehat{\sl}_{n+1})$ can be expressed in terms of matrices over $\mathbb Q(q^{\frac12}, z_1, \dots, z_m)$ this difference is irrelevant in view of Lemma \ref{linear} below, we will return to this later when discussing the transition from $s_i=q^{-\frac{m_i}{2}}$ to arbitrary $s_i$.
\end{itemize}

So, from \cite[Theorem 4]{Cha01} we know that for any fixed $m_1, \dots, m_L\in\mathbb Z_{\geq 1}$ and generic $z_1, \dots, z_L$ the $U'_q(\widehat{\sl}_{n+1})$-module $V:=V(q^{-\frac{m_1}{2}})_{z_1}\otimes\dots\otimes V(q^{-\frac{m_L}{2}})_{z_L}$ is highest $l$-weight. Since the vector $v:=v_{\bm 0}\otimes\dots v_{\bm 0}$  is the unique up to a scalar vector with the maximal $U_q(\sl_{n+1})$-weight $\rho=(q^{\frac 12\sum_{i}m_i},1\dots, 1)$, we have $x^+_{i,m}.v=0$ and $v$ is the highest $l$-weight vector. In other words,  \cite[Theorem 4]{Cha01} can be rephrased as $U'_q(\widehat{\sl}_{n+1}). v=V$. 

To show irreducibility of $V$ we consider $(V^*)^{\widehat\omega}$, that is, we consider the dual representation pulled back along the involution $\widehat{\omega}$. For convenience, we identify the underlying vector space of $(V^*)^{\widehat\omega}$ with $V^*$. We use the following fact from \cite{CP94b}:
\begin{lem}[{\cite[Proposition 5.1]{CP94b}}]
There exists a fixed constant $c$ such that for any $m,z$ we have $((V(q^{-\frac{m}{2}})_z)^*)^{\widehat\omega}\cong V(q^{-\frac{m}{2}})_{cz^{-1}}$.
\end{lem}
Hence we have 
$$
(V^*)^{\widehat\omega}=V(q^{-\frac{m_1}{2}})_{cz_1^{-1}}\otimes\dots\otimes V(q^{-\frac{m_L}{2}})_{cz_L^{-1}}
$$
and so $(V^*)^{\widehat\omega}$ is also a highest $l$-weight representation for generic $z_1, \dots, z_L$. By comparing $U_q(\sl_{n+1})$-weights note that the highest $l$-vector $v^*$ of $(V^*)^{\widehat\omega}$ is dual to $v$ in the sense that $v^*$ vanishes on all weight spaces of $V$ other than $\mathbb C v$. Now assume that $W\subset V$ is a $U_q'(\widehat{\sl}_{n+1})$-submodule. Then $W$ should also have a decomposition into $U_q(\sl_{n+1})$-weight spaces, and hence either $v\in W$ or $v^*\in W^{\perp}$, where $W^{\perp}$ is the annihilator of $W$ in $V^*$. In the first case we have $U'_q(\widehat{\sl}_{n+1}).v=V\subset W$, while in the second case we get
$$
U_q'(\widehat{\sl}_{n+1}).v^*=\widehat\omega\(U_q'(\widehat{\sl}_{n+1})\).v^*=V^*\subset W^{\perp},
$$
implying $W=0$. So $V(q^{-\frac{m_1}{2}})_{z_1}\otimes\dots\otimes V(q^{-\frac{m_L}{2}})_{z_L}$ is irreducible for generic $z_1, \dots, z_L$.

To establish Proposition \ref{irred} we now need to go back from $s_i=q^{-\frac{m_i}{2}}$ to generic $s_i$, which can be done using the following elementary linear algebra fact:

 \begin{lem}\label{linear} Fix $d_1,d_2\in\mathbb Z_{\geq 1}$ and let $X_{i,j}^{(r)}$ be formal variables enumerated by $i,j,r\in\mathbb Z_{\geq 1}$ such that $i\leq d_1, j\leq d_2$. There exist  countable families of polynomials $F_p$ and $G_t$ over $\mathbb Z$ in variables $X_{i,j}^{(r)}$ with the following property:
 
Assume we are given finite-dimensional vector spaces $V,W$ over $\mathcal F$ with dimensions $\dim(W)=d_1, \dim(V)=d_2$, and a countable family of linear operators $A^{(r)}:W\to V$. Fix bases of $V,W$ and let $A^{(r)}_{i,j}$ be the matrix coefficients of $A^{(r)}$ with respect to those bases. Then $\bigcup_{r} \mathrm{Im}\, A^{(r)}=V$ if and only if $F_p(A^{(r)}_{i,j})\neq 0$ for some $p$, and $\bigcap_{r} \mathrm{Ker}\, A^{(r)}=0$ if and only if $G_t(A^{(r)}_{i,j})\neq 0$ for some $t$. Here $F_p(A^{(r)}_{i,j}), G_t(A^{(r)}_{i,j})$ denote the results of substitution $X^{(r)}_{i,j}=A^{(r)}_{i,j}$ into $F_p$ and $G_t$.\qed
\end{lem} 
%take det of the matrix obtained by horizontal concatenation for F, vertical concatination for G.

We apply Lemma \ref{linear} in the following way. Let, as before, $V=V(s_1)_{z_1}\otimes\dots\otimes V(s_L)_{z_L}$, $v=v_{\bm 0}\otimes\dots\otimes v_{\bm 0}$ and $\rho=(\prod_{i}s_i^{-1},1\dots, 1)$. Note that, when viewed as $U_q(\sl_{n+1})$-module, all weights of $V$ are of the form $\rho q^{-\frac12\mu}$ for $\mu\in P^+$ and the weight spaces can be explicitly described: if $\mu=K_1\alpha_1+\dots+K_n\alpha_n$ for a composition $\bK\in \Z_{\geq 0}^n$ then $V_{\rho q^{-\frac12\mu}}$ is spanned by $v_{\bI_1}\otimes\dots \otimes v_{\bI_L}$ such that $\bI_1+\dots+\bI_L=\bK$ and $v_{\bI_i}\in V(s_i)_{z_i}$. The last condition is only relevant when $s_i=\pm q^{-\frac{m}{2}},m\in\Z$, when we require $|\bI_i|\leq m$. In particular, as long as $s_i\notin\{\pm1, \pm q^{-\frac{1}{2}}, \dots, \pm q^{-\frac{|\bK|-1}{2}}\}$ for each $i$, we can identify the vector space $V_{\rho q^{-\frac12\mu}}$ for arbitrary $s_1, \dots, s_L, z_1, \dots, z_L$ with the vector space over $\C$ generated by basis vectors $\{|\bI_1, \dots, \bI_L\rangle\}_{\bI_1+\dots+\bI_L=\bK}$, we denote the latter by $V^{gen}_{\rho q^{-\frac12\mu}}$.

Note that $V$ is irreducible if and only if for any $\mu\in P^+$ we have $V_{\rho q^{-\frac12\mu}}\subset U'_q(\widehat{\sl}_{n+1}).v$ and $v\in U'_q(\widehat{\sl}_{n+1}).v'$ for any $v'\in V_{\rho q^{-\frac12\mu}}\backslash\{0\}$. It is enough to show that for a fixed $\mu\in P^+$ both these conditions hold for generic $s_1,\dots, s_L, z_1, \dots, z_L$. We start with $V_{\rho q^{-\frac12\mu}}\subset U'_q(\widehat{\sl}_{n+1}).v$. Recall that we have a $P$-grading on $U_q'(\widehat{\sl}_{n+1})$, let $\mathcal G_{\mu}$ be the set of words in $k^{\pm 1}_i, x_i^{\pm}, i\in \widehat{I}$ whose total $P$-degree is $\mu$. Using $U'_q(\widehat{\sl}_{n+1})$-action from Proposition \ref{explicitAction}, consider each word $w\in \mathcal G_{-\mu}$ as an operator $A^{(w)}:\mathbb C v\to V_{\rho q^{-\frac12\mu}}$. Recall that for a fixed $K\in\mathbb Z_{\geq0}$ depending only on $\mu$ the vector spaces $V_{\rho q^{-\frac12\mu}}$ for arbitrary $s_1, \dots, s_L, z_1, \dots, z_L$ can be identified with $V_{\rho q^{-\frac12\mu}}^{gen}$, as long as $s_i\notin\{\pm1, \pm q^{-\frac{1}{2}}, \dots, \pm q^{-\frac{K-1}{2}}\}$. Moreover, all matrix coefficients of $A^{(w)}$ with respect to $v_{\bI_1}\otimes\dots \otimes v_{\bI_l}$ are polynomials in $s_1, \dots, s_L, z_1,\dots, z_L$, with coefficients in $\mathbb Q(q^{\frac12})$. Since $V_{\rho q^{-\frac12\mu}}\subset U'_q(\widehat{\sl}_{n+1}).v$ is equivalent to $\bigcup_{w}\mathrm{Im}\, A^{(w)}=V_{\rho q^{-\frac12\mu}}$, Lemma \ref{linear} gives a family of polynomials $F_p(s_1, \dots, s_L, z_1, \dots, z_L)$ with coefficients from $\mathbb Q(q^{\frac12})$, such that $V_{\rho q^{-\frac12\mu}}\subset U'_q(\widehat{\sl}_{n+1}).v$ if and only if $F_p(s_1, \dots, s_L, z_1, \dots, z_L)\neq 0$ for some $p$, still assuming $s_i\notin\{\pm1, \pm q^{-\frac{1}{2}}, \dots, \pm q^{-\frac{K-1}{2}}\}$. Since we work generically over $s_1, \dots, s_L, z_1, \dots, z_L$, it is enough to show that $F_p\neq 0$ as a polynomial in $s_i, z_i$ for at least one $p$.  But from the finite-dimensional case we know that $V_{\rho q^{-\frac12\mu}}\subset U'_q(\widehat{\sl}_{n+1}).v$ for generic $z_1, \dots, z_L$ and $s_i\in q^{-\frac{1}{2}\Z_{\geq 0}}$.   Hence, for any $m_i\geq K$ and generic $z_i$, we have $F_p(q^{-\frac{m_1}{2}}, \dots, q^{-\frac{m_L}{2}}, z_1, \dots, z_L)\neq 0$ for some $p$, implying that $F_p\neq 0$ as a polynomial in $s_i, z_i$ for some $p$. Hence $V_{\rho q^{-\frac12\mu}}\subset U'_q(\widehat{\sl}_{n+1}).v$ for generic $s_i, z_i$. Note in particular, that we have used Lemma \ref{linear} to rewrite the condition $V_{\rho q^{-\frac12\mu}}\subset U'_q(\widehat{\sl}_{n+1}).v$ in terms of vanishing of certain polynomials with coefficients being rational functions in $q^{\frac12}$, this readily implies that we can equivalently consider $q^{\frac12}$ as a transcendental complex number or as a formal variable, making this difference with \cite{Cha01} irrelevant. 

To show that generically $v\in U'_q(\widehat{\sl}_{n+1}).v'$ for fixed $\mu$ and any $v'\in V_{\rho q^{-\frac12\mu}}\backslash\{0\}$ we use the other half of Lemma \ref{linear}. Namely, considering now the words $w\in \mathcal G_{\mu}$ as operators $A^{(w)}:V_{\rho q^{-\frac12\mu}}\to\mathbb C v$, we have $v\in U'_q(\widehat{\sl}_{n+1}).v'$ for any $v'\in V_{\rho q^{-\frac12\mu}}\backslash\{0\}$ if and only if $\bigcap_w \mathrm{Ker} A^{(w)}=0$. By Lemma \ref{linear}, there exist polynomials $G_t(s_1, \dots, s_L, z_1, \dots, z_L)$ such that the last condition is equivalent to existence of $t$ such that $G_t(s_1, \dots, s_L, z_1, \dots, z_L)\neq 0$. Irreducibility of $V$ in finite-dimensional situation implies that at least one polynomial $G_t$ is nonzero, hence we have $v\in U'_q(\widehat{\sl}_{n+1}).v'$ for any $v'\in V_{\rho q^{-\frac12\mu}}\backslash\{0\}$ generically.
\end{proof}

 %%%%%%%%%%%%%%%%%

\section{Explicit expressions for isomorphisms between tensor products} 
\label{sect-intertw}

In this section we present explicit expressions for isomorphisms between representations of the form $V(s_1)_{z_1}\otimes V(s_2)_{z_2}$. The importance of these expressions is two-fold: on one hand we get explicit expressions for the $R$-matrix $V(s_1)_{z_1}\otimes V(s_2)_{z_2}\to V(s_2)_{z_2}\otimes V(s_1)_{z_1}$, reproducing a result of \cite{BM16}. On the other hand, using the generic irreducibility of arbitrary tensor products of the form $V(s_1)_{z_1}\otimes\dots\otimes V(z_L)_{z_L}$, we can use our explicit expressions to explain and generalize the inhomogeneous Yang-Baxter equations found in \cite{BK21} and used to study the $q$-Hahn vertex model.

To simplify expressions in this section we use the following notation. Recall that for a pair of parameters $a,b\in\C\backslash\{0\}$ we set $V^a_b:=V(a/b)_{ab}$. For compositions $\bX, \bY$ and complex parameters $a,b$ define
$$
\Phi(\bX,\bY; a,b):=b^{|\bX|}q^{\sum_{i<j} X_iY_j}\frac{(a;q)_{|\bX|} (b;q)_{|\bY|}}{(ab;q)_{|\bX+\bY|}}\prod_{i=1}^n\frac{(q;q)_{X_i+Y_i}}{(q;q)_{X_i}(q;q)_{Y_i}}.
$$
Note that, while we usually assume that the compositions are positive, the expression $\Phi(\bX,\bY; a,b)$ makes sense as long as for each $i$ at least one of $X_i,Y_i$ is positive, while the other might be negative, see Section \ref{notation}. Alternative expressions there imply that, if $X_i+Y_i\geq0$ for every $i$, then $\Phi(\bX,\bY; a,b)=0$ unless $X_i\geq 0$ and $Y_i\geq 0$ for every $i$. For parameters $a_1, a_2, a_3, b_1,b_2, b_3$ and compositions $\bI,\bJ$ we set
$$
C_{a_1,b_1;a_2,b_2}(\bI,\bJ):=(b_1/a_1)^{|\bJ|} q^{-\frac12\sum_{i<j}J_iI_j}.
$$
In view of Proposition \ref{explicitAction} we have a basis of $V(s_1)_{z_1}\otimes\dots\otimes V(s_L)_{z_L}$ given by $|v_{\bI_1}\otimes v_{\bI_2}\otimes\dots\otimes v_{\bI_L}\rangle$, let $\langle v_{\bI_1}\otimes v_{\bI_2}\otimes\dots\otimes v_{\bI_L}|$ denote the dual vectors in $(V(s_1)_{z_1}\otimes\dots\otimes V(s_L)_{z_L})^*$. 

\subsection{Explicit expressions} We start by describing isomorphisms between $V^{a_1}_{b_1}\otimes V^{a_2}_{b_2}$, $V^{a_2}_{b_1}\otimes V^{a_1}_{b_2}$, $V^{a_1}_{b_2}\otimes V^{a_2}_{b_1}$ and $V^{a_2}_{b_2}\otimes V^{a_1}_{b_1}$ for generic $a_1,a_2,b_1,b_2$. Note that all four representations are irreducible by Proposition \ref{irred}, and by Proposition \ref{Lmod} all of them are isomorphic to $L(\bm f)$ where $f_1(u)=\frac{(a_1^{-1}-a_1u)(a_2^{-1}-a_2u)}{(b_1^{-1}-b_1u)(b_2^{-1}-b_2u)}$ and $f_r(u)=1$ for $r\geq 2$.

\begin{theo}\label{basicintertwiner} For generic parameters $a_1,b_1, a_2, b_2$ the representations $V^{a_1}_{b_1}\otimes V^{a_2}_{b_2}$ and $V^{a_2}_{b_1}\otimes V^{a_1}_{b_2}$ are irreducible and isomorphic, with an isomorphism $W^a$ given explicitly by
$$
W^a:V^{a_1}_{b_1}\otimes V^{a_2}_{b_2}\to V^{a_2}_{b_1}\otimes V^{a_1}_{b_2},
$$
$$
\langle v_{\bK}\otimes v_{\bL}| W^a|v_{\bI}\otimes v_{\bJ}\rangle = \delta_{\bI+\bJ=\bK+\bL}\delta_{\bI\geq \bK}\frac{C_{a_2,b_1;a_1,b_2}(\bK,\bL)}{C_{a_1,b_1;a_2,b_2}(\bI,\bJ)} \Phi(\bI-\bK, \bK; a_1^2/a_2^2, a_2^2/b_1^2).
$$
Similarly, for generic parameters $a_1,b_1, a_2, b_2$ the representations $V^{a_1}_{b_1}\otimes V^{a_2}_{b_2}$ and $V^{a_1}_{b_2}\otimes V^{a_2}_{b_1}$ are irreducible and isomorphic, with an isomorphism $W^b$ given explicitly by
$$
W^b:V^{a_1}_{b_1}\otimes V^{a_2}_{b_2}\to V^{a_1}_{b_2}\otimes V^{a_2}_{b_1}
$$
$$
\langle v_{\bK}\otimes v_{\bL}| W^b|v_{\bI}\otimes v_{\bJ}\rangle = \delta_{\bI+\bJ=\bK+\bL}\delta_{\bJ\geq\bL}\frac{C_{a_1,b_2;a_2,b_1}(\bK,\bL)}{C_{a_1,b_1;a_2,b_2}(\bI,\bJ)} \Phi(\bL, \bJ-\bL; a_2^2/b_1^2, b_1^2/b_2^2).
$$
\end{theo}
\begin{proof} By the discussion above it is enough to check that the morphisms $W^a$ and $W^b$ presented above commute with the action of $U_q'(\widehat{\sl}_{n+1})$. We do it by verifying the statement on the generators $\{k_i^{\pm 1}\}$, $\{x_i^{\pm}\}$ of $U'_q(\widehat{\sl}_{n+1})$. For $k_i$ the check is trivial since $W^a$ and $W^b$ clearly preserve the weight spaces. The checks for $x_i^\pm$ are similar to each other and can be readily done by a straightforward but tedious computation. Here we only provide the verification that $x_1^+$ commutes with the action of $W^a$.

We need to check that $\langle v_{\bK}\otimes v_{\bL}| W^ax_1^+|v_{\bI}\otimes v_{\bJ}\rangle=\langle v_{\bK}\otimes v_{\bL}|x_1^+W^a|v_{\bI}\otimes v_{\bJ}\rangle$ for any compositions $\bI,\bJ,\bK,\bL$. Since $W^a$ preserves the weight and $x_1^+$ increases it by $q^{\frac{1}{2}\alpha_1}$, both sides are zero unless $\bI+\bJ-\be^1=\bK+\bL$. Assuming $\bI+\bJ-\be^1=\bK+\bL$ from now on, note that
$$
x_1^+|v_{\bI}\otimes v_{\bJ}\rangle=\frac{b_1b_2}{a_1a_2}q^{\frac12(-|\bI|-|\bJ|-J_1+1)}[I_1]_q|v_{\bI-\be^1}\otimes v_{\bJ}\rangle+\frac{b_2}{a_2}q^{\frac12(-|\bJ|+1)}[J_1]_q|v_{\bI}\otimes v_{\bJ-\be^1}\rangle,
$$
$$
\langle v_{\bK}\otimes v_{\bL}|x_1^+=\frac{b_1b_2}{a_1a_2}q^{\frac12(-|\bK|-|\bL|-L_1)}[K_1+1]_q \langle v_{\bK+\be^1}\otimes v_{\bL}| + \frac{b_2}{a_1}q^{-\frac12|\bL|}[L_1+1]_q \langle v_{\bK}\otimes v_{\bL+\be^1}|,
$$
where the first line describes the left action on $V^{a_1}_{b_1}\otimes V^{a_2}_{b_2}$ with the assumption that $|v_{\bI'}\otimes v_{\bJ'}\rangle=0$ when the configurations $\bI', \bJ'$ are non-positive, while the second line describes the right action on  $(V^{a_2}_{b_1}\otimes V^{a_1}_{b_2})^*$. So we need to verify that
\begin{multline*}
\frac{b_1b_2}{a_1a_2}q^{\frac12(-|\bI|-|\bJ|-J_1+1)}[I_1]_q\langle v_{\bK}\otimes v_{\bL}|W^a|v_{\bI-\be^1}\otimes v_{\bJ}\rangle+\frac{b_2}{a_2}q^{\frac12(-|\bJ|+1)}[J_1]_q\langle v_{\bK}\otimes v_{\bL}|W^a|v_{\bI}\otimes v_{\bJ-\be^1}\rangle\\
=\frac{b_1b_2}{a_1a_2}q^{\frac12(-|\bK|-|\bL|-L_1)}[K_1+1]_q \langle v_{\bK+\be^1}\otimes v_{\bL}|W^a|v_{\bI}\otimes v_{\bJ}\rangle + \frac{b_2}{a_1}q^{-\frac12|\bL|}[L_1+1]_q \langle v_{\bK}\otimes v_{\bL+\be^1}|W^a|v_{\bI}\otimes v_{\bJ}\rangle.
\end{multline*}
Plugging the expression for $W^a$, cancelling terms, using that $\bI+\bJ-\be^1=\bK+\bL$ and
$$
\frac{C_{a_1,b_1;a_2,b_2}(\bI-\be^1,\bJ)}{C_{a_1,b_1;a_2,b_2}(\bI,\bJ)}=1, \qquad \frac{C_{a_1,b_1;a_2,b_2}(\bI,\bJ-\be^1)}{C_{a_1,b_1;a_2,b_2}(\bI,\bJ)}=(a_1/b_1) q^{\frac12(|\bI|-I_1)},
$$
we can rewrite the needed identity as
\begin{multline*}
(1-q^{I_1})\Phi(\bI-\bK-\be^1, \bK; a_1^2/a_2^2, a_2^2/b_1^2)+q^{I_1}(1-q^{J_1})\Phi(\bI-\bK, \bK; a_1^2/a_2^2, a_2^2/b_1^2)\\
=(1-q^{K_1+1})\Phi(\bI-\bK-\be^1, \bK+\be^1; a_1^2/a_2^2, a_2^2/b_1^2) + q^{K_1}(1-q^{L_1+1}) \Phi(\bI-\bK, \bK; a_1^2/a_2^2, a_2^2/b_1^2).
\end{multline*}
Using 
$$
\frac{\Phi(\bX,\bY; a,b)}{\Phi(\bX-\be^1,\bY; a,b)}=bq^{|\bY|-Y_1}\frac{1-aq^{|\bX|-1}}{1-abq^{|\bX|+|\bY|-1}}\frac{1-q^{X_1+Y_1}}{1-q^{X_1}},
$$
$$
\frac{\Phi(\bX,\bY+\be^1; a,b)}{\Phi(\bX,\bY; a,b)}=\frac{1-bq^{|\bY|}}{1-abq^{|\bX|+|\bY|}}\frac{1-q^{X_1+Y_1+1}}{1-q^{Y_1+1}}
$$
we finally reduce the verification to
\begin{multline*}
\frac{b_1^2}{a_2^2}(1-q^{I_1-K_1})q^{-|\bK|+K_1}\frac{1-q^{|\bI|-1}a_1^2/b_1^2}{1-q^{|\bI|-|\bK|-1}a_1^2/a_2^2} +q^{I_1}(1-q^{J_1})\\
=\frac{b_1^2}{a_2^2}(1-q^{I_1-K_1})q^{-|\bK|+K_1}\frac{1-q^{|\bK|}a_2^2/b_1^2}{1-q^{|\bI|-|\bK|-1}a_1^2/a_2^2}  + q^{K_1}(1-q^{L_1+1}).
\end{multline*}
The latter can be verified by standard algebraic manipulations, using $\bI+\bJ-\be^1=\bK+\bL$. Finally we note that the degenerate cases when $I_1=0$ or $I_1=K_1$ are automatically handled by our convention on $q$-Pochhammer symbols from Section \ref{notation}. 
\end{proof}

\begin{prop} \label{Rexpression} For generic parameters $a_1,b_1, a_2, b_2$ there is an isomorphism 
$$R:V^{a_1}_{b_1}\otimes V^{a_2}_{b_2}\to V^{a_2}_{b_2}\otimes V^{a_1}_{b_1},$$
which is given in the following two equivalent ways:
\begin{equation}
\label{R1}
\langle v_{\bK}\otimes v_{\bL}| R|v_{\bI}\otimes v_{\bJ}\rangle=\delta_{\bI+\bJ=\bK+\bL}\frac{C_{a_2,b_2;a_1,b_1}(\bK,\bL)}{C_{a_1,b_1;a_2,b_2}(\bI,\bJ)}\sum_{\bP}\Phi(\bL-\bP, \bK; a_1^2/a_2^2, a_2^2/b_2^2)\Phi(\bP, \bJ-\bP; a_2^2/b_1^2, b_1^2/b_2^2),
\end{equation}
where the sum is over configurations $\bP$ such that $P_i\leq \min (J_i,L_i)$;
\begin{equation}
\label{R2}
\langle v_{\bK}\otimes v_{\bL}| R|v_{\bI}\otimes v_{\bJ}\rangle=\delta_{\bI+\bJ=\bK+\bL} \frac{C_{a_2,b_2;a_1,b_1}(\bK,\bL)}{C_{a_1,b_1;a_2,b_2}(\bI,\bJ)}\sum_{\bP}\Phi(\bL, \bK-\bP; a_1^2/b_1^2, b_1^2/b_2^2)    \Phi(\bI-\bP, \bP; a_1^2/a_2^2, a_2^2/b_1^2),
\end{equation}
where the sum is over configurations $\bP$ such that $P_i\leq \min(I_i, K_i)$.
\end{prop}
\begin{proof}
For generic $a_1,a_2,b_1,b_2$ the representations $V^{a_1}_{b_1}\otimes V^{a_2}_{b_2}$ and $V^{a_2}_{b_2}\otimes V^{a_1}_{b_1}$ are irreducible and have one-dimensional highest weight space, so there exists, up to a scalar, at most one isomorphism $R$ between them.

The claim now follows from Proposition \ref{basicintertwiner}, which allows to construct isomorphism $R$ as above in two ways corresponding to \eqref{R1} and \eqref{R2}:
$$
W^a\circ W^b: V^{a_1}_{b_1}\otimes V^{a_2}_{b_2}\to V^{a_1}_{b_2}\otimes V^{a_2}_{b_1}\to V^{a_2}_{b_2}\otimes V^{a_1}_{b_1}
$$
\begin{multline*}
\langle v_{\bK}\otimes v_{\bL}| W^a\circ W^b|v_{\bI}\otimes v_{\bJ}\rangle =   \delta_{\bI+\bJ=\bK+\bL} \frac{C_{a_2,b_2;a_1,b_1}(\bK,\bL)}{C_{a_1,b_1;a_2,b_2}(\bI,\bJ)}\\
\times\sum_{\substack{\bQ,\bP\\ \bQ+\bP=\bI+\bJ}}\delta_{\bQ\geq \bK}\delta_{\bJ\geq\bP}\Phi(\bQ-\bK, \bK; a_1^2/a_2^2, a_2^2/b_2^2)\Phi(\bP, \bJ-\bP; a_2^2/b_1^2, b_1^2/b_2^2);
\end{multline*}
$$
W^b\circ W^a: V^{a_1}_{b_1}\otimes V^{a_2}_{b_2}\to V^{a_2}_{b_1}\otimes V^{a_1}_{b_2}\to V^{a_2}_{b_2}\otimes V^{a_1}_{b_1}
$$
\begin{multline*}
\langle v_{\bK}\otimes v_{\bL}| W^b\circ W^a|v_{\bI}\otimes v_{\bJ}\rangle =   \delta_{\bI+\bJ=\bK+\bL} \frac{C_{a_2,b_2;a_1,b_1}(\bK,\bL)}{C_{a_1,b_1;a_2,b_2}(\bI,\bJ)}\\
\times\sum_{\substack{\bP,\bR\\ \bP+\bR=\bI+\bJ}}\delta_{\bR\geq\bL}\delta_{\bI\geq\bP}\Phi(\bL, \bR-\bL; a_1^2/b_1^2, b_1^2/b_2^2)    \Phi(\bI-\bP, \bP; a_1^2/a_2^2, a_2^2/b_1^2).
\end{multline*}
Note that both morphisms above send $v_{\bm 0}\otimes v_{\bm 0}$ to $v_{\bm 0}\otimes v_{\bm 0}$, so these two isomorphisms coincide.
\end{proof}
\begin{rem}
Proposition \ref{Rexpression} provides an expression for the action of the $R$-matrix of $U'_q(\widehat{\sl}_{n+1})$ on $V(s)_z\otimes V(s')_{z'}$. In the case of $U'_q(\widehat{\sl}_2)$ this expression was obtained from various approaches in the works \cite{KR88}, \cite{Man14}, see also \cite{DJKMO88}, \cite{Agg17} for analogous expressions in the more general elliptic case. For $U_q(\widehat{\sl}_{n+1})$ the expression of Proposition \ref{Rexpression} was first obtained in \cite{BM16} using the methods of three-dimensional integrability.
\end{rem}

\subsection{Triple tensor products and inhomogeneous Yang-Baxter equations}

By Propositions \ref{Lmod}, \ref{irred} generically we have $V^{a_1}_{b_1}\otimes \dots\otimes V^{a_m}_{b_m}\cong L(\bm f)$ with $f_1(u)=\prod_i\frac{a_i^{-1}-a_iu}{b_i^{-1}-b_iu}$ and $f_r(u)=1$ for $r>1$. Hence, for two collections of generic parameters $a_1\dots a_m, b_1,\dots, b_m$ and $\tilde a_1\dots \tilde a_m, \tilde b_1,\dots, \tilde b_m$ the representations $V^{a_1}_{b_1}\otimes \dots\otimes V^{a_m}_{b_m}$ and $V^{\tilde a_1}_{\tilde b_1}\otimes \dots\otimes V^{\tilde a_m}_{\tilde b_m}$ are isomorphic if and only if, up to sign changes, $\tilde a$ is a permutation of $a$, $\tilde b$ is a permutation of $b$, and the total number of sign changes is even. Proposition \ref{basicintertwiner} allows to explicitly construct all such isomorphisms, since $W^a$ ($W^b$) is the isomorphism corresponding to a simple transposition of the parameters $a_i$ (respectively $b_i$), while the sign changes are trivial since $V^a_b=V(a/b)_{ab}=V^{-a}_{-b}$ and we have
$$
V^{a_1}_{b_1}\otimes V^{a_2}_{b_2}=V^{a_1}_{b_1}\otimes V^{-a_2}_{-b_2}\cong V^{-a_2}_{b_1}\otimes V^{a_1}_{-b_2}= V^{a_2}_{-b_1}\otimes V^{a_1}_{-b_2}\cong V^{a_1}_{-b_1}\otimes V^{a_2}_{-b_2}=V^{-a_1}_{b_1}\otimes V^{-a_2}_{b_2},
$$
where both isomorphisms above are constructed using $W^a$.

More importantly, since the representations are irreducible, the isomorphism $V^{a_1}_{b_1}\otimes \dots\otimes V^{a_m}_{b_m}\cong V^{\tilde a_1}_{\tilde b_1}\otimes \dots\otimes V^{\tilde a_m}_{\tilde b_m}$ is unique up to a scalar. Hence, if we have multiple ways of expressing the same isomorphism using $W^a$, $W^b$ and $R$, we obtain nontrivial equations, including the remarkable Yang-Baxter equations. Below we demonstrate this idea for triple tensor products, deriving two equalities which will be used in Section \ref{exchange-sect}.

For later use we summarize the expressions form Propositions \ref{basicintertwiner}, \ref{Rexpression}, setting:
\begin{multline}
\label{defWa}
W^a_{a_1,a_2,b_1}(\bI,\bJ,\bK,\bL):= \delta_{\bI+\bJ=\bK+\bL}\Phi(\bI-\bK, \bK; a_1/a_2, a_2/b_1)\\
=\delta_{\bI+\bJ=\bK+\bL}\delta_{\bI\geq\bK}(a_2/b_1)^{|\bI|-|\bK|}q^{\sum_{i<j}(I_i-K_i)K_j}\frac{(a_1/a_2;q)_{|\bI|-|\bK|}(a_2/b_1;q)_{|\bK|}}{(a_1/b_1;q)_{|\bI|}}\prod_{r=1}^n\frac{(q;q)_{I_r}}{(q;q)_{I_r-K_r}(q;q)_{K_r}},
\end{multline}
\begin{multline}
\label{defWb}
W^b_{a_2,b_1,b_2}(\bI,\bJ,\bK,\bL):= \delta_{\bI+\bJ=\bK+\bL} \Phi(\bL, \bJ-\bL; a_2/b_1, b_1/b_2)\\
=\delta_{\bI+\bJ=\bK+\bL}\delta_{\bJ\geq \bL} (b_1/b_2)^{|\bL|}q^{\sum_{i<j}L_i(J_j-L_j)}\frac{(a_2/b_1;q)_{|\bL|}(b_1/b_2;q)_{|\bJ|-|\bL|}}{(a_2/b_2;q)_{|\bJ|}}\prod_{r=1}^n\frac{(q;q)_{J_r}}{(q;q)_{L_r}(q;q)_{J_r-L_r}},
\end{multline}
\begin{multline}
\label{defR}
R_{a_1,b_1,a_2,b_2}(\bI,\bJ,\bK,\bL):=\delta_{\bI+\bJ=\bK+\bL} \sum_{\bP}\Phi(\bL-\bP, \bK; a_1/a_2, a_2/b_2)\Phi(\bP, \bJ-\bP; a_2/b_1, b_1/b_2)\\
=\delta_{\bI+\bJ=\bK+\bL} \sum_{\bP}\Phi(\bL, \bK-\bP; a_1/b_1, b_1/b_2)    \Phi(\bI-\bP, \bP; a_1/a_2, a_2/b_1).
\end{multline}
We do not include $C_{a_1,b_1;a_2,b_2}(\bI,\bJ)$ from Propositions \ref{basicintertwiner}, \ref{Rexpression} in the expressions above, in this way right-hand sides of \eqref{defWa}-\eqref{defR} are rational functions in $q$, not just in $q^{\frac12}$. Note also that we have replaced $a_i^2, b_i^2$ from Propositions \ref{basicintertwiner}, \ref{Rexpression} by $a_i, b_i$.
\begin{rem}
\label{polynomiality}
For the later use in Sections \ref{exchange-sect} and \ref{sqW-sect} we note that $\Phi(\bX,\bY, x/y, y/z)$ is a polynomial in $y$, and $y^{-|\bX|-|\bY|}\Phi(\bX,\bY, x/y, y/z)$ is a polynomial in $y^{-1}$. Hence $W^a_{a_1,a_2,b_1}(\bI,\bJ,\bK,\bL)$ is a polynomial in $a_2$, $W^b_{a_2,b_1,b_2}(\bI,\bJ,\bK,\bL)$ is a polynomial in $b_1$ and $a_2^{-|\bI|}W^a_{a_1,a_2,b_1}(\bI,\bJ,\bK,\bL)$ is a polynomial in $a_2^{-1}$. Moreover, $a_2^{-|\bI|}R_{a_1,b_1,a_2,b_2}(\bI,\bJ,\bK,\bL)$ is a polynomial in $b_1$ and $a_2^{-1}$. 
\end{rem}

\begin{prop} \label{inhomYB1}The following identity of rational functions in $a_1, a_2, a_3, b_1, b_2, b_3$ holds:
\begin{multline*}
\sum_{\bC_1,\bC_2,\bC_3} W^b_{a_3,b_1,b_2}(\bC_1, \bC_2, \bB_2, \bB_1) W^b_{a_2,b_1,b_3}(\bA_1, \bC_3, \bB_3, \bC_1)  W^b_{a_3,b_2,b_3}(\bA_2,\bA_3,\bC_3, \bC_2)\\
=\sum_{\bC_1,\bC_2,\bC_3} W^b_{a_2,b_2,b_3}(\bC_2,\bC_3, \bB_3, \bB_2) W^b_{a_3,b_1,b_3}(\bC_1,\bA_3, \bC_3, \bB_1) W^b_{a_2,b_1,b_2}(\bA_1,\bA_2,\bC_2,\bC_1).
\end{multline*} 
\end{prop}
\begin{proof}
It is enough to prove that for generic $a_1, a_2, a_3, b_1, b_2, b_3$ we have
\begin{multline*}
\sum_{\bC_1,\bC_2,\bC_3} W^b_{a^2_3,b^2_1,b^2_2}(\bC_1, \bC_2, \bB_2, \bB_1) W^b_{a^2_2,b^2_1,b^2_3}(\bA_1, \bC_3, \bB_3, \bC_1)  W^b_{a^2_3,b^2_2,b^2_3}(\bA_2,\bA_3,\bC_3, \bC_2)\\
=\sum_{\bC_1,\bC_2,\bC_3} W^b_{a^2_2,b^2_2,b^2_3}(\bC_2,\bC_3, \bB_3, \bB_2) W^b_{a^2_3,b^2_1,b^2_3}(\bC_1,\bA_3, \bC_3, \bB_1) W^b_{a^2_2,b^2_1,b^2_2}(\bA_1,\bA_2,\bC_2,\bC_1).
\end{multline*} 
Consider two isomorphisms $V^{a_1}_{b_1}\otimes V^{a_2}_{b_2}\otimes V^{a_3}_{b_3}\cong V^{a_1}_{b_3}\otimes V^{a_2}_{b_2}\otimes V^{a_3}_{b_1}$:
$$
(1\otimes W^b)(W^b\otimes 1)(1\otimes W^b):V^{a_1}_{b_1}\otimes V^{a_2}_{b_2}\otimes V^{a_3}_{b_3}\to V^{a_1}_{b_1}\otimes V^{a_2}_{b_3}\otimes V^{a_3}_{b_2}
\to V^{a_1}_{b_3}\otimes V^{a_2}_{b_1}\otimes V^{a_3}_{b_2}\to V^{a_1}_{b_3}\otimes V^{a_2}_{b_2}\otimes V^{a_3}_{b_1},
$$
$$
(W^b\otimes 1)(1\otimes W^b)(W^b\otimes 1):V^{a_1}_{b_1}\otimes V^{a_2}_{b_2}\otimes V^{a_3}_{b_3}\to V^{a_1}_{b_2}\otimes V^{a_2}_{b_1}\otimes V^{a_3}_{b_3}
\to V^{a_1}_{b_2}\otimes V^{a_2}_{b_3}\otimes V^{a_3}_{b_1}\to V^{a_1}_{b_3}\otimes V^{a_2}_{b_2}\otimes V^{a_3}_{b_1}.
$$
Both isomorphisms send $v_{\bm 0}\otimes v_{\bm 0}\otimes v_{\bm 0}$ to $v_{\bm 0}\otimes v_{\bm 0}\otimes v_{\bm 0}$, so by irreducibility they are equal. The claim now follows by applying Proposition \ref{basicintertwiner} to
\begin{multline*}
\langle v_{\bB_3}\otimes v_{\bB_2}\otimes v_{\bB_1}| (1\otimes W^b)(W^b\otimes 1)(1\otimes W^b) |v_{\bA_1}\otimes v_{\bA_2}\otimes v_{\bA_3}\rangle\\
=\langle v_{\bB_3}\otimes v_{\bB_2}\otimes v_{\bB_1}| (W^b\otimes 1)(1\otimes W^b)(W^b\otimes 1) |v_{\bA_1}\otimes v_{\bA_2}\otimes v_{\bA_3}\rangle.
\end{multline*}
Note that all coefficients $C_{a,b;a',b'}(\bI,\bJ,\bK,\bL)$ cancel out.
\end{proof}

\begin{prop}\label{inhomYB2}The following identity of rational functions holds:
\begin{multline*}
\sum_{\bC_1,\bC_2,\bC_3} W^b_{a_2,b_1,b_3}(\bC_1, \bC_2, \bB_2, \bB_1) R_{a_1, b_1, a_3,b_2}(\bA_1, \bC_3, \bB_3, \bC_1)  W^a_{a_2,a_3,b_2}(\bA_2,\bA_3,\bC_3, \bC_2)\\
=\sum_{\bC_1,\bC_2,\bC_3} W^a_{a_1,a_3,b_2}(\bC_2,\bC_3, \bB_3, \bB_2) R_{a_2,b_1,a_3,b_3}(\bC_1,\bA_3, \bC_3, \bB_1) W^b_{a_2,b_1,b_2}(\bA_1,\bA_2,\bC_2,\bC_1)
\end{multline*} 
\end{prop} 
\begin{proof}
The claim follows from the same argument as in Proposition \ref{inhomYB1}, applied to the morphisms
$$
(1\otimes W^b)(R\otimes 1)(1\otimes W^a):V^{a_1}_{b_1}\otimes V^{a_2}_{b_2}\otimes V^{a_3}_{b_3}\to V^{a_1}_{b_1}\otimes V^{a_3}_{b_2}\otimes V^{a_2}_{b_3}
\to V^{a_3}_{b_2}\otimes V^{a_1}_{b_1}\otimes V^{a_2}_{b_3}\to V^{a_3}_{b_2}\otimes V^{a_1}_{b_3}\otimes V^{a_2}_{b_1},
$$
$$
(W^a\otimes 1)(1\otimes R)(W^b\otimes 1):V^{a_1}_{b_1}\otimes V^{a_2}_{b_2}\otimes V^{a_3}_{b_3}\to V^{a_1}_{b_2}\otimes V^{a_2}_{b_1}\otimes V^{a_3}_{b_3}
\to V^{a_1}_{b_2}\otimes V^{a_3}_{b_3}\otimes V^{a_2}_{b_1}\to V^{a_3}_{b_2}\otimes V^{a_1}_{b_3}\otimes V^{a_2}_{b_1}.
$$
\end{proof}

\begin{rem}
The identity from Proposition \ref{inhomYB1} was first established in \cite{BK21}, where it was called an inhomogeneous Yang-Baxter equation and it was proved by algebraic manipulations starting with the Yang-Baxter equation 
$$
(1\otimes R)(R\otimes 1)(R\otimes 1)=(R\otimes 1)(1\otimes R)(R\otimes 1): V^{a_1}_{b_1}\otimes V^{a_2}_{b_2}\otimes V^{a_3}_{b_3}\to V^{a_3}_{b_3}\otimes V^{a_2}_{b_2}\otimes V^{a_1}_{b_1}.
$$
Back then it was not clear for us if there exists some quantum group reasoning behind the existence of such equations and if there is a systematic way to construct them. The discussion of this section answers both questions by considering instead of isomorphisms of the form $V_1\otimes V_2\otimes V_3\cong  V_3\otimes V_2\otimes V_1$ isomorphisms $V_1\otimes V_2\otimes V_3\cong  \tilde V_3\otimes \tilde V_2\otimes \tilde V_1$, where representations $\tilde V_i$ are different from $V_i$. This new understanding allows us to obtain Proposition \ref{inhomYB2} here, which is new.
\end{rem}

%%%%%%%%%%%%%%%

\section{Vertex models and transfer matrices}\label{exchange-sect}

From this point we move away from the quantum affine algebras and focus on applying the relations obtained in Section \ref{sect-intertw} to the algebraic-combinatorial objects of our interest. In this section we introduce row operators $\B(x \mid \calA, \calB)$, $\B^*(y\mid \calA, \calB)$ and prove exchange relations between them by iterating Propositions \ref{inhomYB1} and \ref{inhomYB2}. To make our expressions and manipulations clearer we also explain the language of vertex models here.

We use the following notation. Since in what follows we only need Propositions \ref{inhomYB1} and \ref{inhomYB2} when $n=1$, that is, when the quantum algebra in question is $U_q'(\widehat{\sl}_2)$, we replace all length $1$ compositions $\bI=(I)$  by nonnegative integers $I$. From now on we treat $q$ as a formal variable, and let $\calA=(a_0, a_1, \dots)$, $\calB=(b_0, b_1,\dots)$ denote two infinite sequences of parameters $a_i$ and $b_i$, which we treat as formal variables. Let $\Bbbk=\mathbb Q(q, \calA, \calB)$ denote the field of rational functions in $q,a_0, a_1, \dots$ and $b_0, b_1, \dots$. For a sequence $\mathcal X=(\chi_0, \chi_1, \dots)$ we set
$$
\tau^n\mathcal X=(\chi_n, \chi_{n+1}, \dots, ), \qquad \overline{\mathcal X}=(\chi_0^{-1},\chi_1^{-1},\dots).
$$

\subsection{Vertex models} By a \emph{vertex model} we mean the following data:
\begin{itemize}
 \item A collection of oriented lines in the plane, whose intersections are called \emph{vertices}, while the line segments between vertices are called \emph{edges}. The edges are oriented in the same way as the underlying lines, and each vertex has exactly two incoming edges and two outgoing edges. An edge is \emph{internal} if it connects a pair of distinct vertices, and is \emph{boundary} if it is connected to only one vertex. 
 
\item A collection of pairs of \emph{edge parameters} $(a,b)$, which are assigned to the edges and are constrained by the following rule: for each vertex if $(a, b)$ and $(a',b')$ are the parameters of the incoming edges of the vertex, then its outgoing edges have either parameters $(a,b)$ and $(a',b')$, or parameters $(a,b')$ and $(a',b)$. See Figure \ref{verticesFig} for the assignments satisfying this constrain.
\end{itemize}

A \emph{configuration} of a vertex model is an assignment of non-negative integer labels to the edges. In this text we usually denote these labels by capital letters $I,J,\dots$. Given a configuration around a vertex, that is, a collection of four edge labels $I,J,K,L$ attached to the adjacent edges, we define the corresponding \emph{vertex weight} by tracking the behavior of the edge parameters $a,a',b,b'$ and correspondingly using expressions \eqref{defWa},\eqref{defWb} and \eqref{defR} in the way demonstrated in Figure \ref{verticesFig}.

\begin{figure}
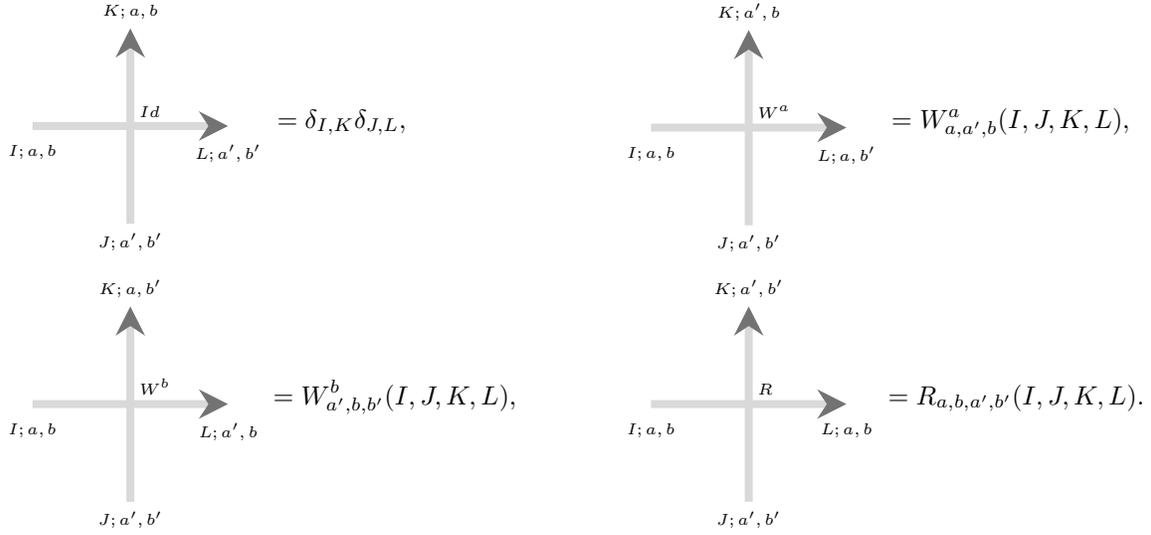

\begin{align*}
&\tikz{1.3}{
	\draw[fused] (-1,0) -- (1,0);
	\draw[fused] (0,-1) -- (0,1);
	\node[below] at (-1,-0.1) {\tiny \shortstack[c]{$I; a,b$}};\node[below] at (1,-0.1) {\tiny \shortstack[c]{$L; a',b'$}};
	\node[below] at (0,-1) {\tiny $J; a',b'$};\node[above] at (0,1) {\tiny $K; a,b$};
	\node[above right] at (0,0) {\tiny $Id$};
}=\delta_{I, K}\delta_{J,L},\qquad
&&\tikz{1.3}{
	\draw[fused] (-1,0) -- (1,0);
	\draw[fused] (0,-1) -- (0,1);
	\node[below] at (-1,-0.1) {\tiny \shortstack[c]{$I; a,b$}};\node[below] at (1,-0.1) {\tiny \shortstack[c]{$L; a,b'$}};
	\node[below] at (0,-1) {\tiny $J; a',b'$};\node[above] at (0,1) {\tiny $K; a',b$};
	\node[above right] at (0,0) {\tiny $W^a$};
}=W^a_{a,a',b}(I,J,K,L),
\\
&\tikz{1.3}{
	\draw[fused] (-1,0) -- (1,0);
	\draw[fused] (0,-1) -- (0,1);
	\node[below] at (-1,-0.1) {\tiny \shortstack[c]{$I; a,b$}};\node[below] at (1,-0.1) {\tiny \shortstack[c]{$L; a',b$}};
	\node[below] at (0,-1) {\tiny $J; a',b'$};\node[above] at (0,1) {\tiny $K; a,b'$};
	\node[above right] at (0,0) {\tiny $W^b$};
}=W^b_{a',b,b'}(I,J,K,L),\qquad
&&\tikz{1.3}{
	\draw[fused] (-1,0) -- (1,0);
	\draw[fused] (0,-1) -- (0,1);
	\node[below] at (-1,-0.1) {\tiny \shortstack[c]{$I; a,b$}};\node[below] at (1,-0.1) {\tiny \shortstack[c]{$L; a,b$}};
	\node[below] at (0,-1) {\tiny $J; a',b'$};\node[above] at (0,1) {\tiny $K; a',b'$};
	\node[above right] at (0,0) {\tiny $R$};
}=R_{a,b,a',b'}(I,J,K,L).
\end{align*}
\caption{\label{verticesFig} Possible vertices in vertex models and their weights.}
\end{figure}
In other words, we define four types of vertices $Id, W^a, W^b, R$, which are determined by the arrangement of the parameters $a,b,a',b'$. For example, the vertex of type $W^a$ preserves the parameters $a,a'$ but swaps the parameters $b,b'$ between the lines. Note that in each of these situations the vertex weight vanishes unless $I+J=K+L$, we call this fact the \emph{conservation law}.

The \emph{weight} of a configuration of a vertex model is the product of the weights of all vertices. A \emph{boundary condition} for a vertex model is an assignment of labels to the boundary edges. Given a vertex model and a boundary condition we define the corresponding \emph{partition function} as the sum of the weights of all configurations satisfying the boundary condition, that is, configurations whose labels of the boundary edges coincide with the boundary condition. For example, both sides of the identities from Propositions \ref{inhomYB1}, \ref{inhomYB2} can be rewritten as partition functions in the following way :
\begin{equation}
\label{WYB1}
\tikzbase{1.5}{-3}{
	\draw[fused] 
	(-1,1) node[left,scale=0.6] {\color{black} $A_1; a_1, b_1$} -- (1,1) -- (2,0) node[below,scale=0.6] {\color{black} $B_1; a_3,b_1$};
	\draw[fused] 
	(-1,0) node[left,scale=0.6] {\color{black} $A_2; a_2, b_2$} -- (1,0) -- (2,1) node[above,scale=0.6] {\color{black} $B_2; a_2, b_2$};
	\draw[fused] 
	(0,-1) node[below,scale=0.6] {\color{black} $A_3; a_3, b_3$} -- (0,0.5) -- (0,2) node[above,scale=0.6] {\color{black} $B_3; a_1, b_3$};
	\node[above right] at (0,0) {\tiny{ $W^b$}};
	\node[above right] at (0,1) {\tiny{ $W^b$}};
	\node[right] at (1.5,0.5) {\ \tiny{$W^b$}};
}
\qquad=
\qquad\tikzbase{1.5}{-3}{
	\draw[fused]
	(-2,0.5) node[above,scale=0.6] {\color{black} $A_1; a_1, b_1$} -- (-1,-0.5) -- (1,-0.5) node[right,scale=0.6] {\color{black} $B_1; a_3,b_1$};
	\draw[fused] 
	(-2,-0.5) node[below,scale=0.6] {\color{black} $A_2; a_2, b_2$} -- (-1,0.5)  -- (1,0.5) node[right,scale=0.6] {\color{black} $B_2;a_2,b_2$};
	\draw[fused] 
	(0,-1.5) node[below,scale=0.6] {\color{black} $A_3; a_3, b_3$} -- (0,0) -- (0,1.5) node[above,scale=0.6] {\color{black} $B_3;a_1, b_3$};
	\node[above right] at (0,-0.5) {\tiny{ $W^b$}};
	\node[above right] at (0,0.5) {\tiny{ $W^b$}};
	\node[right] at (-1.5,0) {\ \tiny{$W^b$}};
}
\end{equation}

\begin{equation}
\label{WYB2}
\tikzbase{1.5}{-3}{
	\draw[fused] 
	(-1,1) node[left,scale=0.6] {\color{black} $A_1; a_1, b_1$} -- (1,1) -- (2,0) node[below,scale=0.6] {\color{black} $B_1; a_2, b_1$};
	\draw[fused] 
	(-1,0) node[left,scale=0.6] {\color{black} $A_2; a_2, b_2$} -- (1,0) -- (2,1) node[above,scale=0.6] {\color{black} $B_2; a_1, b_3$};
	\draw[fused] 
	(0,-1) node[below,scale=0.6] {\color{black} $A_3; a_3, b_3$} -- (0,0.5) -- (0,2) node[above,scale=0.6] {\color{black} $B_3; a_3,b_2$};
	\node[above right] at (0,0) {\tiny{ $W^a$}};
	\node[above right] at (0,1) {\tiny{ $R$}};
	\node[right] at (1.5,0.5) {\ \tiny{$W^b$}};
}
\qquad
=\qquad
\tikzbase{1.5}{-3}{
	\draw[fused]
	(-2,0.5) node[above,scale=0.6] {\color{black} $A_1;a_1,b_1$} -- (-1,-0.5) -- (1,-0.5) node[right,scale=0.6] {\color{black} $B_1;a_2,b_1$};
	\draw[fused] 
	(-2,-0.5) node[below,scale=0.6] {\color{black} $A_2;a_2, b_2$} -- (-1,0.5)  -- (1,0.5) node[right,scale=0.6] {\color{black} $B_2;a_1,b_3$};
	\draw[fused] 
	(0,-1.5) node[below,scale=0.6] {\color{black} $A_3;a_3,b_3$} -- (0,0) -- (0,1.5) node[above,scale=0.6] {\color{black} $B_3;a_3,b_2$};
	\node[above right] at (0,-0.5) {\tiny{ $R$}};
	\node[above right] at (0,0.5) {\tiny{ $W^a$}};
	\node[right] at (-1.5,0) {\ \tiny{$W^b$}};
}
\end{equation}
Here to each boundary edge we assign a triple $(I;a,b)$, where $I$ is the label of the boundary condition, and $(a,b)$ are the corresponding edge parameters. To lighten the notation, instead of specifying the edge parameters of the internal edges we specify the types of the vertices, $W^a,W^b$ or $R$; the labels of the internal edges can be uniquely reconstructed from this information. 

\begin{rem}
The vertex models described above are directly related to the representations from Section \ref{sect-intertw}: we can think about the edge with parameters $(a,b)$ as of the representation $V^{\sqrt{a}}_{\sqrt{b}}$, and each vertex corresponds to one of the operators $V^{\sqrt{a_1}}_{\sqrt{b_1}}\otimes V^{\sqrt{a_2}}_{\sqrt{b_2}}\to V^{\sqrt{\tilde a_1}}_{\sqrt{\tilde b_1}}\otimes V^{\sqrt{\tilde a_2}}_{\sqrt{\tilde b_2}}$. Note that the whole partition function corresponds to a matrix coefficient of the isomorphism between two tensor products of the form $V^{\sqrt{a_1}}_{\sqrt{b_1}}\otimes\dots\otimes V^{\sqrt{a_L}}_{\sqrt{b_L}}$, with the parameters $a_i$ and $b_i$ determined by the parameters of the incoming and outgoing edges.
\end{rem}

\begin{rem} 
One can also write the usual Yang-Baxter equation $(1\otimes R)(R\otimes 1)(R\otimes 1)=(R\otimes 1)(1\otimes R)(R\otimes 1)$ using vertex models:
\begin{equation}
\label{YB}
\tikzbase{1.5}{-3}{
	\draw[fused] 
	(-1,1) node[left,scale=0.6] {\color{black} $A_1; a_1, b_1$} -- (1,1) -- (2,0) node[below,scale=0.6] {\color{black} $B_1; a_1, b_1$};
	\draw[fused] 
	(-1,0) node[left,scale=0.6] {\color{black} $A_2; a_2, b_2$} -- (1,0) -- (2,1) node[above,scale=0.6] {\color{black} $B_2; a_2, b_2$};
	\draw[fused] 
	(0,-1) node[below,scale=0.6] {\color{black} $A_3; a_3, b_3$} -- (0,0.5) -- (0,2) node[above,scale=0.6] {\color{black} $B_3; a_3,b_3$};
	\node[above right] at (0,0) {\tiny{ $R$}};
	\node[above right] at (0,1) {\tiny{ $R$}};
	\node[right] at (1.5,0.5) {\ \tiny{$R$}};
}\qquad
=\qquad
\tikzbase{1.5}{-3}{
	\draw[fused]
	(-2,0.5) node[above,scale=0.6] {\color{black} $A_1;a_1,b_1$} -- (-1,-0.5) -- (1,-0.5) node[right,scale=0.6] {\color{black} $B_1;a_1,b_1$};
	\draw[fused] 
	(-2,-0.5) node[below,scale=0.6] {\color{black} $A_2;a_2, b_2$} -- (-1,0.5)  -- (1,0.5) node[right,scale=0.6] {\color{black} $B_2;a_2,b_2$};
	\draw[fused] 
	(0,-1.5) node[below,scale=0.6] {\color{black} $A_3;a_3,b_3$} -- (0,0) -- (0,1.5) node[above,scale=0.6] {\color{black} $B_3;a_3,b_3$};
	\node[above right] at (0,-0.5) {\tiny{ $R$}};
	\node[above right] at (0,0.5) {\tiny{ $R$}};
	\node[right] at (-1.5,0) {\ \tiny{$R$}};
}
\end{equation}
Note that when $a_1=a_2=a_3$, the equation \eqref{WYB1} becomes equivalent to the usual Yang-Baxter equation, while setting $a_1=a_2, b_2=b_3$ makes \eqref{WYB2} and \eqref{YB} equivalent. 
\end{rem}

\subsection{Row operators} For $N\in\Z_{\geq 0}$ define $V^{(N)}$ as the vector space over $\Bbbk$ with a basis consisting of the vectors $|A_1, A_2, \dots, A_N\rangle$ enumerated by $N$-tuples $(A_1, \dots, A_N)\in\mathbb Z^N_{\geq 0}$. Equivalently, the same basis can be enumerated by partitions $\mu$ of length at most $N$ in the following way:
$$
|\mu\rangle=|\mu_1-\mu_2, \mu_2-\mu_3, \dots, \mu_N\rangle;
$$
we use both notations interchangeably. For $N<M$ we have an embedding $V^{(N)}\subset V^{(M)}$ which is defined by $|\mu\rangle\mapsto|\mu\rangle$ where $l(\mu)\leq N<M$, or, equivalently, $|A_1, \dots, A_N\rangle\mapsto|A_1, \dots, A_N, 0, \dots, 0\rangle$. Let $\langle\lambda|$, $\langle B_1, \dots, B_N|$ denote the vectors dual to the basis $\{|\mu\rangle\}_{l(\mu)\leq N}$. 

Recall that $\calA=(a_0, a_1, a_2, \dots)$ and $\calB=(b_0, b_1, b_2, \dots)$ denote two infinite sequences of parameters. For variables $x,y$ we define row-to-row transfer matrices $T^{b,N}_{I, L}(x\mid \calA,\calB), T^{a,N}_{K, J}(y\mid \calA,\calB):V^{(N)}\to V^{(N)}$ by setting
\begin{equation}
 \label{defTb}
 T^{b,N}_{I,L}(x \mid\calA,\calB):|J_1, \dots, J_N\rangle\mapsto \sum_{\substack{K_1,\dots, K_N:\\ L_r:=I+J_{[1,r]}-K_{[1,r]}\geq 0,\\ L_N=L}} \prod_{r=1}^NW^b_{a_{r+1},x, b_r}(L_{r-1},J_r,K_r, L_r)|K_1,\dots, K_N\rangle
\end{equation}
\begin{equation}
 \label{defTa}
 T^{a,N}_{K,J}(y \mid\calA,\calB):|I_1, \dots, I_N\rangle\mapsto \sum_{\substack{L_1,\dots, L_N:\\ J_r:=K+L_{[1,r]}-I_{[1,r]}\geq 0,\\ J_N=J}}\prod_{r=1}^N (y/b_0)^{-I_{r}}W^a_{a_r,y,b_r}(I_r, J_r, J_{r-1},L_r)|L_1,\dots, L_N\rangle,
\end{equation}
where for a sequence $(A_1, \dots, A_N)$ we set $A_{[1,0]}=0$, $A_{[1,r]}=A_1+\dots+A_r$ for $r\geq 1$. The same definitions can be graphically represented using the following partition functions:
\begin{equation}
\label{defTbgraph}
 \langle K_1, \dots, K_N|\, T^{b,N}_{I,L}(x \mid\calA,\calB)\, |J_1, \dots, J_N\rangle=
\tikzbase{0.9}{-2}{
	\draw[fused*] (-1,0) -- (3.1,0);
	\foreach\x in {0,...,1}{
		\draw[fused] (2*\x,-1) -- (2*\x,1);
	}
	\draw[fused] (5,-1) -- (5,1);
	\node[left] at (-1,0) {\tiny $I;a_1,x$};
	\node[right] at (3.2,0) {\dots};
	\draw[fused] (4,0) -- (6,0);
	\node[right] at (6,0) {\tiny $L; a_N, x$};
	\node[below] at (0,-1) {\tiny $J_1;a_2,b_1$};\node[above] at (0,1) {\tiny $K_1;a_1, b_1$};
	\node[below] at (2,-1) {\tiny $J_2;a_3,b_2$};\node[above] at (2,1) {\tiny $K_2;a_2, b_{2}$};
	\node[below] at (5,-1) {\tiny $J_N;a_{N+1},b_N$};\node[above] at (5,1) {\tiny $K_N;a_{N}, b_N$};
	\node[above right] at (0,0) {\tiny $W^b$};
	\node[above right] at (2,0) {\tiny $W^b$};
	\node[above right] at (5,0) {\tiny $W^b$};
    }
\end{equation}
\begin{equation*}
\langle L_1, \dots, L_N|\, T^{a,N}_{K,J}(y \mid\calA,\calB)\, |I_1, \dots, I_N\rangle=(y/b_0)^{-I_{[1,N]}}\(
 \tikzbase{0.9}{-2}{
	\draw[fused] (3.1,0) -- (-1,0);
	\foreach\x in {0,...,1}{
		\draw[fused] (2*\x,-1) -- (2*\x,1);
	}
	\draw[fused] (5,-1) -- (5,1);
	\node[below] at (-1,-0.1) {\tiny $K;y,b_1$};
	\node[right] at (3.2,0) {\dots};
	\draw[fused*] (4,0) -- (6,0);
	\node[below] at (6,0) {\tiny $J; y, b_{N+1}$};
	\node[below] at (0,-1) {\tiny $I_1;a_1,b_1$};\node[above] at (0,1) {\tiny $L_1;a_1, b_2$};
	\node[below] at (2,-1) {\tiny $I_2;a_2,b_2$};\node[above] at (2,1) {\tiny $L_2;a_2, b_3$};
	\node[below] at (5,-1) {\tiny $I_N;a_N,b_N$};\node[above] at (5,1) {\tiny $L_N;a_{N}, b_{N+1}$};
	\node[above right] at (0,0) {\tiny $W^a$};
	\node[above right] at (2,0) {\tiny $W^a$};
	\node[above right] at (5,0) {\tiny $W^a$};
    }\)
\end{equation*}

Here we follow the same conventions as in \eqref{WYB1}, \eqref{WYB2}, and in the second partition function we use the vertices of type $W^a$ rotated by $90^\circ$. Note that for each partition function above there exists at most one configuration with nonzero weight, and $L_r, J_r$ from \eqref{defTb}, \eqref{defTa} respectively are the labels of the internal edges in these configurations.

In the following statement we summarize simple facts about the operators $T^{b,N}_{I,0}$ and $T^{a,N}_{K,0}$:

\begin{prop}\label{stability} Let $N\in\mathbb Z_{\geq 0}$ and $\lambda,\mu$ be partitions of length at most $N$.

(1) Duality: The operators $T^{a,N}_{K,0}$ and $T^{b,N}_{K,0}$ are dual in the following sense:
$$
\langle \mu |\, T^{a,N}_{K,0}(y\mid \calA,\calB)\, | \lambda\rangle=(a_0/y)^{K}\frac{(y/b_1;q)_{K}}{(q;q)_K}\frac{\psi_\lambda(\calA,\calB)}{\psi_\mu(\calA,\tau\calB)}\langle \lambda |\, T ^{b,N}_{K,0}(y^{-1}\mid \overline{\calB}, \overline{\calA}) \,| \mu\rangle,
$$
where $\tau\calB$ denotes the shifted sequence and $\psi_\lambda(\calA,\calB)$ is defined by
$$
\psi_\lambda(\calA,\calB)=\prod_{r\geq 0} (b_r/a_r)^{\lambda_{r+1}}\prod_{r\geq 1}\frac{(q;q)_{\lambda_r-\lambda_{r+1}}}{(a_r/b_r;q)_{\lambda_r-\lambda_{r+1}}}.
$$

(2) Polynomiality: $\langle \lambda|\, T^{b,N}_{I,L}(x\mid \calA, \calB)\, |\mu\rangle$ is a polynomial in $x$, and $\langle \mu|\, T^{a,N}_{K,J}(y\mid \calA, \calB)\, |\lambda\rangle$ is a polynomial in $y^{-1}$. Moreover, both are rational functions in $q$ regular at $q=0$.

(3) Interlacing: $\langle \lambda|\, T^{b,N}_{I,0}(x\mid\calA, \calB)\, |\mu\rangle=\langle \mu|\, T^{a,N}_{I,0}(y\mid \calA, \calB)\, |\lambda\rangle=0$ unless $\lambda_1-\mu_1=I$ and $\mu\prec \lambda$, that is,
$$
\lambda_1\geq\mu_1\geq\lambda_2\geq\mu_2\geq\dots.
$$
In particular, for any $M>N$ we have $T^{b,M}_{I,0}(x\mid\calA, \calB)\, V^{(N)}\subset V^{(N+1)}$ and $T^{a,M}_{K,0}(y\mid \calA, \calB)\, V^{(N)}\subset V^{(N)}$.

(4) Stability: For any $M>N$ we have
$$
\langle \lambda|\, T^{b,M}_{I,0}(x\mid\calA, \calB)\, |\mu\rangle=\langle \lambda|\, T^{b,N}_{I,0}(x\mid\calA, \calB)\, |\mu\rangle,\qquad \langle \mu|\, T^{a,M}_{K,0}(x\mid \calA, \calB)\, |\lambda\rangle=\langle \mu|\, T^{a,N}_{K,0}(x\mid \calA, \calB)\, |\lambda\rangle.
$$
\end{prop}
\begin{proof} 
Part (1) follows at once from \eqref{defTb}, \eqref{defTa} and the relation
\begin{multline*}
W^a_{a_r,y,b_r}(I_r,J_r,J_{r-1},L_r)\\
=(a_r/b_{r})^{-J_r}(y/b_r)^{L_r}\frac{(q;q)_{I_r}}{(a_{r}/b_r;q)_{I_r}}\frac{(q;q)_{J_r}}{(y/b_{r+1};q)_{J_r}}\frac{(y/b_r;q)_{J_{r-1}}}{(q;q)_{J_{r-1}}}\frac{(a_r/b_{r+1};q)_{L_r}}{(q;q)_{L_r}}W^b_{b_{r+1}^{-1},y^{-1},a_r^{-1}}(J_{r-1}, L_r, I_r, J_r),
\end{multline*}
which is verified using \eqref{defWa} and \eqref{defWb}. Part (2) is also immediate from \eqref{defWa}, \eqref{defWb}, see Remark \ref{polynomiality}.

For the interlacing statement from part (3), by part (1) it is enough to prove it for $T^{b,N}_{I,0}(x\mid\calA, \calB)$. Set $J_r=\mu_{r}-\mu_{r+1}$, $K_r=\lambda_r-\lambda_{r+1}$ and $L_r=I+J_{[1,r]}-K_{[1,r]}$. Recall from \eqref{defWb} that the weights $W^b_{a_{r+1}, x, b_r}(L_{r-1},J_r,K_r,L_r)$ vanish unless $J_r\geq L_r\geq 0$. Hence, for $\langle K_1, \dots, K_N|\, T^{b,N}_{I,0}(x\mid\calA, \calB)\, |J_1, \dots, J_N\rangle$ to be nonzero we must have $J_r\geq L_r\geq 0$ for every $r$, and moreover $L_N=0$. Noting that $L_r= I+\mu_1-\mu_{r+1}-\lambda_1+\lambda_{r+1}$ and, in particular, $L_N=I+\mu_1-\lambda_1$, we get that $\langle \lambda|\, T^{b,N}_{I,0}(x\mid\calA,\calB)\, |\mu\rangle=0$ unless $I=\lambda_1-\mu_1$ and $\mu_{r}-\mu_{r+1}\geq I+\mu_1-\mu_{r+1}-\lambda_1+\lambda_{r+1}\geq 0$ for every $r$. The last inequality is equivalent to $\mu_r\geq\lambda_{r+1}\geq \mu_{r+1}$ for all $r=1, \dots, N$, and since $\lambda_1-\mu_1=I\geq 0$ the interlacing $\mu\prec\lambda$ follows.

The second statement of part (3) follows immediately from the interlacing by noticing that $\mu\prec\lambda$ implies $l(\mu)\leq l(\lambda)\leq l(\mu)+1$. Finally, part (4) follows from the fact that $W^b_{a,x,b}(0,0,0,0)=W^a_{a,y,b}(0,0,0,0)=1$, and so if $l(\lambda),l(\mu)\leq N<M$ we have
$$
\langle \lambda|\, T^{b,M}_{I,0}(x\mid\calA, \calB)\, |\mu\rangle=W^b_{a_{N+1},x,b_N}(0,0,0,0) \langle \lambda|\, T^{b,M-1}_{I,0}(x\mid\calA, \calB)\, |\mu\rangle =\langle \lambda|\, T^{b,M-1}_{I,0}(x\mid\calA, \calB)\, |\mu\rangle
$$
and similarly for $T^{a,M}_{K,0}$.
\end{proof}

Let $V:=\bigcup_N V^{(N)}$ be a vector space over $\Bbbk$ with a basis $\{|\mu\rangle\}_{\mu\in\Y}$ where we have no restrictions on the length of $\mu$. The natural embeddings $V^{(N)}\subset V$ are given by identifying the vectors $|\mu\rangle$ in $V$ and $V^{(N)}$ when $l(\mu)\leq N$. By parts (3) and (4) of Proposition \ref{stability} we can define operators $\mathbb T^b_I(x\mid\calA,\calB),\mathbb T^a_K(y\mid\calA,\calB): V\to V$ by setting for $v\in V^{(N)}$
$$
\mathbb T^b_I(x\mid \calA,\calB)\ v= T^{b,N+1}_{I,0}(x\mid\calA,\calB)\ v, \qquad \mathbb T^a_K(y\mid\calA,\calB)\ v= T^{a,N}_{K,0}(y\mid\calA,\calB)\ v.
$$
Using these operators, we can formally define
$$
\mathbb B(x\mid \calA,\calB)=\sum_{r\geq 0}(x/b_0)^r\frac{(a_1 x^{-1};q)_r}{(q;q)_r} \mathbb T^b_{r}(x\mid\calA,\calB),
$$
$$
\mathbb B^*(y\mid \calA,\calB)=\sum_{r\geq 0}\mathbb T^a_{r}(y\mid\calA,\calB).
$$
More precisely, $\B^*(y\mid\calA,\calB)$ is an operator $V\to V$ such that
$$
\langle \mu|\, \B^*(y\mid\calA,\calB)\, |\lambda\rangle  = \langle \mu|\, \mathbb T^a_{\lambda_1-\mu_1}(y\mid\calA,\calB)\, |\lambda\rangle,
$$
which is well-defined since $\langle \mu|\, \mathbb T^a_{\lambda_1-\mu_1}(y\mid\calA,\calB)\, |\lambda\rangle\neq 0$ only if $\mu\prec \lambda$, and there is a finite number of such $\mu$ given a fixed $\lambda$. On the other hand, $\B(x\mid\calA, \calB)$ is not a well-defined operator $V\to V$, since given a fixed $\mu$ we have $\mathbb T^b_r(x\mid\calA, \calB)\, |\mu\rangle\neq 0$ for infinitely many $r$ in general. However, we can resolve this issue by considering $q$ and $x$ as formal variables: by Proposition \ref{stability} the matrix coefficients $\langle \lambda|\, \mathbb T_{\lambda_1-\mu_1}^b(x\mid\calA, \calB)\,|\mu\rangle$ can be viewed as elements of the algebra of formal power series $\Bbbk[[x,q]]$, while the combined degree with respect to $x$ and $q$ of $(x/b_0)^r\frac{(a_1 x^{-1};q)_r}{(q;q)_r}$ is at least $r$. Hence $\B(x\mid \calA, \calB)$ is a well-defined operator $V[[x,q]]\to V[[x,q]]$, where $V[[x,q]]$ is the vector space of formal power series in $x,q$ with coefficients in $V$.

The operators $\mathbb B(x\mid \calA,\calB)$ and $\mathbb B^*(y\mid \calA, \calB)$ are dual to each other in the following sense:
\begin{prop}\label{BCdual}
For any $\lambda,\mu$ we have 
$$
\langle \mu |\, \B^*(y\mid \calA,\calB)\, | \lambda\rangle=\frac{\psi_\lambda(\calA,\calB)}{\psi_\mu(\calA,\tau\calB)}\langle \lambda |\, \B(y^{-1}\mid \overline{\calB}, \overline{\calA})\, | \mu\rangle,
$$
where $\psi_\lambda(\calA,\calB)$ is defined by
$$
\psi_\lambda(\calA,\calB)=\prod_{r\geq 0} (b_r/a_r)^{\lambda_{r+1}}\prod_{r\geq 1}\frac{(q;q)_{\lambda_r-\lambda_{r+1}}}{(a_r/b_r;q)_{\lambda_r-\lambda_{r+1}}}.
$$
\end{prop}
\begin{proof}
Follows from Proposition \ref{stability}, part (1).
\end{proof}

\subsection{Exchange relations} Now we consider commutation relations between the operators $\B(x\mid\calA, \calB)$ and $\B^*(y\mid\calA,\calB)$, which follow from the equations \eqref{WYB1}, \eqref{WYB2}. In fact, our construction of the operators $\mathbb T^b_{r}(x\mid\calA, \calB),\mathbb T^a_r(y\mid\calA,\calB)$ is motivated by \eqref{WYB1}, \eqref{WYB2} and the desire for the arguments below to work, while the linear combinations $\B(x\mid\calA, \calB)$, $\B^*(x\mid\calA, \calB)$ are distinguished by especially nice commutation relations.

\begin{prop}\label{BCexchange} We have
$$
\B^*(y\mid \calA,\calB)\B(x\mid\calA,\calB)=\frac{(a_1/y;q)_\infty(x/b_1;q)_\infty}{(a_1/b_1;q)_\infty(x/y;q)_\infty}\,\B(x\mid\calA,\tau\calB)\B^*(y\mid\tau\calA,\calB),
$$
or, more precisely, the following relation holds:
$$
\sum_{\lambda}\langle \nu|\, \B^*(y\mid \calA,\calB)\, |\lambda\rangle\langle\lambda|\, \B(x\mid\calA,\calB)\, |\mu\rangle=\frac{(a_1/y;q)_\infty(x/b_1;q)_\infty}{(a_1/b_1;q)_\infty(x/y;q)_\infty}\sum_{\lambda} \langle\nu|\,\B(x\mid\calA,\tau\calB)\, |\lambda\rangle\langle\lambda|\, \B^*(y\mid\tau\calA,\calB)\,|\mu\rangle,
$$
where both sides are viewed as formal power series in $x,y^{-1}$ and $q$.
\end{prop}
\begin{proof} We prove the latter identity, fixing $\mu,\nu$ from the statement and setting $L_r=\nu_r-\nu_{r+1}$, $J_r=\mu_r-\mu_{r+1}$.
\paragraph{{\bf Step 1:}} We start by consequently applying \eqref{WYB2} to get a relation between the operators $T^{b,N}_{I,0}$ and $T^{a,N}_{K,0}$. Using \eqref{WYB2} we have
$$
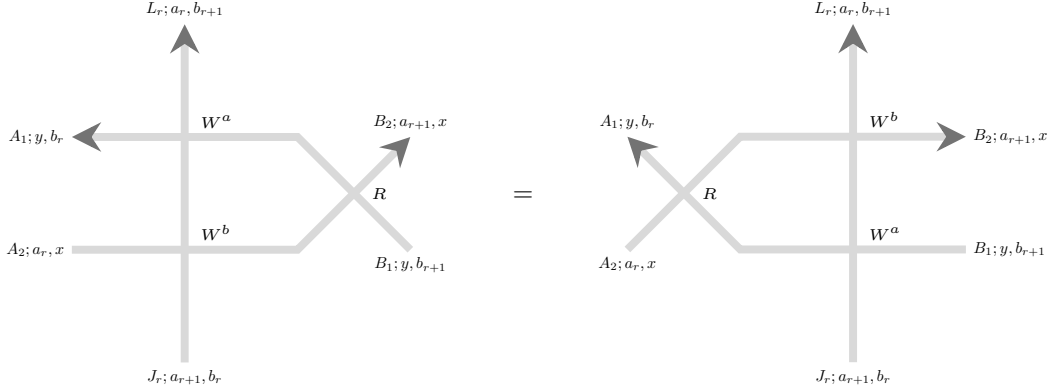
\tikzbase{1.5}{-3}{
	\draw[fused] 
	(2,0) node[below,scale=0.6] {\color{black} $B_1; y, b_{r+1}$} -- (1,1) -- (-1,1) node[left,scale=0.6] {\color{black} $A_1; y, b_r$} ;
	\draw[fused] 
	(-1,0) node[left,scale=0.6] {\color{black} $A_2; a_r, x$} -- (1,0) -- (2,1) node[above,scale=0.6] {\color{black} $B_2; a_{r+1}, x$};
	\draw[fused] 
	(0,-1) node[below,scale=0.6] {\color{black} $J_r; a_{r+1}, b_r$} -- (0,0.5) -- (0,2) node[above,scale=0.6] {\color{black} $L_r; a_r,b_{r+1}$};
	\node[above right] at (0,0) {\tiny{ $W^b$}};
	\node[above right] at (0,1) {\tiny{ $W^a$}};
	\node[right] at (1.5,0.5) {\ \tiny{$R$}};
}\qquad=
\qquad
\tikzbase{1.5}{-3}{
	\draw[fused]
	(1,-0.5) node[right,scale=0.6] {\color{black} $B_1;y,b_{r+1}$} -- (-1,-0.5) -- (-2,0.5) node[above,scale=0.6] {\color{black} $A_1;y,b_r$};
	\draw[fused] 
	(-2,-0.5) node[below,scale=0.6] {\color{black} $A_2;a_r, x$} -- (-1,0.5)  -- (1,0.5) node[right,scale=0.6] {\color{black} $B_2;a_{r+1},x$};
	\draw[fused] 
	(0,-1.5) node[below,scale=0.6] {\color{black} $J_r;a_{r+1},b_r$} -- (0,0) -- (0,1.5) node[above,scale=0.6] {\color{black} $L_r;a_r,b_{r+1}$};
	\node[above right] at (0,-0.5) {\tiny{ $W^a$}};
	\node[above right] at (0,0.5) {\tiny{ $W^b$}};
	\node[right] at (-1.5,0) {\ \tiny{$R$}};
}
$$
where $A_1,A_2,B_1,B_2$ are arbitrary nonnegative integers. Choosing $N\in\Z_{\geq 0}$ such that $N> l(\mu),l(\nu)$  and applying these identities consequently with $r=N, N-1, \dots, 2, 1$ we get for arbitrary $I,K\in\Z_{\geq 0}$
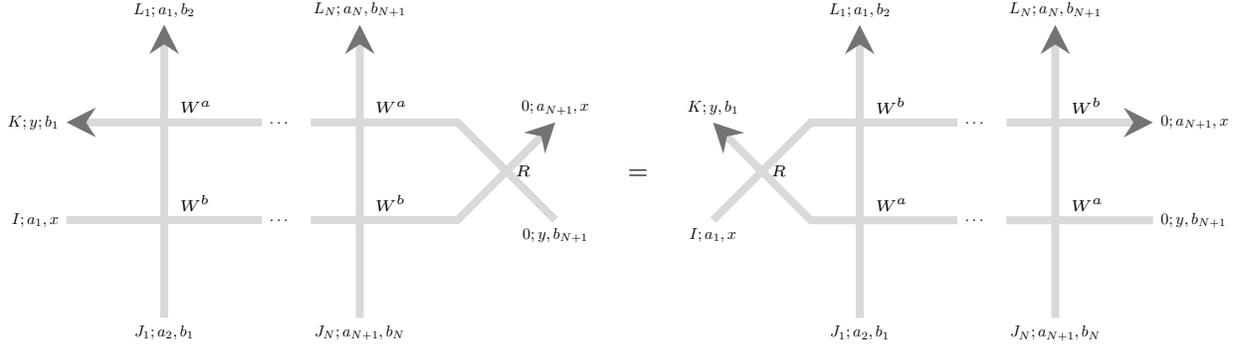
\begin{equation}
\label{BCZipper}
\tikzbase{1.3}{-3}{
	\draw[fused] 
	(-1,1) node[right,scale=0.6] {\color{black} $\dots$} -- (-3,1) node[left,scale=0.6] {\color{black} $K;y;b_1$};
	\draw[fused*] 
	(-3,0)  node[left,scale=0.6] {\color{black} $I;a_1,x$} -- (-1,0) node[right,scale=0.6] {\color{black} $\dots$};
	\draw[fused*] 
	(2,0) node[below,scale=0.6] {\color{black} $0;y,b_{N+1}$} -- (1,1)  -- (-0.5,1);
	\draw[fused] 
	(-0.5,0)  -- (1,0)  -- (2,1) node[above,scale=0.6] {\color{black} $0;a_{N+1},x$};
	\draw[fused] 
	(0,-1) node[below,scale=0.6] {\color{black} $J_N;a_{N+1},b_N$} -- (0,2) node[above,scale=0.6] {\color{black} $L_N;a_N,b_{N+1}$};
	\draw[fused] 
	(-2,-1) node[below,scale=0.6] {\color{black} $J_1;a_2,b_1$} -- (-2,2) node[above,scale=0.6] {\color{black} $L_1;a_1,b_2$};
	\node[above right] at (0,1) {\tiny{ $W^a$}};
	\node[above right] at (0,0) {\tiny{ $W^b$}};
	\node[above right] at (-2,1) {\tiny{ $W^a$}};
	\node[above right] at (-2,0) {\tiny{ $W^b$}};
	\node[right] at (1.5,0.5) {\tiny{$R$}};
}\quad
=\quad
\tikzbase{1.3}{-3}{
	\draw[fused]
	(2,-0.5) node[right,scale=0.6] {\color{black} $\dots$} -- (0.5,-0.5)  -- (-0.5,0.5) node[above,scale=0.6] {\color{black} $K;y,b_1$};
	\draw[fused*] 
	(-0.5,-0.5) node[below,scale=0.6] {\color{black} $I;a_1,x$} -- (0.5,0.5)  -- (2,0.5) node[right,scale=0.6] {\color{black} $\dots$};
	\draw[fused*]
	(4,-0.5) node[right,scale=0.6] {\color{black} $0;y,b_{N+1}$} -- (2.5,-0.5);
	\draw[fused] 
	(2.5,0.5)  -- (4,0.5) node[right,scale=0.6] {\color{black} $0;a_{N+1},x$};
	\draw[fused] 
	(1,-1.5) node[below,scale=0.6] {\color{black} $J_1;a_2,b_1$} -- (1,1.5) node[above,scale=0.6] {\color{black} $L_1;a_1,b_2$};
	\draw[fused] 
	(3,-1.5) node[below,scale=0.6] {\color{black} $J_N;a_{N+1},b_N$} -- (3,1.5) node[above,scale=0.6] {\color{black} $L_N;a_N,b_{N+1}$};
	\node[above right] at (1,-0.5) {\tiny{ $W^a$}};
	\node[above right] at (1,0.5) {\tiny{ $W^b$}};
	\node[above right] at (3,-0.5) {\tiny{ $W^a$}};
	\node[above right] at (3,0.5) {\tiny{ $W^b$}};
	\node[right] at (0,0) {\tiny{$R$}};
}
\end{equation}
Comparing with \eqref{defTb}, \eqref{defTa}, the identity of partition functions above is equivalent to
\begin{multline}
\label{exchangeproof1}
\sum_{\tilde I, \tilde K}\sum_{\lambda}(y/b_0)^{\lambda_1} \langle\nu|\, T^{a,N}_{K,\tilde K}(y\mid\calA,\calB)\, |\lambda\rangle\langle\lambda|\, T^{b,N}_{I,\tilde I}(x\mid\calA,\calB)\, |\mu\rangle\ R_{a_{N+1},x,y,b_{N+1}}(\tilde I,0,\tilde K, 0)\\
=\sum_{J, L}\sum_{\lambda}R_{a_{1},x,y,b_{1}}(I,J,K, L) \langle\nu|\,T^{b,N}_{L,0}(x\mid\calA,\tau\calB)\,|\lambda\rangle  (y/b_0)^{\mu_1} \langle\lambda|\, T^{a,N}_{J,0}(y\mid\tau\calA,\calB)|\mu\rangle,
\end{multline}
where in the left-hand side the sum is over partitions $\lambda$ such that $l(\lambda)\leq l(\mu)+1\leq N$, with the restriction on length coming from Proposition \ref{stability}, part (3).

From \eqref{defTb} we have
$$\langle\lambda|\,T^{b,N}_{I,\tilde I}(x\mid\calA,\calB)\,|\mu\rangle=\langle\widetilde\lambda|\,T^{b,N-1}_{I,\tilde I+\lambda_{N}-\mu_N}(x\mid\calA,\calB)\, |\widetilde\mu\rangle\ W^b_{a_{N+1},x,b_{N}}(\tilde I+\lambda_{N}-\mu_N, \mu_N, \lambda_{N}, \tilde I),$$
where $\widetilde\lambda:=(\lambda_1-\lambda_N, \dots,\lambda_{N-2}-\lambda_N, \lambda_{N-1}-\lambda_N,0,\dots)$ and $\widetilde\mu:=(\mu_1-\mu_N, \dots, \mu_{N-2}-\mu_N, \mu_{N-1}-\mu_N,0\dots)$. But $N>l(\mu)$, so $\mu_N=0$ and hence 
$$\langle\lambda|\,T^{b,N}_{I,\tilde I}(x\mid\calA,\calB)\,|\mu\rangle=\langle\widetilde\lambda|\,T^{b,N-1}_{I,\tilde I+\lambda_{N}}(x\mid\calA,\calB)\, |\mu\rangle\ W^b_{a_{N+1},x,b_{N}}(\tilde I+\lambda_{N}, 0, \lambda_{N}, \tilde I).$$
Recall that $W^b_{a,x,b}(I,J,K,L)$ vanishes unless $L\leq J$, so the expression above vanishes unless $\tilde I=0$; hence in the left-hand side of \eqref{exchangeproof1} the nonzero terms have $\tilde I=0$. Since $R_{a_{N+1},x,y,b_{N+1}}(0,0,\tilde K, 0)=\delta_{\tilde K,0}$, the non-zero summands in the left-hand side of \eqref{exchangeproof1} must also have $\tilde K=0$ and we obtain
\begin{multline}
\label{exchangeproof2}
\sum_{\lambda}(y/b_0)^{\lambda_1} \langle\nu|\,T^{a,N}_{K,0}(y\mid\calA,\calB)\,|\lambda\rangle\langle\lambda|\,T^{b,N}_{I,0}(x\mid\calA,\calB)\,|\mu\rangle\\
=\sum_{J, L}\sum_{\lambda}R_{a_{1},x,y,b_{1}}(I,J,K, L) \langle\nu|\,T^{b,N}_{L,0}(x\mid\calA,\tau\calB)\,|\lambda\rangle  (y/b_0)^{\mu_1} \langle\lambda|\,T^{a,N}_{J,0}(y\mid\tau\calA,\calB)|\mu\rangle.
\end{multline}

\paragraph{{\bf Step 2:}} Now we use \eqref{exchangeproof2} to obtain the claim. First, using the definitions of $\mathbb T^b_I$ and $\mathbb T^a_K$, we obtain  
\begin{multline*}
\sum_{\lambda}b_0^{-I} \langle\nu|\,\mathbb T^{a}_{K}(y\mid\calA,\calB)\,|\lambda\rangle\langle\lambda|\,\mathbb T^{b}_{I}(x\mid\calA,\calB)\,|\mu\rangle\\
=\sum_{J, L}y^{-I}R_{a_{1},x,y,b_{1}}(I,J,K, L) \langle\nu|\,\mathbb T^{b}_{L}(x\mid\calA,\tau\calB)\mathbb T^{a}_{J}(y\mid\tau\calA,\calB)\,|\mu\rangle,
\end{multline*}
where in the left-hand side we have used that $\lambda_1=\mu_1+I$ to simplify the term $(y/b_0)^{\lambda_1}$. Both sides are now elements of $\Bbbk[[q,x,y^{-1}]]$, \emph{cf.} Remark \ref{polynomiality} and part (2) of Proposition \ref{stability}. So we can multiply both sides by $x^I\frac{(a_1/x;q)_I}{(q;q)_I}$ and take the sum over $I,K$ such that $I+\mu_1=K+\nu_1$, obtaining an identity in $\Bbbk[[q,x,y^{-1}]]$:
\begin{multline}
\label{exchangeproof4}
\sum_{\lambda}\langle\nu|\,\mathbb B^*(y\mid\calA,\calB)\,|\lambda\rangle\langle\lambda|\,\mathbb B(x\mid\calA,\calB)\,|\mu\rangle\\
=\sum_{\substack{I,K,J, L:\\I-K= L-J=\nu_1-\mu_1}}(x/y)^{I}\frac{(a_1/x;q)_I}{(q;q)_I} R_{a_{1},x,y,b_{1}}(I,J,K, L) \langle\nu|\,\mathbb T^{b}_{L}(x\mid\calA,\tau\calB)\mathbb T^{a}_{J}(y\mid\tau\calA,\calB)\,|\mu\rangle.
\end{multline}

To prove the claim it is now enough to show that for fixed $J,L\geq 0$
\begin{equation}
\label{qGauss}
\sum_{\substack{I,K\geq 0:\\I-K= L- J}}(x/y)^{I}\frac{(a/x;q)_I}{(q;q)_I}R_{a,x,y,b}(I,J,K, L)=\frac{(a/y;q)_\infty(x/b;q)_\infty}{(a/b;q)_\infty(x/y;q)_\infty} (x/b)^{L}\frac{(a/x;q)_{L}}{(q;q)_{L}}.
\end{equation}
Indeed, if \eqref{qGauss} holds we can apply it to the right-hand side of \eqref{exchangeproof4} to get 
\begin{multline}
\label{exchangeproof5}
\sum_{\lambda}\langle\nu|\,\mathbb B^*(y\mid\calA,\calB)\,|\lambda\rangle\langle\lambda|\,\mathbb B(x\mid\calA,\calB)\,|\mu\rangle\\
=\frac{(a_1/y;q)_\infty(x/b_1;q)_\infty}{(a_1/b_1;q)_\infty(x/y;q)_\infty}  \sum_{\substack{J,L:\\ L-J=\nu_1-\mu_1}}(x/b_1)^{L}\frac{(a_1/x;q)_{L}}{(q;q)_{L}} \langle\nu|\,\mathbb T^{b}_{L}(x\mid\calA,\tau\calB)\mathbb T^{a}_{J}(y\mid\tau\calA,\calB)\,|\mu\rangle\\
=\frac{(a_1/y;q)_\infty(x/b_1;q)_\infty}{(a_1/b_1;q)_\infty(x/y;q)_\infty}\, \langle\nu|\,\mathbb B(x\mid\calA,\tau\calB)\mathbb B^*(y\mid\tau\calA,\calB)\,|\mu\rangle.
\end{multline}

\paragraph{{\bf Step 3:}} To finish the proof we need to establish \eqref{qGauss}. Since both sides are power series in $x,y^{-1},q$ with coefficients in $\mathbb C(a,b)$, it is enough to consider the case when $x=q^{S}a$ for $S\in \mathbb Z_{\geq 0}$: if for a formal series $f(q,x)\in\mathcal F[[q,x]]$ we have $f(q, q^S)=0$ for all $S\in\mathbb Z_{\geq 0}$ then $f(q,x)=0$, \emph{cf.} \cite[Lemma 3.2]{Ste88}. So, from now on we set $x=q^{S}a$.

Let $A\in \mathbb Z_{\geq 0}$ be a sufficiently large integer. From \eqref{WYB2} we have
\begin{equation}
\label{qGaussproof}
\tikzbase{1.3}{-3}{
	\draw[fused] 
	(2,0) node[below,scale=0.6] {\color{black} $J;y,b$} -- (1,1)  -- (-1,1) node[left,scale=0.6] {\color{black} $0;y;b$};
	\draw[fused] 
	(-1,0) node[left,scale=0.6] {\color{black} $S;a,q^Sa$}   -- (1,0)  -- (2,1) node[above,scale=0.6] {\color{black} $L;a,q^Sa$};
	\draw[fused] 
	(0,-1) node[below,scale=0.6] {\color{black} $A;a,b$} -- (0,2) node[above,scale=0.6] {\color{black} $A+S+J-L;a,b$};
	\node[above right] at (0,1) {\tiny{ $W^a$}};
	\node[above right] at (0,0) {\tiny{ $W^b$}};
	\node[right] at (1.5,0.5) {\tiny{$R$}};
}\quad
=\quad
\tikzbase{1.3}{-3}{
	\draw[fused]
	(2,-0.5) node[right,scale=0.6] {\color{black} $J;y,b$}  -- (0,-0.5)  -- (-1,0.5) node[above,scale=0.6] {\color{black} $0;y,b$};
	\draw[fused] 
	(-1,-0.5) node[below,scale=0.6] {\color{black} $S;a,q^Sa$} -- (0,0.5)  -- (2,0.5) node[right,scale=0.6] {\color{black} $L;a,q^Sa$};
	\draw[fused] 
	(1,-1.5) node[below,scale=0.6] {\color{black} $A;a,b$} -- (1,1.5) node[above,scale=0.6] {\color{black} $A+S+J-L;a,b$};
	\node[above right] at (1,-0.5) {\tiny{ $W^a$}};
	\node[above right] at (1,0.5) {\tiny{ $W^b$}};
	\node[right] at (-0.5,0) {\tiny{$R$}};
}
\end{equation}

Consider the right-hand side of \eqref{qGaussproof}. Note that for $\tilde{J},\tilde{L}\geq 0$ we have
\begin{multline*}
R_{a,q^{S}a,y,b}(S,\tilde{J},0,\tilde{L})=\delta_{S+\tilde{J}=\tilde{L}} \Phi(\tilde{L}, 0; q^{-S}, q^Sa/b)    \Phi(S, 0; a/y, q^{-S}y/a)\\
=\delta_{S+\tilde{J}=\tilde{L}} (q^Sa/b)^{\tilde L}\frac{(q^{-S};q)_{\tilde L}}{(a/b;q)_{\tilde L}} (q^{-S}y/a)^S\frac{(a/y;q)_S}{(q^{-S};q)_S},
\end{multline*}
where for the first equality we have used the second line of \eqref{defR}, noting that since $K=0$ the only nonzero summand has $P=0$. Since $\tilde{L}=\tilde{J}+S\geq S$ and $(q^{-\tilde{L}};q)_S$ vanishes unless $\tilde{L}\leq S$, we get
$$
R_{a,q^{S}a,y,b}(S,\tilde{J},0,\tilde{L})=\begin{cases}
(y/b)^S\frac{(a/y;q)_S}{(a/b;q)_S},\qquad &\text{if} \ \tilde{J}=0, \tilde{L}=S;\\
0, \qquad &\text{otherwise}.
\end{cases}
$$
Hence in the right-hand side \eqref{qGaussproof} the configuration around the vertex of type $R$ is fixed, and the partition function is equal to
\begin{multline*}
(y/b)^S\frac{(a/y;q)_S}{(a/b;q)_S}W^a_{a,y,b}(A,J,0,A+J) W^b_{a,q^Sa,b}(S,A+J,A+J+S-L,L)\\
=(y/b)^S\frac{(a/y;q)_S}{(a/b;q)_S} (y/b)^A\frac{(a/y;q)_A}{(a/b;q)_A} (q^Sa/b)^L\frac{(q^{-S};q)_{L}(q^Sa/b;q)_{A+J-L}}{(a/b;q)_{A+J}}\frac{(q;q)_{A+J}}{(q;q)_L(q;q;)_{A+J-L}}.
\end{multline*}
For the left-hand side of \eqref{qGaussproof} we can use \eqref{defWa},\eqref{defWb} to write it as
\begin{multline*}
\sum_{I,K} (q^Sa/b)^{I}\frac{(q^{-S};q)_I(q^Sa/b;q)_{A-I}}{(a/b;q)_A}\frac{(q;q)_A}{(q;q)_I(q;q)_{A-I}}(y/b)^{A+S-I}\frac{(a/y;q)_{A+S-I}}{(a/b;q)_{A+S-I}}R_{a,q^Sa,y,b}(I,J,K,L).
\end{multline*}
Note that in the summation above we can assume $I\leq S$, since $(q^{-S};q)_I$ vanishes otherwise. In particular, the number of terms is bounded by $S$ regardless of the value of $A$.

To finish the proof we need to consider the dependence on $A$. Namely, we rewrite \eqref{qGaussproof} in a way that both sides are rational in $q^A$: using the identity
$$
\frac{(u;q)_{A+X}}{(u;q)_{A+Y}}=(uq^{A+Y};q)_{X-Y}
$$
and the expressions above for the both sides we readily obtain
\begin{multline}
\label{qGauss3}
\sum_{\substack{I\leq S\\K=I+J-L}} (q^Sa/y)^{I}\frac{(q^{-S};q)_I}{(q;q)_I}\frac{(q^{A-I+1};q)_I}{(q^{A+J-L+1};q)_{L}}\frac{(a q^A/y;q)_{S-I}}{(a q^{A+J}/b;q)_{S-L}}R_{a,q^Sa,y,b}(I,J,K,L)\\
=\frac{(a/y;q)_S}{(a/b;q)_S} (q^Sa/b)^L\frac{(q^{-S};q)_{L}}{(q;q)_L}.
\end{multline}
Note that the both sides of \eqref{qGauss3} are rational functions in $q^A$, which are equal when $A$ is a sufficiently large integer. Hence \eqref{qGauss3} holds for any value of $q^A$, and in particular we can set $q^A=0$ getting
$$
\sum_{\substack{I\leq S\\K=I+J-L}} (q^Sa/y)^{I}\frac{(q^{-S};q)_I}{(q;q)_I}R_{a,q^Sa,y,b}(I,J,K,L)=\frac{(a/y;q)_S}{(a/b;q)_S} (q^Sa/b)^L\frac{(q^{-S};q)_{L}}{(q;q)_L},
$$
which is exactly \eqref{qGauss} when $x=q^Sa$.
\end{proof}

\begin{rem} One can check that \eqref{qGauss} from the proof above is equivalent to the $q$-Gauss identity
$$
\sum_{k\geq 0} \(\frac{c}{ab}\)^k\frac{(a;q)_k(b;q)_k}{(c;q)_k(q;q)_k}=\frac{(c/a;q)_\infty(c/b;q)_\infty}{(c;q)_\infty(c/(ab);q)_\infty}.
$$
\end{rem}

\begin{prop}[{\cite[Proposition 4.5]{BK21}}]\label{Ccommute}
The following relation holds
$$
\mathbb B(x_1\mid\calA,\calB)\mathbb B(x_2\mid\tau\calA,\calB)=\mathbb B(x_2\mid \calA,\calB)\mathbb B(x_1\mid \tau\calA,\calB)
$$
\end{prop}
\begin{proof}[Idea of the proof] This is exactly \cite[Proposition 4.5]{BK21}, and the proof is similar to Proposition \ref{BCexchange} above, so we provide only a brief sketch of it. First, similarly to steps $1$,$2$ of the proof of Proposition \ref{BCexchange}, we use \eqref{WYB1} to get a commutation relation
$$
\mathbb T_I(x_1\mid \calA, \calB) \mathbb T_J(x_2 \mid \tau\calA, \calB)=\sum_{K,L} W^b_{a_2,x_1,x_2}(I,J,K,L) \mathbb T_K(x_2\mid \calA, \calB) \mathbb T_L(x_1 \mid \tau\calA, \calB).
$$ 
Then the claim follows from multiplying both sides by $(x_1/b_0)^I(x_2/b_0)^J\frac{(a_1/x_1;q)_I(a_2/x_2;q)_J}{(q;q)_I(q;q)_J}$, taking the sum over $I,J$ and applying the relation 
$$
\sum_{I,J}(x_1/b_0)^I(x_2/b_0)^J\frac{(a_1/x_1;q)_I(a_2/x_2;q)_J}{(q;q)_I(q;q)_J} W^b_{a_2,x_1,x_2}(I,J,K,L) =(x_2/b_0)^K\frac{(a_1/x_2;q)_K}{(q;q)_K}(x_1/b_0)^L\frac{(a_2/x_1;q)_L}{(q;q)_L}
$$
where $K,L$ are fixed. The last relation can be proved by taking \eqref{WYB1} and setting $A_1=A_2=0, B_1=L, B_2=K$, while $B_3\to \infty$ and $A_3=B_1+B_2+B_3$.
\end{proof}

\section{Inhomogeneous spin $q$-Whittaker polynomials}
\label{sqW-sect}
In this section we describe inhomogeneous spin $q$-Whittaker polynomials, originally introduced in \cite{BK21}, and apply the results from previous sections to establish a new Cauchy-type identity and a new characterization theorem for these functions. We continue to use $\calA=(a_0, a_1,\dots)$ and $\calB=(b_0,b_1,\dots)$ to denote sequences of parameters and we continue to use the notation $\tau\calA, \overline\calA$ from Section \ref{exchange-sect}.

\subsection{Basic properties} For a pair of partitions $\lambda,\mu$ and a collection of variables $x_1, \dots, x_n$ the \emph{inhomogeneous spin $q$-Whittaker polynomial} $\F_{\lambda/\mu}(x_1, \dots, x_n\mid \calA,\calB)$ is defined in terms of the operators $\B(x\mid\calA,\calB)$ from Section \ref{exchange-sect} by
\begin{equation}
\label{defF}
\F_{\lambda/\mu}(x_1, \dots, x_n\mid \calA,\calB)=\langle \lambda|\, \B(x_1\mid\calA,\calB)\B(x_2\mid\tau\calA, \calB)\dots \B(x_n\mid \tau^{n-1}\calA,\calB)\, |\mu\rangle.
\end{equation}
When $\mu=\varnothing$ we write $\F_{\lambda}$ instead of $\F_{\lambda/\varnothing}$. 

The following properties of $\F_{\lambda/\mu}(x_1, \dots, x_n\mid \calA,\calB)$ were proved in \cite{BK21};\footnote{Our functions $\F_{\lambda/\mu}$ were denoted by $\F^{\sf s}_{\lambda/\mu}$ in \cite{BK21}, and our $a_r,b_r$ are equal to $\xi_rs_r$ and $\xi_r/s_r$ from \cite{BK21}.} for the sake of completeness we sketch their proofs:

\begin{prop}\label{1variable} We have the following expression for the single-variable function $\F_{\lambda/\mu}(x\mid \calA,\calB)$:
$$
\F_{\lambda/\mu}(x\mid \calA,\calB)=
\begin{cases}
x^{|\lambda|-|\mu|}\displaystyle\prod_{r\geq 1} b_{r-1}^{\mu_{r}-\lambda_{r}}\dfrac{(a_{r}/x;q)_{\lambda_{r}-\mu_{r}}(x/b_{r};q)_{\mu_r-\lambda_{r+1}}(q;q)_{\mu_r-\mu_{r+1}}}{(q;q)_{\lambda_{r}-\mu_{r}}(q;q)_{\mu_r-\lambda_{r+1}}(a_{r+1}/b_r;q)_{\mu_{r}-\mu_{r+1}}},\qquad &\text{if}\ \mu\prec\lambda;\\
0,\qquad &\text{otherwise}.
\end{cases}
$$
\end{prop}
\begin{proof}
The vanishing part follows from Proposition \ref{stability}, part (3). For the explicit expression we use \eqref{defWb}, \eqref{defTb} and the relation
$$
\langle\lambda|\, \B(x\mid \calA, \calB)\,|\mu\rangle=(x/b_0)^{\lambda_1-\mu_1}\frac{(a_1/x;q)_{\lambda_1-\mu_1}}{(q;q)_{\lambda_1-\mu_1}}\langle\lambda|\, \mathbb T^b_{\lambda_1-\mu_1}(x\mid \calA, \calB)\,|\mu\rangle.
$$
\end{proof}

\begin{prop}\label{branching} The following branching rule holds:
$$
\F_{\lambda/\nu}(x_1, x_2, \dots, x_n\mid\calA,\calB)=\sum_{\mu\prec \lambda}\F_{\lambda/\mu}(x_1\mid \calA, \calB)\F_{\mu/\nu}(x_2, \dots, x_n\mid \tau\calA, \calB),
$$
where the sum is over partitions $\mu$ such that $\lambda_{r}\geq\mu_{r}\geq\lambda_{r+1}$ for all $r$.
\end{prop}
\begin{proof}
Follows from \eqref{defF} and Proposition \ref{stability}, part (3).
\end{proof}

\begin{cor}
$\F_{\lambda/\mu}(x_1, \dots, x_n\mid\calA,\calB)=0$ unless $\mu\subset\lambda$ and $l(\lambda)\leq l(\mu)+n$.\qed
\end{cor}

\begin{prop} The functions $\F_{\lambda/\mu}(x_1, x_2, \dots, x_n\mid\calA,\calB)$ are symmetric polynomials in $x_1, \dots, x_n$.
\end{prop}
\begin{proof}
$\F_{\lambda/\mu}(x_1, x_2, \dots, x_n\mid\calA,\calB)$ is a polynomial since $\langle \lambda|\, \B(x\mid \calA,\calB)\,|\mu\rangle$ is a polynomial in $x$ for any $\lambda,\mu$, and $\F_{\lambda/\mu}(x_1, x_2, \dots, x_n\mid\calA,\calB)$ is a finite sum of products of such matrix elements thanks to Proposition \ref{branching}. The symmetry in $x_1, \dots, x_n$ follows from Proposition \ref{Ccommute}.
\end{proof}

\begin{rem}\label{graphicalF} Using \eqref{defTbgraph}, we can also obtain a graphical definition of the functions $\F_{\lambda/\mu}(x_1, \dots, x_n\mid\calA,\calB)$ in terms of vertex models:
$$
\F_{\lambda/\mu}(x_1, \dots, x_n\mid\calA,\calB)=\sum_{r_1, \dots, r_n\geq 0}\prod_{i=1}^n(x_i/b_0)^{r_i}\frac{(a_ix^{-1}_i;q)_{r_i}}{(q;q)_{r_i}} Z^{r_1, \dots, r_n}_{\lambda/\mu}
$$
where $Z^{r_1, \dots, r_n}_{\lambda/\mu}$ is the partition function from Figure \ref{partitionFunctionZ}.
\begin{figure}
\tikz{1}{
	\foreach\y in {4,...,5}{
		\draw[fused*] (1,\y) -- (8.1,\y);
		\node[right] at (8.1,\y) {$\dots$};
		\draw[fused] (8.7,\y) -- (10,\y);
	}
	\draw[fused*] (1,2) -- (8.1,2);
	\node[right] at (8.1,2) {$\dots$};
	\draw[fused] (8.7,2) -- (10,2);
	\foreach\x in {1,...,3}{
		\draw[fused] (2.5*\x-0.5,3.3) -- (2.5*\x-0.5,6);
		\draw[fused*] (2.5*\x-0.5,1) -- (2.5*\x-0.5,2.5);
		\node at (2.5*\x-0.5, 3) {$\vdots$};
	}
	\node[above right] at (2,5) {\tiny $W^b$};
	\node[above right] at (4.5,5) {\tiny $W^b$};
	
	\node[above right] at (2,4) {\tiny $W^b$};
	\node[above right] at (4.5,4) {\tiny $W^b$};
	
	\node[above right] at (2,2) {\tiny $W^b$};
	\node[above right] at (4.5,2) {\tiny $W^b$};
	%top labels
	\node[above] at (7,6) {$\cdots$};
	\node[above] at (4.5,6) {\tiny $\lambda_2-\lambda_3; a_2, b_2$};
	\node[above] at (2,6) {\tiny $\lambda_1-\lambda_2; a_1, b_1$};
	%left labels
	\node[left] at (1,2) {\tiny $r_n; a_n, x_n$};
%	\node[left] at (1,3) {$\vdots$};
	\node[left] at (1,4) {\tiny $r_2; a_2, x_2$};
	\node[left] at (1,5) {\tiny $r_1; a_1, x_1$};
	%bottom labels
	\node[below] at (7,1) {$\cdots$};
	\node[below] at (4.5,1) {\tiny $\mu_2-\mu_3; a_{n+2}, b_2$};
	\node[below] at (2,1) {\tiny $\mu_1-\mu_2; a_{n+1}, b_1$};
	%right labels
	\node[right] at (10,2) {$0$};
	\node[right] at (10,3) {$\vdots$};
	\node[right] at (10,4) {$0$};
	\node[right] at (10,5) {$0$};
}
\caption{\label{partitionFunctionZ} The partition function $Z^{r_1, \dots, r_n}_{\lambda/\mu}$ used to compute $\F_{\lambda/\mu}(x_1, \dots, x_n\mid\calA,\calB)$.}
\end{figure}
\end{rem}

\begin{rem} Spin $q$-Whittaker polynomials were originally introduced in \cite{BW17}, and later a different but related version was constructed in \cite{MP20}. The inhomogeneous spin $q$-Whittaker  polynomials $\F_{\lambda/\mu}(x_1, \dots, x_n\mid\calA,\calB)$ generalize both these versions: when $a_1=a_2=\dots=s$ and $b_1=b_2=\dots=s^{-1}$ the function $\F_{\lambda/\mu}(x_1, \dots, x_n\mid\calA,\calB)$ degenerates to the corresponding spin $q$-Whittaker polynomial from \cite{BW17}, while setting $b_1=b_2=\dots=s^{-1}$, $a_{n+l(\mu)}=0$ and $a_i=s$ for all $i\neq n+l(\mu)$ reduces $\F_{\lambda/\mu}(x_1, \dots, x_n\mid\calA,\calB)$ to the version from \cite{MP20}.
\end{rem}

\subsection{Cauchy identity}
To formulate the Cauchy identity we define dual functions $\F^*(y_1, \dots, y_m\mid \calA,\calB)$ by
$$
\F^*_{\lambda/\mu}(y_1, \dots, y_m\mid \calA,\calB)=\frac{\psi_\lambda(\calA,\calB)}{\psi_\mu(\tau^m\calA,\calB)}\F_{\lambda/\mu}(y_1, \dots, y_m\mid \calA, \calB),
$$
where $\psi_\lambda(\calA,\calB)$ is defined in Proposition \ref{BCdual}:
$$
\psi_\lambda(\calA,\calB)=\prod_{r\geq 0} (b_r/a_r)^{\lambda_{r+1}}\prod_{r\geq 1}\frac{(q;q)_{\lambda_r-\lambda_{r+1}}}{(a_r/b_r;q)_{\lambda_r-\lambda_{r+1}}}.
$$
Equivalently, the dual functions can be defined using the operators $\B^*(y\mid\calA,\calB)$ from Section \ref{exchange-sect}:

\begin{prop}\label{dual} We have
$$
\F^*_{\lambda/\mu}(y_1, \dots, y_m\mid \overline{\calB},\overline{\calA})=\langle \mu|\, \B^*(y^{-1}_m\mid\calA,\tau^{m-1}\calB)\dots\B^*(y^{-1}_2\mid\calA, \tau\calB)\B^*(y_1^{-1}\mid \calA,\calB)\, |\lambda\rangle.   
$$
\end{prop}
\begin{proof}
Follows immediately from Proposition \ref{BCdual}. Note that $\psi_\lambda(\calA,\calB)=\psi_\lambda(\overline{\calB},\overline{\calA})$.
\end{proof}

\begin{theo} \label{cauchyTheo}For any partitions $\mu,\nu$ the following equality of formal power series in $x_i,y_j,q$ holds:
\begin{multline*}
\sum_{\lambda}\F_{\lambda/\mu}(x_1, \dots, x_n\mid\calA,\calB)\F^*_{\lambda/\nu}(y_1, \dots, y_m\mid \overline{\calB}, \overline{\calA})\\
=\prod_{i=1}^n\prod_{j=1}^m\frac{(a_iy_j;q)_\infty(x_i/b_j;q)_\infty}{(x_iy_j;q)_\infty(a_i/b_j;q)_\infty}\sum_{\lambda}\F^*_{\mu/\lambda}(y_1, \dots, y_m\mid\overline\calB,\tau^n\overline\calA) \F_{\nu/\lambda}(x_1, \dots, x_n\mid \calA,\tau^m\calB).
\end{multline*}
In particular, setting $\mu=\nu=\varnothing$ yields
\begin{equation}
\label{Cauchy}
\sum_{\lambda}\F_{\lambda}(x_1, \dots, x_n\mid\calA,\calB)\F^*_{\lambda}(y_1, \dots, y_m\mid \overline{\calB}, \overline{\calA})=\prod_{i=1}^n\prod_{j=1}^m\frac{(a_iy_j;q)_\infty(x_i/b_j;q)_\infty}{(x_iy_j;q)_\infty(a_i/b_j;q)_\infty}.
\end{equation}
\end{theo}
\begin{proof} This follows from Proposition \ref{BCexchange} in the following way: set 
$$
\B^{(n)}(x_1,\dots,x_n\mid \calA,\calB):=\B(x_1\mid \calA,\calB)\B(x_2\mid\tau\calA,\calB)\dots\B(x_n\mid\tau^{n-1}\calA,\calB).
$$
Repeatedly using the exchange relation from Proposition \ref{BCexchange}, shifting $\calA$ on each step, we have
$$
\B^*(y\mid \calA,\calB)\B^{(n)}(x_1,\dots,x_n\mid \calA,\calB)=\prod_{i=1}^n\frac{(a_i/y;q)_\infty(x_i/b_1;q)_\infty}{(x_i/y;q)_\infty(a_i/b_1;q)_\infty}\,\B^{(n)}(x_1,\dots,x_n\mid \calA,\tau\calB)\B^*(y\mid \tau^{n}\calA,\calB).
$$
Iterating the identity above $m$ times, substituting $y=y_j^{-1}$ for $j=1, \dots, m$, we get 
\begin{multline*}
\B^*(y_m^{-1}\mid \calA,\tau^{m-1}\calB)\dots\B^*(y_1^{-1}\mid \calA,\calB)\B^{(n)}(x_1,\dots,x_n\mid \calA,\calB)\\
=\prod_{i=1}^n\prod_{j=1}^m\frac{(a_iy_j;q)_\infty(x_i/b_j;q)_\infty}{(x_iy_j;q)_\infty(a_i/b_j;q)_\infty}\,\B^{(n)}(x_1,\dots,x_n\mid \calA,\tau^m\calB)\B^*(y_m^{-1}\mid \tau^{n}\calA,\tau^{m-1}\calB)\dots \B^*(y_1^{-1}\mid \tau^n\calA,\calB).
\end{multline*}
Evaluating both sides of the last identity at $\langle \nu|\cdot|\mu\rangle$ yields the claim.
\end{proof}

\begin{rem} In the case $b_0=b_1=b_2=\dots=1$ Theorem \ref{cauchyTheo} was proved in \cite{BK21} using the dual skew Cauchy identity between spin $q$-Whittaker functions and spin Hall-Littlewood functions.
\end{rem}

\subsection{Vanishing and characterization properties} It turns out that the functions $\F_\lambda(x_1, \dots, x_n\mid\calA,\calB)$ satisfy vanishing and characterization properties similar to interpolation symmetric functions. We start with the vanishing property. Recall that $\Y^n$ denotes the set of partitions of length at most $n$. For a partition $\mu\in\Y^n$ set 
\begin{equation}
\label{xpoints}
\x^n_{\calA}(\mu):=(a_1q^{\mu_1-\mu_2},a_2q^{\mu_2-\mu_3},\dots,a_nq^{\mu_n}).
\end{equation}

\begin{prop} \label{vanishing}Let $\lambda,\mu\in\Y^n$ be partitions of length at most $n$. Then we have
$$
\F_\lambda(\x^n_{\calA}(\mu)\mid\calA, \calB)=\F_{\lambda}(a_1q^{\mu_1-\mu_2}, a_2q^{\mu_2-\mu_3},\dots, a_nq^{\mu_n}\mid \calA,\calB)=0,\qquad \text{unless}\ \lambda\subseteq\mu.
$$
Moreover, when $\lambda=\mu$, we have
$$
\F_{\lambda}(\x_{\calA}^{n}(\lambda)\mid \calA,\calB)=H_\lambda(\calA,\calB)\neq 0.
$$
where
$$
H_\lambda(\calA,\calB)=(-1)^{\lambda_1}q^{\frac{\lambda_1(\lambda_1-1)}{2}}\prod_{i=1}^n (a_i/b_{i-1})^{\lambda_{i}}\displaystyle\prod_{i,j= 1}^n\dfrac{(q^{\lambda_{i+1}-\lambda_i}a_{j+i}/a_{i};q)_{\lambda_{j+i}-\lambda_{j+i+1}}}{(a_{j+i}/b_j;q)_{\lambda_{j+i}-\lambda_{j+i+1}}}.
$$
\end{prop} 
\begin{proof}
The proof is by induction on $n$, with the case $n=0$ being trivial. 

Assume that we have proved the claim for $n-1$. Using the branching rule from Proposition \ref{branching} we get
$$
\F_\lambda(\x_{\calA}^n(\mu)\mid\calA,\calB)=\sum_{\nu\prec\lambda}\F_{\lambda/\nu}(a_1q^{\mu_1-\mu_2}\mid\calA,\calB)\F_\nu(\x_{\tau\calA}^{n-1}(\widetilde{\mu})\mid \tau\calA,\calB),
$$
where $\widetilde{\mu}=(\mu_2, \mu_3, \dots)$. Note that by induction hypothesis $\F_\nu(\x_{\tau\calA}^{n-1}(\widetilde{\mu})\mid \tau\calA,\calB)=0$ unless $\nu\subseteq\widetilde{\mu}$. But $\nu\prec \lambda$ implies $\widetilde{\lambda}:=(\lambda_2,\lambda_3, \dots)\subseteq\nu$, so if in the sum above the summand corresponding to $\nu$ does not vanish then $\widetilde{\lambda}\subseteq\nu\subseteq\widetilde{\mu}$.

At the same time, $\F_{\lambda/\nu}(a_1q^{\mu_1-\mu_2}\mid\calA,\calB)$ vanishes unless $\lambda_1-\nu_1\leq\mu_1-\mu_2$, because of the factor $(q^{\mu_{2}-\mu_1};q)_{\lambda_1-\nu_1}$ in
\begin{multline*}
\F_{\lambda/\nu}(a_1q^{\mu_1-\mu_2}\mid\calA,\calB)=\langle\lambda|\, \B(a_1q^{\mu_1-\mu_2}\mid\calA,\calB)\,|\mu\rangle\\
=\(\frac{a_1q^{\mu_1-\mu_2}}{b_0}\)^{\lambda_1-\nu_1}\frac{(q^{\mu_{2}-\mu_1};q)_{\lambda_1-\nu_1}}{(q;q)_{\lambda_1-\nu_1}}\langle\lambda|\, \mathbb T^b_{\lambda_1-\nu_1,0}(a_1q^{\mu_1-\mu_2}\mid\calA,\calB)\,|\mu\rangle.
\end{multline*}
Hence, if $\F_\lambda(\x_{\calA}^n(\mu)\mid\calA,\calB)\neq0$, then for some $\nu$ we have $\lambda_1-\nu_1\leq\mu_1-\mu_2$ and $\widetilde{\lambda}\subseteq\nu\subseteq\widetilde{\mu}$. But then
$$
\lambda_1\leq \mu_1-\mu_2+\nu_1\leq\mu_1,
$$
so $\lambda\subseteq \mu$ as desired.

To prove the second statement, note that if $\lambda=\mu$ then the only possible $\nu$ in the discussion above is $\nu=\widetilde{\lambda}=(\lambda_2,\lambda_3,\dots)$. Hence
$$
\F_\lambda(\x_{\calA}^n(\lambda)\mid\calA,\calB)=\F_{\lambda/\widetilde{\lambda}}(a_1q^{\lambda_1-\lambda_2}\mid\calA,\calB)\F_{\widetilde{\lambda}}(\x_{\tau\calA}^{n-1}(\widetilde{\lambda})\mid \tau\calA,\calB)
$$
and the second statement readily follows from Proposition \ref{1variable}.
\end{proof}

\begin{rem} We can use the graphical definition of $\F_\lambda(x_1, \dots, x_n\mid\calA,\calB)$ from Remark \ref{graphicalF} to sketch an alternative proof of Proposition \ref{vanishing}. Namely, we have
$$
\F_\lambda(\x^n_{\calA}(\mu)\mid\calA,\calB)=\sum_{r_1, \dots, r_n\geq 0}\prod_{i=1}^n(a_iq^{\mu_i-\mu_{i+1}}/b_0)^{r_i}\frac{(q^{\mu_{i+1}-\mu_i};q)_{r_i}}{(q;q)_{r_i}} Z^{r_1, \dots, r_n}_\lambda,
$$
where $Z^{r_1, \dots, r_n}_\lambda$ is the partition function from Figure \ref{partitionFunctionZ} with $x_i=a_iq^{\mu_i-\mu_{i+1}}$. Note that the factors $(q^{\mu_{i+1}-\mu_i};q)_{r_i}$ force $r_i\leq \mu_i-\mu_{i+1}$ in the sum above. On the other hand, since the weights $W^b(I,J,K,L)$ vanish unless $I+J=K+L$ and $I\leq K$, the partition function $Z^{r_1, \dots, r_n}_\lambda$ vanishes unless $r_1+\dots+r_n=\lambda_1$ and $r_1+\dots +r_i\leq \lambda_1-\lambda_{i+1}$ for any $i$. These restrictions on $r_i$ imply that non-zero terms in the sum above must satisfy $\lambda_{i}\leq r_{i}+\dots+r_n\leq \mu_{i}$, hence $\F_\lambda(\x^n_{\calA}(\mu)\mid\calA,\calB)$ vanishes unless $\lambda_{i}\leq\mu_i$.
\end{rem}

Now we can consider a characterization for $\F_\lambda(x_1, \dots,x_n\mid \calA,\calB)$ in terms of the vanishing property. To do so, we need the following notation. For an integer $n$ and a partition $\mu\in\Y^n$ define
$$
G_\mu^{n}(x_1, \dots, x_n\mid\calA,\calB)=\prod_{i=1}^n\prod_{r\geq 1}\frac{(x_i/b_r;q)_{\mu_r-\mu_{r+1}}}{(a_i/b_r;q)_{\mu_r-\mu_{r+1}}}.
$$
Note that $G_\mu^{n}$ are symmetric polynomials in $x_1, \dots, x_n$. Define a filtration $\mathcal G^n_0\subset \mathcal G^n_1\subset \dots\subset \Bbbk[x_1, \dots, x_n]^{S_n}$ of the algebra of symmetric polynomials in $x_1, \dots,x_n$ with coefficients in $\Bbbk=\mathbb Q(q, \calA,\calB)$ by
$$
\mathcal G^n_m:=\mathrm{span} \{G_\mu^{n}\}_{\substack{\mu\in\Y^n\\ |\mu|\leq m}}.
$$

\begin{theo} \label{charTheo} For each partition $\lambda\in\Y^n$ the spin $q$-Whittaker function $\F_\lambda(x_1, \dots, x_n\mid \calA,\calB)$ is uniquely characterized by the following properties:
\begin{enumerate}
\item $\F_\lambda(x_1, \dots, x_n\mid\calA,\calB)\in\mathcal G^{n}_{|\lambda|}$.
\item For any partition $\mu$ such that $|\mu|\leq |\lambda|$ and $\mu\neq\lambda$ we have $\F_\lambda(\x^n_{\calA}(\mu)\mid\calA,\calB)=0$.
\item $\F_\lambda(\x^n_{\calA}(\lambda)\mid\calA,\calB)=H_\lambda(\calA,\calB)$, where $H_\lambda(\calA,\calB)$ is defined in Proposition \ref{vanishing} above.
\end{enumerate}
Moreover, for each $m\in\mathbb Z_{\geq 0}$ both $\{\F_\lambda(x_1, \dots, x_n\mid\calA,\calB)\}_{\substack{\lambda\in\Y^n\\ |\lambda|\leq m}}$ and $\{G_\mu^{n}(x_1, \dots, x_n\mid\calA,\calB)\}_{\substack{\mu\in\Y^n\\ |\mu|\leq m}}$ are bases of $\mathcal G^n_m$.
\end{theo}
\begin{proof}
We start with the last statement. Let $\mu\in\Y^n$ and consider the Cauchy identity \eqref{Cauchy} with $m=n$ and $y_i=q^{\mu_i-\mu_{i+1}}/b_i$. We get
\begin{equation}
\label{specializedCauchy}
\sum_{\lambda}\F_{\lambda}(x_1, \dots, x_n\mid\calA,\calB)\F^*_{\lambda}(q^{\mu_1-\mu_2}/b_1, \dots, q^{\mu_n}/b_n \mid \overline{\calB}, \overline{\calA})=\prod_{i,j=1}^n\frac{(x_i/b_j;q)_{\mu_j-\mu_{j+1}}}{(a_i/b_j;q)_{\mu_j-\mu_{j+1}}}=G^n_{\mu}(x_1, \dots, x_n\mid\calA,\calB).
\end{equation}
Note that $(q^{\mu_1-\mu_2}/b_1, \dots, q^{\mu_n}/b_n)=\x^n_{\overline{\calB}}(\mu)$, so we can apply the vanishing property from Proposition \ref{vanishing} to $\F^*_{\lambda}(\x^n_{\overline \calB}(\mu)\mid \overline{\calB}, \overline{\calA})$, obtaining 
$$
\F^*_{\lambda}(\x^n_{\overline \calB}(\mu)\mid \overline{\calB}, \overline{\calA})=\psi_\lambda(\overline\calB,\overline\calA)\F_{\lambda}(\x^n_{\overline \calB}(\mu)\mid \overline{\calB}, \overline{\calA})=0,\qquad\text{unless}\ \lambda\subseteq\mu,
$$
$$
\F^*_{\mu}(\x^n_{\overline \calB}(\mu)\mid \overline{\calB}, \overline{\calA})=\psi_\mu(\overline\calB,\overline\calA)H_\mu(\overline\calB,\overline\calA)\neq0,
$$
where $\psi_\lambda(\calA,\calB)$ is defined in Proposition \ref{BCdual}. Hence we can rewrite \eqref{specializedCauchy} as
\begin{equation}
\label{GtoF}
G^n_{\mu}(x_1, \dots, x_n\mid\calA,\calB)=\sum_{\lambda\subseteq \mu}c_{\lambda\mu}\ \F_{\lambda}(x_1, \dots, x_n\mid\calA,\calB)
\end{equation}
for some coefficients $c_{\lambda\mu}$ such that $c_{\mu\mu}\neq 0$. Thus, the transition matrix expressing $G^n_\mu$ in terms of $\F_\lambda$ is upper-triangular with respect to the partial order of inclusion of Young diagrams. Moreover, this transition matrix has non-zero diagonal entries $c_{\mu\mu}$, so it is invertible and 
\begin{equation}
\label{FtoG}
\F_{\lambda}(x_1, \dots, x_n\mid\calA,\calB)=\sum_{\mu\subseteq \lambda}\overline{c}_{\mu\lambda}\ G^{n}_{\mu}(x_1, \dots, x_n\mid\calA,\calB)
\end{equation}
for some other coefficients $\overline{c}_{\mu\lambda}$.

Note that \eqref{GtoF},\eqref{FtoG} imply that 
$$
\mathrm{span}\{\F_\lambda(x_1, \dots, x_n\mid\calA,\calB)\}_{\substack{\lambda\in\Y^n\\ |\lambda|\leq m}}=\mathrm{span}\{G_\mu^{n}(x_1, \dots, x_n\mid\calA,\calB)\}_{\substack{\mu\in\Y^n\\ |\mu|\leq m}}=\mathcal G_m^n.
$$
On the other hand, the vanishing property implies that the functions $\F_\lambda$ are linearly independent: if we have
$$
\sum_{\lambda} \alpha_\lambda \F_\lambda(x_1, \dots, x_n\mid\calA,\calB)=0,
$$
with $\alpha_\lambda$ being nonzero for some $\lambda$, then we get a contradiction by plugging $(x_1, \dots, x_n)=\x^n_{\calA}(\mu)$ for a minimal $\mu$ such that $\alpha_\mu\neq 0$ and using the vanishing property from Proposition \ref{vanishing}. Hence $\{\F_\lambda(x_1, \dots, x_n\mid\calA,\calB)\}_{\substack{\lambda\in\Y^n\\ |\lambda|\leq m}}$ is a basis of $\mathcal G^n_m$, and consequently $\{G_\mu^{n}(x_1, \dots, x_n\mid\calA,\calB)\}_{\substack{\mu\in\Y^n\\ |\mu|\leq m}}$ is also a basis of $\mathcal G_m^n$ in view of \eqref{GtoF},\eqref{FtoG}.

Now we can prove the characterization property. We have already proved that $\F_\lambda(x_1, \dots, x_n\mid\calA,\calB)\in\mathcal G^n_m$, and the other properties follow from Proposition \ref{vanishing}, so we only need to prove that $\F_\lambda(x_1, \dots, x_n\mid\calA,\calB)$ is the unique function satisfying properties (1)-(3) above. Let $f(x_1,\dots, x_n)$ be a function which also satisfies properties (1)-(3), and consider the difference $g(x_1,\dots, x_n)=f(x_1, \dots, x_n)-\F_\lambda(x_1, \dots, x_n\mid\calA,\calB)$. We have $g(x_1, \dots, x_n)\in\mathcal G_m^n$ and $g(\x^n_\calA(\mu))=0$ for any $\mu$ such that $|\mu|\leq|\lambda|$. Since {$\{\F_\lambda(x_1,\dots, x_n\mid\calA,\calB)\}_{\substack{\lambda\in\Y^n\\|\lambda|\leq m}}$} is a basis of $\mathcal G_{|\lambda|}^n$, for some coefficients $\alpha_\nu\in \Bbbk$ we have
$$
g(x_1,\dots, x_n)=\sum_{\nu:l(\nu)\leq n, |\nu|\leq|\lambda|}\alpha_\nu \F_\nu(x_1, \dots, x_n\mid \calA,\calB).
$$
Assume that $g\neq 0$ and let $\mu$ be a minimal partition such that $\alpha_\mu\neq 0$. Then by Proposition \ref{vanishing}
$$
g(\x^n_{\calA}(\mu))=\alpha_\mu \F_\mu(\x^n_{\calA}(\mu)\mid \calA,\calB)\neq 0,
$$
which leads to contradiction. Hence, $g=\F_\lambda-f=0$, and the uniqueness follows.
\end{proof}

\subsection{Degenerations of $\F_{\lambda}$} In this section we introduce two degenerations of the functions $\mathbb F_{\lambda}$, the importance of which is explained in Section \ref{interpolationSect} where the interpolation symmetric functions are discussed. 

We start by considering $b_i\to\infty$ in the functions $\F_{\lambda/\mu}(x_1, \dots, x_n\mid\calA, \calB)$. Recall that $\F_{\lambda/\mu}(x_1, \dots, x_n\mid\calA, \calB)$ is a rational function in $b_i$, but Proposition \ref{1variable} implies that when we set $b_i^{-1}=0$ for all $i\geq 1$, the function just vanishes. So, to get a meaningful object we first need to renormalize the function: set
$$
\widetilde{\F}_{\lambda/\mu}(x_1, \dots, x_n\mid\calA,\calB)=\prod_{r=1}^nb_{r-1}^{\lambda_r-\mu_r}\ \F(x_1, \dots, x_n\mid\calA, \calB)
$$
and let $\widetilde{\F}_{\lambda/\mu}(x_1, \dots, x_n\mid\calA,\infty)$ denote the result of the substitution $b_0^{-1}=b_1^{-1}=\dots=0$. 

Almost all the properties of $\F_{\lambda/\mu}(x_1,\dots x_n\mid\calA,\calB)$ described earlier in this section can be readily modified to obtain properties of $\widetilde{\F}_{\lambda/\mu}(x_1, \dots, x_n\mid\calA,\infty)$. For instance, we can compute $\widetilde{\F}_{\lambda/\mu}(x_1, \dots, x_n\mid\calA,\infty)$ using degenerations of the explicit expression from Proposition \ref{1variable} and the branching rule from Proposition \ref{branching}: we have
\begin{equation}
\label{degbranching}
\widetilde\F_{\lambda/\nu}(x_1, x_2, \dots, x_n\mid\calA,\infty)=\sum_{\mu\prec \lambda}\widetilde\F_{\lambda/\mu}(x_1\mid \calA, \infty)\widetilde\F_{\mu/\nu}(x_2, \dots, x_n\mid \tau\calA, \infty),
\end{equation}
\begin{equation}
\label{deg1var}
\widetilde\F_{\lambda/\mu}(x\mid \calA,\infty)=
\begin{cases}
x^{|\lambda|-|\mu|}\displaystyle\prod_{r\geq 1} \dfrac{(a_{r}/x;q)_{\lambda_{r}-\mu_{r}}(q;q)_{\mu_r-\mu_{r+1}}}{(q;q)_{\lambda_{r}-\mu_{r}}(q;q)_{\mu_r-\lambda_{r+1}}},\qquad &\text{if}\ \mu\prec\lambda;\\
0,\qquad &\text{otherwise}.
\end{cases}
\end{equation}
In the same manner we can degenerate the vanishing property from Proposition \ref{vanishing}, since it does not depend on $\calB$ in a significant way:
\begin{prop} \label{degvanishing}Let $\lambda,\mu$ be partitions of length at most $n$. Then we have
$$
\widetilde\F_\lambda(\x^n_{\calA}(\mu)\mid\calA, \infty)=0,\qquad \text{unless}\ \lambda\subseteq\mu,
$$
Moreover, when $\lambda=\mu$, we have
$$
\widetilde\F_{\lambda}(\x_{\calA}^{n}(\lambda)\mid \calA,\infty)=(-1)^{\lambda_1}q^{\frac{\lambda_1^2-\lambda_1}{2}}\prod_{i=1}^n a_i^{\lambda_{i}}\displaystyle\prod_{i,j= 1}^n(q^{\lambda_{i+1}-\lambda_i}a_{j+i}/a_{i};q)_{\lambda_{j+i}-\lambda_{j+i+1}}.
$$
\qed
\end{prop}

The only property proved earlier which cannot be immediately degenerated to $\calB=\infty$ is Theorem \ref{charTheo}, since our definition of the filtration $\mathcal G^n_m$ makes little sense when $b_i^{-1}\equiv0$. To state the appropriate characterization, let $\Bbbk[x_1, \dots, x_n]^{S_n}_{\leq m}$ denote the space of degree $\leq m$ symmetric polynomials with coefficients in $\Bbbk=\C(\calA,q)$, and let $P_{\lambda/\mu}(x_1,\dots, x_n; q,0)$ denote the $q$-Whittaker function corresponding to $\lambda/\mu$, that is, the $t=0$ specialization of the Macdonald symmetric polynomial $P_{\lambda/\mu}(x_1, \dots ,x_n;q,t)$, \emph{cf.} \cite[Chapter VI]{Mac95}.

\begin{prop}\label{degchartheo} The function $\widetilde\F_\lambda(x_1, \dots, x_n\mid\calA, \infty)$ can be characterized in the following two equivalent ways:  it is the unique function satisfying
\begin{enumerate}[label=({\arabic*})]
\item $\widetilde \F_\lambda(x_1, \dots, x_n\mid\calA,\infty)\in \Bbbk[x_1, \dots, x_n]^{S_n}_{\leq |\lambda|}$,
\item $\widetilde\F_\lambda(\x^n_{\calA}(\mu)\mid\calA,\infty)=0$ for any partition $\mu$ such that $|\mu|\leq |\lambda|$ and $\mu\neq\lambda$,
\item $\widetilde\F_{\lambda}(\x_{\calA}^{n}(\lambda)\mid \calA,\infty)=(-1)^{\lambda_1}q^{\frac{\lambda_1^2-\lambda_1}{2}}\prod_{i=1}^n a_i^{\lambda_{i}}\displaystyle\prod_{i,j= 1}^n(q^{\lambda_{i+1}-\lambda_i}a_{j+i}/a_{i};q)_{\lambda_{j+i}-\lambda_{j+i+1}}$;
\end{enumerate}
and it is also the unique function satisfying
\begin{enumerate}[label=({\arabic*}')]
\item $\widetilde \F_\lambda(x_1, \dots, x_n\mid\calA,\infty)\in \Bbbk[x_1, \dots, x_n]^{S_n}_{\leq |\lambda|}$, identically to (1),
\item $\widetilde\F_\lambda(\x^n_{\calA}(\mu)\mid\calA,\infty)=0$ for any partition $\mu$ such that $|\mu|< |\lambda|$,
\item The top degree homogeneous component of $\widetilde\F_{\lambda}(x_1, \dots, x_n\mid \calA,\infty)$ has degree $|\lambda|$ and is equal to $P_{\lambda}(x_1, \dots, x_n;q,0)$.
\end{enumerate}
\end{prop}
\begin{proof} First we show that $\widetilde\F_\lambda(x_1, \dots, x_n\mid\calA, \infty)$ indeed satisfy the properties from the statement. Properties (2), (2') and (3) follow directly from Proposition \ref{degvanishing}. For the remaining properties (1), (1') and (3') note that by \eqref{deg1var} the one-variable function $\widetilde\F_{\lambda/\mu}(x\mid \calA,\infty)$ is a polynomial of degree $|\lambda|-|\mu|$, whose top term coincides with $P_{\lambda/\mu}(x; q,0)$, \emph{cf.} \cite[Chapter VI, (7.13')]{Mac95}. Then \eqref{degbranching} implies that $\widetilde\F_{\lambda}(x_1, x_2, \dots, x_n\mid\calA,\infty)$ has degree at most $|\lambda|$, and since the branching identical to \eqref{degbranching} holds for the functions $P_{\lambda/\mu}$, we deduce that the top-degree component of $\widetilde\F_{\lambda}(x_1, \dots, x_n\mid \calA,\infty)$ is indeed $P_{\lambda}(x_1, \dots, x_n; q,0)$.

To show uniqueness for both characterizations it is enough to prove that if $g\in \Bbbk[x_1, \dots, x_n]^{S_n}_{\leq m}$ and $g(\x_{\calA}^n(\mu))=0$ for all $|\mu|\leq m$, then $g=0$. Indeed, for the first chracterization fix $\lambda$ and assume that $f\in\Bbbk[x_1, \dots, x_n]^{S_n}$ satisfies the conditions (1)-(3). Then $g=f-\widetilde\F_{\lambda}(x_1, \dots, x_n\mid\calA,\infty)$ is a polynomial in $\Bbbk[x_1, \dots, x_n]^{S_n}_{\leq |\lambda|}$ such that $g(\x^n_{\calA}(\mu))=0$ for any $\mu\in\mathbb Y^n, |\mu|\leq|\lambda|$, since the conditions (2) and (3) fix the values of $f$ and $\widetilde\F_{\lambda}(x_1, \dots, x_n\mid\calA,\infty)$ at $\x^n_{\calA}(\mu)$ for $|\mu|\leq|\lambda|$. Similarly, if $f'$ is a function satisfying (1')-(3') set $g'=f'-\widetilde\F_{\lambda}(x_1, \dots, x_n\mid\calA,\infty)$. By (2') the functions $f'$ and $\widetilde\F_{\lambda}(x_1, \dots, x_n\mid\calA,\infty)$ both vanish at $\x^n_{\calA}(\mu)$ when $|\mu|\leq|\lambda|-1$, and by (3') the top degree homogeneous components of these functions coincide. Hence $g'\in \Bbbk[x_1, \dots, x_n]^{S_n}_{\leq |\lambda|-1}$ and $g'(\x_{\calA}^n(\mu))=0$ for $|\mu|\leq |\lambda|-1$. 

So, it is enough to prove that if $g\in \Bbbk[x_1, \dots, x_n]^{S_n}_{\leq m}$ and $g(\x_{\calA}^n(\mu))=0$ for all $\mu$ such that $|\mu|\leq m$, then $g=0$. This follows from Lemma \ref{tempPnondeglem} below, which for later convenience we state in a much greater generality.
\end{proof}

\begin{lem}\label{tempPnondeglem} Assume that for a function $\mho:\mathbb Y^n\to\Bbbk^n$ there exists a family of polynomials $F_\lambda(x_1, \dots, x_n\mid\mho)$ such that 
\begin{itemize}
\item $F_{\lambda}(x_1, \dots, x_n\mid\mho)\in \Bbbk[x_1, \dots, x_n]^{S_n}_{\leq |\lambda|}$;
\item $F_{\lambda}(\mho(\mu)\mid\mho)=0$ for any partitions $\lambda,\mu\in \mathbb Y^n$ such that $|\mu|\leq |\lambda|, \lambda\neq\mu$;
\item $F_{\lambda}(\mho(\lambda)\mid\mho)\neq0$.
\end{itemize}
Then the following holds
\begin{enumerate}
\item The functions $F_{\lambda}(x_1, \dots, x_n\mid \mho)$ are uniquely determined up to a scalar;
\item Let $m\in \Z_{\geq 0}$. If $f\in \Bbbk[x_1, \dots, x_n]^{S_n}_{\leq m}$ and $f(\mho(\mu))=0$ for all $\mu\in\mathbb Y^n$ such that $|\mu|\leq m$, then $f=0$.
\item For each $m\in \mathbb Z_{\geq 0}$ the functions $F_{\lambda}(x_1, \dots, x_n\mid \mho)$ with $|\lambda|\leq m$ form a basis of $\Bbbk[x_1, \dots, x_n]^{S_n}_{\leq m}$;
\item The degree of $F_{\lambda}(x_1, \dots, x_n\mid \mho)$ is $|\lambda|$.
\end{enumerate}
\end{lem}
\begin{proof}
Fix $m\in\mathbb Z_{\geq 0}$ and fix a choice of functions $F_{\lambda}(x_1, \dots, x_n\mid\mho)$ for $|\lambda|\leq m$. We claim that these functions are linearly independent: assume that
$$
\sum_{\lambda: |\lambda|\leq m} \alpha_{\lambda} F_{\lambda}(x_1, \dots, x_n\mid\mho)=0
$$ 
with $\alpha_\lambda\neq 0$ for some $\lambda$. Choose $\mu$ such that $\alpha_{\mu}\neq 0$ with minimal possible $|\mu|$. Then for all $\lambda\neq \mu$ either $|\lambda|<|\mu|$ and $\alpha_{\lambda}=0$, or $|\mu|\leq|\lambda|$ and $F_{\lambda}(\mho(\mu)\mid\mho)=0$ by the definition of $F_\lambda(x_1, \dots, x_n\mid\mho)$. Hence 
$$
\sum_{\lambda: |\lambda|\leq m} \alpha_{\lambda} F_{\lambda}(\mho(\mu)\mid\mho)=\alpha_\mu F_{\mu}(\mho(\mu)\mid\mho)\neq 0,
$$ 
leading to contradiction. 

Note that the $\Bbbk$-dimension of $\Bbbk[x_1, \dots, x_n]^{S_n}_{\leq m}$ is equal to the number of partitions $\lambda\in\mathbb Y^n$ such that $|\lambda|\leq m$, hence the functions $F_{\lambda}(x_1, \dots, x_n\mid\mho)$  with $|\lambda|\leq m$ form a basis of $\Bbbk[x_1, \dots, x_n]^{S_n}_{\leq m}$, proving (3). In particular, $F_\lambda(x_1, \dots, x_n\mid\mho)$ with $|\lambda|\leq m-1$ form a basis of $\Bbbk[x_1, \dots, x_n]^{S_n}_{\leq m-1}$, so  if $|\lambda|=m$ then the degree of $F_\lambda(x_1, \dots, x_n\mid\mho)$ is $m$, proving (4). Finally, (2) implies (1) by considering the difference between two candidates for $F_\lambda(x_1, \dots, x_n\mid\mho)$, so we only need to prove the former. Let $f\in \Bbbk[x_1, \dots, x_n]^{S_n}_{\leq m}$ be such that $f(\mho(\mu))=0$ for all $\mu$ satisfying $|\mu|\leq m$, and assume that $f\neq 0$. Consider the expansion
$$
f=\sum_{\lambda:|\lambda|\leq m} \alpha_\lambda F_\lambda(x_1, \dots, x_n\mid\mho).
$$
and choose $\mu$ such that $\alpha_\mu\neq 0$ with minimal $|\mu|$. Then, in the same way as in the first part of the proof, 
$$
f(\mho(\mu))=\sum_{\lambda:|\mu|\leq|\lambda|\leq m} \alpha_\lambda F_\lambda(\mho(\mu)\mid\mho)=\alpha_\mu F_\mu(\mho(\mu)\mid\mho)\neq 0,
$$
leading to contradiction.
\end{proof}

Another degeneration of $\F_\lambda$ is obtained by considering the regime $q,x_i,a_i\to 1$ in $\widetilde{\F}_\lambda(x_1, \dots, x_n\mid\calA, \infty)$. More precisely, let $\ve, d,r_1, \dots, r_n$ be parameters such that $d\neq 0$, and let $\mathcal C=(c_0, c_1, \dots)$ be an infinite sequence of real parameters.  Define $\F^{el}_{\lambda/\mu}(r_1, \dots, r_n\mid \calC, \infty)$ as the limit
$$
\F^{el}_{\lambda/\mu}(r_1, \dots, r_n\mid \calC, \infty)=(-d)^{\lambda_1-\mu_1} \lim_{\ve\to 0}\ve^{|\mu|-\mu_1-|\lambda|+\lambda_1}\widetilde \F_{\lambda/\mu}(x_1, \dots, x_n\mid\calA, \infty)
$$ 
in the following regime:
\begin{equation}
\label{limitregime}
x_i=e^{\ve r_i}, \qquad a_i=e^{\ve c_i}, \qquad q=e^{\ve d}.
\end{equation}
From \eqref{deg1var} we get the following expression for the one-variable function
\begin{equation}
\label{degdeg1var}
\F^{el}_{\lambda/\mu}(r\mid \calC,\infty)=
\begin{cases}
\displaystyle\prod_{i\geq 1}\frac{(\mu_i-\mu_{i+1})!\prod_{j=0}^{\lambda_i-\mu_i-1}(r-c_i-jd)}{(\lambda_i-\mu_{i})!(\mu_i-\lambda_{i+1})!},\qquad &\text{if}\ \mu\prec\lambda;\\
0,\qquad &\text{otherwise},
\end{cases}
\end{equation}
where we have used the following relation
$$
\lim_{\ve\to0}\ve^{-k}(e^{\ve u}; e^{\ve d})_k=(-1)^k\prod_{j=0}^{k-1}(u+jd).
$$
From \eqref{degdeg1var} we see that $\F^{el}_{\lambda/\mu}(r\mid \calC,\infty)$ is a polynomial in $r$ of degree $|\lambda|-|\mu|$, with coefficients depending polynomially on $d$ and $c_i$. Moreover, the top homogeneous degree term is $P_{\lambda/\mu}(r_1, \dots, r_n; 1,0)$, that is, it is a $q=1, t=0$ specialization of the Macdonald polynomial.  The branching rule \eqref{degbranching} implies that ${\F}^{el}_{\lambda}(r_1, \dots, r_n\mid \calC, \infty)$ is a symmetric polynomial in $r_1, \dots, r_n$ of degree $|\lambda|$ with coefficients in $\Bbbk^{el}=\mathbb Q(d,\calC)$, and the top homogeneous degree of ${\F}^{el}_{\lambda}(r_1, \dots, r_n\mid \calC, \infty)$ is $P_\lambda(r_1, \dots, r_n; 1,0)$. From \cite[VI.4]{Mac95} it is known that $P_\lambda(r_1, \dots, r_n; 1,0)=e_{\lambda'}(r_1, \dots, r_n)$ is the elementary symmetric polynomial, where $\lambda'$ is the partition conjugate to $\lambda$, and the elementary symmetric polynomials are defined by
$$
e_{\mu}(x_1, \dots, x_n)=\prod_{i\geq 1} e_{\mu_i}(x_1, \dots, x_n),
$$
$$
e_k(x_1, \dots, x_n)=\sum_{1\leq i_1<i_2<\dots<i_k\leq n} x_{i_1}x_{i_2}\dots x_{i_n}.
$$

We have the following analogues of the vanishing and characterization properties, obtained as limits of Proposition \ref{degvanishing} and Proposition \ref{degchartheo}. For a partition $\mu\in\mathbb Y^n$ set 
\begin{equation}
\label{rpoints}
\mathbf r^n_{\calC}(\mu)=(c_1+(\mu_1-\mu_2)d, c_2+(\mu_2-\mu_3)d, \dots, c_n+\mu_nd).
\end{equation}
\begin{prop} \label{degdegvanishing}Let $\lambda,\mu$ be partitions of length at most $n$. Then we have
$$
\F^{el}_\lambda({\mathbf r}^n_{\calC}(\mu)\mid \calC, \infty)=0,\qquad \text{unless}\ \lambda\subseteq\mu.
$$
Moreover, when $\lambda=\mu$, we have
$$
\F^{el}_\lambda({\mathbf r}_{\calC}^{n}(\lambda)\mid \calC,\infty)=(-1)^{|\lambda|}(-d)^{\lambda_1}\displaystyle\prod_{i,j= 1}^n\prod_{k=1}^{\lambda_{i+j}-\lambda_{i+j+1}}(c_{j+i}-c_i+d(\lambda_{i+1}-\lambda_i+k-1)).
$$
\qed
\end{prop} 
\begin{proof}
Note that in the regime \eqref{limitregime}
$$
\F^{el}_\lambda({\mathbf r}^n_{\calC}(\mu)\mid \calC, \infty)=(-d)^{\lambda_1} \lim_{\ve\to 0}\ve^{-|\lambda|+\lambda_1}\widetilde{\F}_\lambda(\x^n_{\calA}(\mu)\mid\mathcal A,\infty),
$$
so we obtain the claim as the limit of Proposition \ref{degvanishing}.
\end{proof}

\begin{prop}\label{degdegchartheo} The function $\F^{el}_\lambda(r_1, \dots, r_n\mid \calC, \infty)$ is the unique function satisfying
\begin{enumerate}[label=({\arabic*})]
\item $\F^{el}_\lambda(r_1, \dots, r_n\mid \calC, \infty)\in \Bbbk[r_1, \dots, r_n]^{S_n}_{\leq |\lambda|}$,
\item $\F^{el}_\lambda(\mathbf r^n_{\calC}(\mu)\mid \calC,\infty)=0$ for any partition $\mu$ such that $|\mu|\leq |\lambda|$ and $\mu\neq\lambda$,
\item $\F^{el}_{\lambda}(\mathbf r_{\calC}^{n}(\lambda)\mid \calC,\infty)=(-1)^{|\lambda|}(-d)^{\lambda_1}\displaystyle\prod_{i,j= 1}^n\prod_{k=1}^{\lambda_{i+j}-\lambda_{i+j+1}}(c_{j+i}-c_i+d(\lambda_{i+1}-\lambda_i+k-1))$;
\end{enumerate}
and it is also the unique function satisfying
\begin{enumerate}[label=({\arabic*}')]
\item $\F^{el}_\lambda(r_1, \dots, r_n\mid\calC,\infty)\in \Bbbk[r_1, \dots, r_n]^{S_n}_{\leq |\lambda|}$, identically to (1),
\item $\F^{el}_\lambda(\mathbf r^n_{\calC}(\mu)\mid\calC,\infty)=0$ for any partition $\mu$ such that $|\mu|< |\lambda|$,
\item The top degree homogeneous component of $\F^{el}_{\lambda}(r_1, \dots, r_n\mid \calC,\infty)$ has degree $|\lambda|$ and is equal to $e_{\lambda'}(r_1, \dots, r_n)$.
\end{enumerate}
\end{prop}
\begin{proof}
By the discussion above and Proposition \ref{degdegvanishing} $\F^{el}_\lambda(r_1, \dots, r_n\mid \calC, \infty)$ satisfies (1)-(3) and (1')-(3'). The uniqueness follows in the same way as in Theorem \ref{degchartheo}, by applying Lemma \ref{tempPnondeglem} with $\mho(\mu)=\mathbf r^n_{\calC}(\mu)$.
\end{proof}

In view of Proposition \ref{degdegchartheo}, we call $\F^{el}_\lambda(r_1, \dots, r_n\mid \calC, \infty)$ interpolation elementary polynomials.

\begin{rem}\label{specremark}
For the later use we note that instead of considering the fields $\Bbbk=\mathbb Q(q,\calA),\Bbbk^{cl}=\mathbb Q(d,\calC)$ we can consider $q, a_i, d, c_i$ above as elements of an arbitrary field $\Bbbk$. The induced functions $\widetilde\F_\lambda(x_1, \dots, x_n\mid\calA, \infty)\in\Bbbk[x_1, \dots, x_n]^{S_n}$ are well-defined as long as $q$ is not a root of unity, and our proofs of the vanishing property and the characterization for $\widetilde\F_\lambda(x_1, \dots, x_n\mid\calA, \infty)$ hold as long as $\widetilde\F_\lambda(\x^n_{\calA}(\lambda)\mid\calA, \infty)\neq 0$, which is equivalent to $a_i/a_j\notin q^{\Z_{\geq 0}}$  for any $i,j$ such that $i\neq j$. Similarly, the functions $\F^{el}_\lambda(r_1, \dots, r_n\mid \calC, \infty)\in\Bbbk[r_1, \dots, r_n]^{S_n}$ are well-defined as long as the characteristic of $\Bbbk$ is $0$, while for the characterization property we also need $c_i-c_j\notin d\Z$ for any $i\neq j$.
\end{rem}

\section{Classification of interpolation symmetric polynomials}\label{interpolationSect}

Our interest in the vanishing property from Proposition \ref{degvanishing} and the characterization of Proposition \ref{degchartheo} comes from the fact that the symmetric functions with similar properties were actively studied earlier. More precisely, there are roughly three classes of known symmetric polynomials with similar vanishing and characterization properties: factorial monomial polynomials, factorial Schur polynomials and interpolation Macdonald functions. In \cite{Ok97} these three classes and their variations were distinguished as the only solutions to a certain interpolation problem. In this section we describe how our new functions fit into this classification, expanding it.

\subsection{Interpolation problem}

Below we state the general interpolation problem from \cite{Ok97}. Let $\Bbbk$ be a field; for simplicity we assume that $\mathrm{char}\, \Bbbk=0$. As before, we use $\mathbb Y^n$ to denote the set of partitions of length at most $n$ and $\Bbbk[x_1, \dots, x_n]^{S_n}_{\leq m}$ to denote the space of symmetric polynomials in $n$ variables, of degree $\leq m$ and with coefficients in $\Bbbk$.  

An \emph{$n$-grid} is a map $\mho:\mathbb Y^n\to \Bbbk^{n}$. An $n$-grid $\mho$ is \emph{non-degenerate} if for every $\lambda\in\mathbb Y^n$ there exists a symmetric polynomial $F_\lambda(x_1, \dots, x_n\mid\mho)$ satisfying:
\begin{enumerate}
\item $F_{\lambda}(x_1, \dots, x_n\mid\mho)\in \Bbbk[x_1, \dots, x_n]^{S_n}_{\leq |\lambda|}$;
\item $F_{\lambda}(\mho(\mu)\mid\mho)=0$ for every $\mu\in \mathbb Y^n$ such that $|\mu|\leq |\lambda|, \lambda\neq\mu$;
\item $F_{\lambda}(\mho(\lambda)\mid\mho)\neq0$.
\end{enumerate}
An $n$-grid $\mho$ is called \emph{perfect} if it is non-degenerate and the polynomials $F_\lambda$ additionally satisfy the following vanishing property:
\begin{equation}
\label{vanishing-grid}
F_\lambda(\mho(\mu)\mid\mho)=0,\qquad \text{unless}\ \lambda\subseteq\mu.
\end{equation}

The main result of \cite{Ok97} is the following classification of all perfect grids of a certain form.

\begin{theo}[{\cite{Ok97}}] Assume that $\mho$ is an $n$-grid of the form
$$
\mho(\lambda)=(f_1(\lambda_1), \dots, f_n(\lambda_n)).
$$
Then $\mho$ is perfect if and only if one of the following cases holds:
\begin{enumerate}
\item[\emph{($\bf E_1$)}] $f_i(j)=c_j$ where $c_0,c_1, \dots$ are pairwise distinct elements of $\Bbbk$;

\item[\emph{($\bf E_2$)}] $f_i(j)=c_{j-i}$ where $\dots, c_{-1},c_0,c_1,\dots$ are pairwise distinct elements of $\Bbbk$;

\item[\emph{($\bf I$)}] $f_i(j)=a+bq^jt^i+cq^{-j}t^i$ where $a,b,c,q,t$ are elements of a field extension of $\Bbbk$;

\item[\emph{($\bf II$)}] $f_i(j)=\alpha+\beta j+\beta' i+\gamma(\beta j+\beta' i)^2$, where $\alpha,\beta,\beta',\gamma\in\Bbbk$;

\item[\emph{($\bf III$)}] $f_i(j)=\alpha+\epsilon^j{\epsilon'}^i(\alpha'+\beta j+\beta'i)$ where $\epsilon,\epsilon'\in\{\pm 1\}$ and $\alpha,\alpha',\beta,\beta'\in\Bbbk$;

\item[\emph{($\bf IV$)}]This case only exists when $n=2$, then $f_1(j)=\alpha+\beta q^j, f_2(j)=\alpha+\beta' q^{-j}$ where $\alpha,\beta,\beta',q\in\Bbbk$.
\end{enumerate}
All cases above should additionally satisfy $f_i(j)\neq f_{i'}(j')$ for all integers $i\leq i', j>j'$.
\end{theo}

The polynomials $F_\lambda(x_1, \dots, x_n\mid\mho)$ corresponding to the cases $\bf  E_1, E_2, I$ above are respectively factorial monomial polynomials, factorial Schur polynomials and interpolation Macdonald functions. The polynomial $F_\lambda$ has degree $|\lambda|$, and in these three cases the top-degree homogeneous component of $F_\lambda$ is respectively a monomial, Schur or Macdonald symmetric polynomial. Moreover, the functions $F_\lambda$ for case $\bf II$ above include interpolation functions related to Jack polynomials, see \cite{KS96}, \cite{OO96}.

Propositions \ref{degvanishing}-\ref{degdegchartheo} for the functions $\widetilde{\F}_\lambda(x_1, \dots, x_n\mid\calA,\infty)$ and $\F^{el}_{\lambda}(r_1, \dots, r_n\mid \calC,\infty)$ indicate that there is another class of solutions to this interpolation problem, where $n$-grids have an alternative form $\mho(\mu)=(f_1(\mu_1-\mu_2),\dots, f_n(\mu_n))$, \emph{cf.} the expressions for $\x^n_{\calA}(\mu)$ and $\mathbf r^n_{\calC}(\mu)$ from \eqref{xpoints},\eqref{rpoints}. In this section we show that the function $\widetilde{\F}, {\F}^{el}$ actually lead to all perfect grids of this alternative form, at least when $n\geq 3$.

\begin{theo}\label{perfectTheorem} Assume that $n\geq 3$ and $\mho$ is an $n$-grid of the form
$$
\mho(\lambda)=(f_1(\lambda_1-\lambda_2), f_2(\lambda_2-\lambda_3), \dots, f_n(\lambda_n)).
$$
Then $\mho$ is perfect if and only if the functions $f_i$ have one of the following two forms:
\begin{enumerate}
\item $f_i(k)=c+a_i q^k$ for constants $c, q, a_1, \dots, a_n\in \Bbbk$ such that $q$ is not a root of unity, $q\neq 0$ and $a_i/a_j\neq q^k$ for any $i\neq j, k\in\mathbb Z$. In this case the functions $F_\lambda(x_1, \dots, x_n\mid\mho)$ are proportional to $\widetilde\F_\lambda(x_1-c, x_2-c, \dots, x_n-c\mid\calA,\infty)$, where $a_i$ are identified with elements of $\calA$.

\item $f_i(k)=c_i+kd$ for constants $d, c_1, c_2, \dots, c_n\in \Bbbk$ such that $d\neq 0$ and $c_i-c_j\neq kd$ for any $i\neq j, k\in\mathbb Z$. In this case the functions $F_\lambda(x_1, \dots, x_n\mid\mho)$ are proportional to $\F^{el}_\lambda(x_1,\dots, x_n\mid \calC,\infty)$, where $c_i$ are identified with elements of $\calC$.
\end{enumerate}
\end{theo}
The remainder of this section is devoted to the proof of Theorem \ref{perfectTheorem}.

\begin{rem}
Following the existing terminology, the characterization property from Proposition \ref{degchartheo} allows us to call the functions $\widetilde\F$ interpolation $q$-Whittaker polynomials. Note that, while setting $t=0$ reduces Macdonald polynomials to $q$-Whittaker polynomials, setting $t=0$ in interpolation Macdonald polynomials does not result in any interpolation polynomials, because the characterization property for  interpolation Macdonald polynomials does not survive in any form after setting $t=0$. 
\end{rem}

\begin{rem}
Note that Theorem \ref{perfectTheorem} does not cover perfect $2$-grids of the form $\mho(\mu)=(f_1(\mu_1-\mu_2), f_2(\mu_2))$. While both types of grids listed in Theorem \ref{perfectTheorem} are well-defined and perfect when $n=2$, numerical simulations suggest that there exist more general perfect grids when $n=2$. However, our proof of Theorem \ref{perfectTheorem} does not cover the $n=2$ case.
\end{rem}

\subsection{General properties of $n$-grids} We start with adapting some arguments from \cite{Ok97} to our setting. From now on we always assume that $n$-grids $\mho$ have the form
$$
\mho(\lambda)=(\mho(1; \lambda_1-\lambda_2), \mho(2; \lambda_2-\lambda_3), \dots, \mho(n;\lambda_n)),
$$
where $\mho(1;\cdot), \dots, \mho(n;\cdot)$ are functions $\Z_{\geq 0}\to\Bbbk$.

For later convenience, we below restate Lemma \ref{tempPnondeglem} in terms of non-degenerate grids:
\begin{lem} \label{basicnondegLemma}Let $\mho$ be a non-degenerate $n$-grid.
\begin{enumerate}
\item The functions $F_{\lambda}(x_1, \dots, x_n\mid \mho)$ are uniquely determined up to a scalar;
\item Let $m\in \Z_{\geq 0}$. If $f\in \Bbbk[x_1, \dots, x_n]^{S_n}_{\leq m}$ and $f(\mho(\mu))=0$ for all $\mu\in\mathbb Y^n$ such that $|\mu|\leq m$, then $f=0$.
\item For each $m\in \mathbb Z_{\geq 0}$ the functions $F_{\lambda}(x_1, \dots, x_n\mid \mho)$ with $|\lambda|\leq m$ form a basis of $\Bbbk[x_1, \dots, x_n]^{S_n}_{\leq m}$;
\item The degree of $F_{\lambda}(x_1, \dots, x_n\mid \mho)$ is $|\lambda|$.
\end{enumerate}\qed
\end{lem}
In particular, given a non-degenerate $n$-grid $\mho$ the functions $F_\lambda(x_1, \dots, x_n\mid\mho )$ are well-defined up to scalars. We will fix a convenient normalization later.

The following operations on grids will be useful. Let $\mho$ be an $n$-grid.
\begin{itemize}
\item For $m\in\mathbb Z_{\geq 1}$ such that $m\leq n$ define an $m$-grid $\mho_m$ by
$$
\mho_m(\lambda)=(\mho(1; \lambda_1-\lambda_2), \mho(2; \lambda_2-\lambda_3), \dots, \mho(m;\lambda_m)), \qquad \lambda\in\mathbb Y^m.
$$
That is, $\mho_m(i; j)=\mho(i;j)$ for $i=1, \dots, m$. 
\item For $l\in\mathbb Z_{\geq 0}$ such that $l<n$ define an $(n-l)$-grid $\prescript{}{l}{\mho}$ by
$$
\prescript{}{l}{\mho}(\lambda)=(\mho(l+1; \lambda_1-\lambda_2), \mho(l+2; \lambda_2-\lambda_3), \dots, \mho(n;\lambda_{n-l})), \qquad \lambda\in\mathbb Y^{n-l}.
$$
That is, $\prescript{}{l}{\mho}(i; j)=\mho(i+l;j)$ for $i=1, \dots, n-l$. 
\item For $k\in\mathbb Z$ define an $n$-grid $\mho^k$ by
$$
\mho^k(\lambda)=\mho(\lambda+k^n)=(\mho(1; \lambda_1-\lambda_2), \dots, \mho(n-1;\lambda_{n-1}-\lambda_n), \mho(n;\lambda_n+k)),
$$
where $\lambda+k^n$ denotes the partition with parts $\lambda_i+k$. In other words, $\mho^k(i; j)=\mho(i;j)$ for $i=1, \dots, n-1$ and $\mho^k(n,j)=\mho(n;j+k)$.
\end{itemize}

Our first goal is to show that the three operations above preserve perfect grids. We start with $\mho_m$.

\begin{prop}\label{restrictionGrid} Assume that $\mho$ is a non-degenerate $n$-grid. Then for any $m\leq n$ the $m$-grid $\mho_m$ is non-degenerate with
\begin{equation}
\label{restrictionGrideq}
F_\lambda(x_1, \dots, x_m\mid\mho_m)=F_\lambda(x_1, \dots, x_m, \mho(m+1, 0), \dots, \mho(n,0)\mid\mho)
\end{equation}
for each $\lambda\in\mathbb Y^{m}$. Moreover, if $\mho$ is perfect then $\mho_m$ is perfect as well.
\end{prop}
\begin{proof}
For the first statement it is enough to show that the functions
$$
f_\lambda(x_1, \dots, x_m):=F_\lambda(x_1, \dots, x_m, \mho(m+1, 0), \dots, \mho(n,0)\mid\mho), \qquad \lambda\in\mathbb Y^m,
$$
satisfy the defining properties of $F_\lambda(x_1, \dots, x_m\mid\mho_m)$, which readily follows from noticing that for $\mu\in\mathbb Y^m$
$$
f_\lambda(\mho_m(\mu))=F_\lambda(\mho(\mu)\mid\mho).
$$
The last statement follows immediately from the vanishing property for $\mho$ and the identity above.
\end{proof}

\begin{prop} \label{uppershiftFirstStep}For any non-degenerate $n$-grid $\mho$ the following statements hold:
\begin{enumerate}
\item $\mho(i;j)\neq\mho(n;0)$ for any integer pair $(i,j)\neq (n,0)$;
\item $\mho^1$ is non-degenerate and $F_\lambda$ can be chosen so that
\begin{equation}
\label{uppershiftPexpression}
F_\lambda(x_1, \dots, x_n\mid\mho^1)\prod_{i=1}^n(x_i-\mho(n;0))=F_{\lambda+1^n}(x_1, \dots, x_n\mid\mho);
\end{equation}
\item If $\mho$ is perfect, then $\mho^1$ is perfect.
\end{enumerate}
\end{prop}
\begin{proof} Let $\lambda\in\mathbb Y^n$ and consider
$$
f_\lambda(x_1, \dots, x_{n-1}):=F_{\lambda+1^n}(x_1, \dots, x_{n-1}, \mho(n;0)\mid\mho).
$$
Note that $f_\lambda(x_1, \dots, x_{n-1})$ has degree at most $|\lambda|+n$, and for any $\mu\in\mathbb Y^{n-1}$ such that $|\mu|\leq |\lambda|+n$ we have 
$$
f_{\lambda}(\mho_{n-1}(\mu))=F_{\lambda+1^n}(\mho(1;\mu_1-\mu_2), \dots, \mho(n-1;\mu_{n-1}), \mho(n;0)\mid\mho)=F_{\lambda+1^n}(\mho(\mu)\mid\mho)=0,
$$
since $\mu\neq\lambda+1^n$. Hence, by Lemma \ref{basicnondegLemma} we have $f_\lambda=0$. In other words, $\restr{F_{\lambda+1^n}(x_1, \dots, x_n\mid\mho)}{x_n=\mho(n;0)}=0$, so $x_n-\mho(n;0)$ divides $F_{\lambda+1^n}(x_1, \dots, x_n\mid\mho)$. By symmetry, $x_i-\mho(n;0)$ for $i=1, \dots, n$ divide $F_{\lambda+1^n}(x_1, \dots, x_n\mid\mho)$ and so
\begin{equation}
\label{proofOfGridShift}
F_{\lambda+1^n}(x_1, \dots, x_n\mid\mho)=g_{\lambda}(x_1, \dots, x_n)\prod_{i=1}^n(x_i-\mho(n;0))
\end{equation}
for some $g_{\lambda}(x_1, \dots, x_n)\in \Bbbk[x_1, \dots,x_n]^{S_n}_{\leq |\lambda|}$.

To prove (1) take $j\in\mathbb Z_{\geq 0}$ and set $\lambda=(nj, (n-1)j, \dots, 2j, j)$. Recall that $F_{\lambda+1^n}(\mho(\lambda+1^n)\mid\mho)\neq 0$, so from \eqref{proofOfGridShift} we have
$$
F_{\lambda+1^n}(\mho(\lambda+1^n)\mid\mho)=g_{\lambda}(\mho(\lambda+1^n))  (\mho(n;j+1)-\mho(n;0)) \prod_{i=1}^{n-1}(\mho(i;j)-\mho(n;0))\neq 0.
$$
Hence $\mho(i;j)\neq \mho(n;0)$ for any $i<n$ and $\mho(n;j+1)\neq\mho(n;0)$, which implies (1) since $j\in\Z_{\geq0}$ was arbitrary. 

To prove (2) note that from \eqref{proofOfGridShift} we have for any $\mu\in\Y^n$
\begin{equation}
\label{proofOfGridShiftEx}
g_{\lambda}(\mho^1(\mu))\prod_{i=1}^n(\mho^1(i;\mu_i-\mu_{i+1})-\mho(n;0))=F_{\lambda+1^n}(\mho^1(\mu)\mid\mho)=F_{\lambda+1^n}(\mho(\mu+1^n)\mid\mho).
\end{equation}
By (1) we know that $\mho^1(i;j)-\mho(n;0)=\mho(i;j+\delta_{i,n})-\mho(n;0)\neq 0$ for any $j\geq 0$, hence \eqref{proofOfGridShiftEx} implies that $F_{\lambda+1^n}(\mho(\mu+1^n)\mid\mho)=0$ if and only if $g_{\lambda}(\mho^1(\mu))=0$. In particular, for any $\mu\in\Y^n$ such that $|\mu|\leq|\lambda|$ and $\mu\neq\lambda$ we get $g_{\lambda}(\mho^1(\mu))=0$, while $g_{\lambda}(\mho^1(\lambda))\neq0$. Since the degree of $F_{\lambda+1^n}$ is $|\lambda|+n$, the degree of $g_{\lambda}$ is not greater than $|\lambda|$. So, $g_{\lambda}$ satisfies the defining properties of $F_{\lambda}(x_1, \dots, x_n\mid \mho^1)$, which proves (2).

To prove part (3) assume that $\mho$ is perfect. Then $F_{\lambda+1^n}(\mho(\mu+1^n)\mid\mho)=0$ unless $\lambda\subseteq\mu$, and hence $g_{\lambda}(\mho^1(\mu))=F_{\lambda}(\mho^1(\mu)\mid \mho^1)=0$ unless $\lambda\subseteq\mu$. So the vanishing property for $\mho$ implies the vanishing property for $\mho^1$.
\end{proof}

Using Proposition \ref{uppershiftFirstStep} we can inductively get the following results.

\begin{cor}\label{uppershiftPerf}
If $\mho$ is non-degenerate then $\mho^k$ is non-degenerate for any $k\in\mathbb Z_{\geq0}$. If $\mho$ is perfect then $\mho^k$ is perfect for any $k\in\mathbb Z_{\geq0}$. \qed
\end{cor}

\begin{cor}\label{nondegeneccesary}
If $\mho$ is non-degenerate then 
$
\mho(i;j)\neq\mho(i',j')
$
for all pairs $(i,j)\neq(i',j')$.
\end{cor}
\begin{proof}
Let $\mho$ be a non-degenerate $n$-grid and fix $(i,j)\neq (i',j')$. Without loss of generality we can assume that $i\leq i'$. By Proposition \ref{restrictionGrid} $\mho_{i'}$ is also non-degenerate, hence replacing $\mho$ by $\mho_{i'}$ we can assume that $i'=n$. If $i<n$, then $\mho^{j'}$ is non-degenerate, and by Proposition \ref{uppershiftFirstStep} applied to $\mho^{j'}$
$$
\mho(i;j)=\mho^{j'}(i;j)\neq\mho^{j'}(n;0)=\mho(i';j').
$$
If $i=i'=n$, then we can additionally assume that $j>j'$ and use Proposition \ref{uppershiftFirstStep} to get 
$$
\mho(n;j)=\mho^{j'}(n;j-j')\neq\mho^{j'}(n;0)=\mho(n;j').
$$
\end{proof}

It turns out that the converse to Corollary \ref{nondegeneccesary} is also true, so we can classify all non-degenerate grids:

\begin{prop}\label{nondegsuff} An $n$-grid $\mho$ is non-degenerate if and only if  $\mho(i;j)\neq\mho(i',j')$ for all $(i,j)\neq(i',j')$.
\end{prop}
\begin{proof}
By Corollary \ref{nondegeneccesary} we only need to prove that $\mho$ is non-degenerate if $\mho(i;j)\neq\mho(i',j')$ for all $(i,j)\neq(i',j')$. It is enough to prove the following statement: for any $\mho$ as in the statement, any $m\in\mathbb Z_{\geq 0}$ and any function 
$$
\phi(\lambda):\{\lambda\in\mathbb Y^n: |\lambda|\leq m\}\to \Bbbk
$$
there exists a polynomial $f\in \Bbbk[x_1, \dots, x_n]^{S_n}_{\leq m}$ such that $f(\mho(\lambda))=\phi(\lambda)$ for any $\lambda\in\mathbb Y^n, |\lambda|\leq m$.

We prove the latter claim by induction on $n$ and $m$, reducing the claim for $(n,m)$ to the claims for $(n',m')$ with either $n'<n$ or $m'<m$. Note that the cases when $n=1$ or $m=0$ are trivial.

For the inductive step recall that the monomial symmetric functions $m_\lambda(x_1, \dots, x_n)$ are defined by
$$
m_\lambda(x_1, \dots, x_n)=\sum_{\alpha}x_1^{\alpha_1}\dots x_n^{\alpha_n},
$$
where the sum is over all permutations $\alpha$ of the $n$-tuple $(\lambda_1, \dots, \lambda_n)$. Let $n\geq 2$ and fix $n$-grid $\mho$ as in the claim. Define a degree-preserving map
$$
\Sym: \Bbbk[x_1, \dots, x_{n-1}]^{S_{n-1}}\to \Bbbk[x_1, \dots, x_{n}]^{S_{n}}
$$
by sending $m_{\lambda}(x_1-\mho(n;0), \dots, x_{n-1}-\mho(n;0))$ to $m_{\lambda}(x_1-\mho(n;0), \dots, x_{n}-\mho(n;0))$. Note that for any $f\in \Bbbk[x_1, \dots, x_n]^{S_n}$ we have
\begin{equation}
\label{mainPropOfExt}
(\Sym f)(x_1, \dots, x_{n-1}, \mho(n;0))=f(x_1, \dots, x_{n-1}).
\end{equation}
Let $\phi(\lambda)$ be an arbitrary function on the partitions $\lambda\in\mathbb Y^n, |\lambda|\leq m$. We will construct the required function $f$ as
$$
f=\Sym f_1+ f_2 \prod_{i=1}^n(x_i-\mho(n;0)).
$$ 
Consider the restriction of $\phi$ to the partitions from $\mathbb Y^{n-1}$. By the induction hypothesis, there exists a function in $n-1$ variables $f_1\in \Bbbk[x_1, \dots, x_{n-1}]^{S_{n-1}}_{\leq m}$ such that $f_1(\mho_{n-1}(\lambda))=\phi(\lambda)$ for all partitions $\lambda\in\mathbb Y^{n-1}, |\lambda|\leq m$. Since for all $\lambda\in\mathbb Y^{n-1},|\lambda|\leq m$, the substitution $(x_1,\dots, x_n)=\mho(\lambda)$ sets $x_n=\mho(n;0)$, by \eqref{mainPropOfExt} we have 
$$
(\Sym f_1)(\mho(\lambda))=f_1(\mho_{n-1}(\lambda))=\phi(\lambda), \qquad \lambda\in\Y^{n-1},|\lambda|\leq m.
$$
If $n<m$ then any partition $\lambda$ such that $|\lambda|\leq m$ is in $\Y^{n-1}$ and we are done. If $n\geq m$ consider the following function $\phi_2$ on $\{\lambda\in\mathbb Y^n: |\lambda|\leq m-n\}$:
$$
\phi_2(\lambda)=\frac{\phi(\lambda+1^n)-\Sym f_1(\mho(\lambda+1^n))}{(\mho(n;1)-\mho(n;0))\prod_{i=1}^{n-1}(\mho(i;\lambda_i-\lambda_{i+1})-\mho(n;0))}.
$$
By induction in $m$ and since $\mho^1$ also satisfies the assumptions on $\mho$ from the statement, we can construct a function $f_2\in \Bbbk[x_1, \dots, x_n]^{S_n}_{\leq m-n}$ such that $f_2(\mho(\lambda+1^n))=f_2(\mho^1(\lambda))=\phi_2(\lambda)$. The resulting function
$$
f=\Sym f_1+ f_2 \prod_{i=1}^n(x_i-\mho(n;0))
$$ 
satisfies $f\in\Bbbk[x_1, \dots, x_n]^{S_n}_{\leq m}$ and $f(\mho(\lambda))=\phi(\lambda)$.
\end{proof}

We can use \eqref{restrictionGrideq} and \eqref{uppershiftPexpression} to introduce a natural normalization for the functions $F_{\lambda}$. For an $n$-grid $\mho$ let $\mho^k_m$ denote the $m$-grid $(\mho_m)^k$.
 
\begin{prop}\label{Pnormalization} For a non-degenerate $n$-grid $\mho$ there exists a unique choice of functions $F_{\lambda}(x_1,\dots, x_m\mid\mho_m^k)$ for all $m=1,\dots, n$ and $k\in\mathbb Z_{\geq 0}$ satisfying
\begin{enumerate}
\item $F_{\varnothing}(x_1, \dots, x_m\mid\mho_m^k)=1$ for all $m,k$;
\item $F_{\lambda}$ are consistent with \eqref{restrictionGrideq}, that is, for any $\lambda\in \mathbb Y^{m-1}$
$$
F_{\lambda}(x_1, \dots, x_{m-1}\mid\mho_{m-1})=F_{\lambda}(x_1, \dots, x_{m-1}, \mho(m;k)\mid\mho_{m}^k).
$$
Note that $(\mho_m^k)_{m-1}=\mho_{m-1}$.
\item $F_\lambda$ are consistent with \eqref{uppershiftPexpression}, that is, for any $\lambda\in\mathbb Y^{m}$
$$
F_{\lambda+1^m}(x_1, \dots, x_{m}\mid\mho_{m}^k)=F_{\lambda}(x_1, \dots, x_{m}\mid\mho_{m}^{k+1})\prod_{i=1}^m(x_i-\mho(m;k)).
$$
\end{enumerate}
This unique choice of functions $F_{\lambda}(x_1,\dots, x_m\mid\mho_m^k)$ is determined by setting 
\begin{equation}
\label{hookFromula}
F_\lambda(\mho(\lambda)\mid\mho)=\prod_{r\geq 1}\prod_{i=1}^r\prod_{j=\lambda_{r+1}}^{\lambda_{r}-1}(\mho(i;\lambda_i-\lambda_{i+1})-\mho(r;j-\lambda_{r+1}))
\end{equation}
for any non-degenerate grid $\mho$.
\end{prop}
\begin{proof} We use induction on $n$. When $n=1$ all functions $F_{(\lambda_1)}(x_1\mid\mho^k)$ satisfying the conditions (1) and (3) from the statement must be of the form 
$$
F_{(\lambda_1)}(x_1\mid\mho^k)=\prod_{j=0}^{\lambda_{1}-1}(x_1-\mho(1;j+k))=\prod_{j=0}^{\lambda_{1}-1}(x_1-\mho^k(1;j)),
$$
and the claim follows.

Now assume that $\mho$ is a non-degenerate $n$-grid and we have proved the claim for $\mho_{n-1}$. Note that the functions $F_\lambda(x_1, \dots, x_m\mid\mho^k_m)$ are uniquely determined by the numbers $H_\lambda(\mho^k_m):=F_\lambda(\mho^k_m(\lambda)\mid\mho^k_m)$, defined for $\lambda\in\mathbb Y^m$. Moreover, if the choice of $F_\lambda(x_1, \dots, x_m\mid\mho^k_m)$ satisfies conditions (1)-(3) from the statement, then by the inductive assumption applied to the $(n-1)$-grid $\mho_{n-1}$ we have
$$
H_\lambda(\mho_m^k)=\prod_{r\geq 1}\prod_{i=1}^r\prod_{j=\lambda_{r+1}}^{\lambda_{r}-1}(\mho_m^k(i;\lambda_i-\lambda_{i+1})-\mho_m^k(r;j-\lambda_{r+1}))
$$
for any $k\in\mathbb Z_{\geq 0}$, $m<n$ and $\lambda\in\mathbb Y^m$. So we only need to prove that $H_\lambda(\mho^k)$ are uniquely determined and are given by \eqref{hookFromula}.

Let $\lambda\in\mathbb Y^n$. If $l(\lambda)<n$, then by the consistency with \eqref{restrictionGrideq} we have $H_{\lambda}(\mho^k)=H_\lambda(\mho_{n-1})$, and we are done. If $l(\lambda)=n$, then we can write $\lambda=\overline{\lambda}+a^n$ for some $\overline{\lambda}\in\mathbb Y^{n-1}$ and $a=\lambda_n$. Then, using consistency with \eqref{uppershiftPexpression}, we have
$$
H_{\overline\lambda+a^n}(\mho^k)=H_{\overline{\lambda}}(\mho^{k+a})\prod_{i=1}^n\prod_{j=0}^{a-1}(\mho^{k+a}(i;\overline{\lambda}_i-\overline{\lambda}_{i+1})-\mho^k(n;j)).
$$
Since $\overline{\lambda}\in\mathbb Y^{n-1}$, we already know that $H_{\overline{\lambda}}(\mho^{k+a})$ is givn by \eqref{hookFromula}, and one can readily check that the resulting expression for $H_{\lambda}(\mho^k)=H_{\overline\lambda+a^n}(\mho^k)$ is also consistent with \eqref{hookFromula}.
\end{proof}
From now on we will assume that $\{F_\lambda\}_\lambda$ are normalized as in Proposition \ref{Pnormalization}.

\begin{prop}\label{leftshiftPerfect} Let $\mho$ be a perfect $n$-grid with $n\geq 2$. Then $\prescript{}{1}\mho$ is perfect and we have
$$
F_{\lambda}(x_1, \dots, x_n\mid\mho)=x_1^{\lambda_1} F_{\widetilde{\lambda}}(x_2, \dots, x_n\mid\prescript{}{1}{\mho})+\text{\emph{lower $x_1$-degree terms}},
$$
where $\widetilde{\lambda}=(\lambda_2, \lambda_3,\dots)$.
\end{prop}
\begin{proof}
As before, let $m_\lambda(x_1, \dots, x_n)$ denote monomial symmetric function corresponding to $\lambda$, and let $\mu\preceq\lambda$ denote the lexicographical order on the partitions (that is $\mu\preceq\lambda$ if and only if either $\mu=\lambda$ or for some $r$ we have $\lambda_i=\mu_i$ for $i<r$ and $\mu_r<\lambda_r$).

First we want to prove that if $\mho$ is a non-degenerate $n$-grid and $\lambda\in\mathbb Y^n$ is a partition such that for $\mu\in\Y^n$
\begin{equation}
\label{assumptionProofOfLeftShift}
F_\lambda(\mho(\mu)\mid\mho)=0\qquad \text{unless}\ \lambda\subset\mu,
\end{equation}
then we have
\begin{equation}
\label{Puppertriangular}
F_{\lambda}(x_1, \dots, x_n\mid\mho)=\sum_{\mu\preceq\lambda}\alpha_{\mu\lambda}m_{\mu}(x_1, \dots, x_n).
\end{equation}
We prove the latter statement by induction on $n$, with $n=1$ case being trivial. Fix $\mho,\lambda$ as above and let $d=\deg_{x_1} F_{\lambda}(x_1, \dots, x_n\mid\mho)$. If $d<\lambda_1$ then we are trivially done with $\alpha_{\lambda\lambda}=0$, so consider the case when $d\geq \lambda_1$. Define $g(x_2, \dots, x_n)$ by
$$
F_{\lambda}(x_1, \dots, x_n\mid\mho)=x_1^{d} g(x_2, \dots, x_n)+\text{lower $x_1$-degree terms}.
$$
Clearly, $g\neq 0$. Since the total degree of $F_{\lambda}$ is $\lambda$, the total degree of $g$ satisfies $\deg g= |\lambda|-d\leq |\lambda|-\lambda_1=|\widetilde{\lambda}|$, where $\widetilde{\lambda}=(\lambda_2, \lambda_3,\dots)$. Let $\mu\in\mathbb Y^{n-1}$ be a partition such that $\widetilde{\lambda}\not\subset\mu$. Then the partition $\mu(k)=(\mu_1+k, \mu_1, \mu_2, \dots)$ satisfies $\lambda\not\subset\mu(k)$ for any $k\in\mathbb Z_{\geq 0}$, so by assumption \eqref{assumptionProofOfLeftShift}
$$
F_{\lambda}(\mho(1;k), \mho(2;\mu_1-\mu_2), \dots, \mho(n;\mu_{n-1}))=F_{\lambda}(\mho(\mu(k))\mid\mho)=0
$$
for any $k$. In other words, the polynomial $F_\lambda(x_1, \prescript{}{1}\mho(1;\mu_1-\mu_2), \dots, \prescript{}{1}\mho(n-1;\mu_{n-1})\mid\mho)$ vanishes at $x_1=\mho(1;k)$, and since all these points are distinct by Corollary \ref{nondegeneccesary}, this implies 
$$
F_\lambda(x_1, \prescript{}{1}\mho(1;\mu_1-\mu_2), \dots, \prescript{}{1}\mho(n-1;\mu_{n-1})\mid\mho)=0
$$
as a polynomial in $x_1$. In particular, for $\mu\in\Y^{n-1}$ we get 
\begin{equation}
\label{temporaryvanishing}
g(\prescript{}{1}{\mho}(\mu))=g(\mho(2;\mu_1-\mu_2), \dots, \mho(n;\mu_{n-1}))=0 \qquad \text{unless}\ \widetilde{\lambda}\subset\mu.
\end{equation}
Recall that the degree of $g$ is at most $|\widetilde{\lambda}|$. By Proposition \ref{nondegsuff} the $(n-1)$-grid $\prescript{}{1}{\mho}$ is non-degenerate, so Lemma \ref{basicnondegLemma} and \eqref{temporaryvanishing} imply that $g=c \cdot F_{\widetilde{\lambda}}(x_2, \dots, x_n\mid\prescript{}{1}{\mho})$ for some $c\in\Bbbk$. Note that $g\neq 0$, hence $c\neq 0$ and the degree of $g$ is exactly $|\widetilde{\lambda}|$, forcing $d=\lambda_1$. Moreover, \eqref{temporaryvanishing} implies that the assumption of our statement holds for the pair ($\prescript{}{1}{\mho}, \widetilde{\lambda}$), hence by induction
$$
g(x_2,\dots, x_n)=c\cdot F_{\widetilde{\lambda}}(x_2, \dots, x_n\mid\prescript{}{1}{\mho})=\sum_{\mu\preceq\widetilde{\lambda}}\alpha'_{\mu\widetilde{\lambda}}m_{\mu}(x_2, \dots, x_n).
$$
This implies \eqref{Puppertriangular}:
$$
F_{\lambda}(x_1, \dots, x_n\mid\mho)=\sum_{\mu\preceq\widetilde{\lambda}}\alpha'_{\mu\widetilde{\lambda}}x_1^{\lambda_1}m_{\mu}(x_2, \dots, x_n)+\text{lower $x_1$-degree terms}.
$$

Thus, we have proved that if $\mho$ is perfect, then the transition matrix expressing $F_\lambda(x_1, \dots, x_n\mid\mho)$ in terms of $m_{\lambda}(x_1, \dots, x_n)$ is upper-triangular with respect to the lexicographical order. But both families of functions form bases of the ring of symmetric polynomials, hence the transition matrix must be invertible, and $\alpha_{\lambda\lambda}\neq 0$  in \eqref{Puppertriangular}. In particular, in the argument above $d=\lambda_1$ always holds and the case $d<\lambda_1$ never happens, so the construction of $g(x_2, \dots, x_n)$ can be performed for any $\lambda$: 
$$
F_{\lambda}(x_1, \dots, x_n\mid\mho)=c_{\lambda} x_1^{\lambda_1} F_{\widetilde{\lambda}}(x_2, \dots, x_n\mid\prescript{}{1}{\mho})+\text{lower $x_1$-degree terms}
$$
for $c_{\lambda}\neq 0$. Moreover, \eqref{temporaryvanishing} implies that $\prescript{}{1}{\mho}$ is perfect.

It only remains to show that $c_\lambda=1$. Note that it is enough to show that for perfect $\mho$ we always have $\alpha_{\lambda\lambda}=1$ in \eqref{Puppertriangular}. Recall that if $\mho$ is perfect then all grids of the form $\mho_m^k$ are also perfect, so we can define a renormalization $\widetilde F_{\lambda}(x_1, \dots, x_m\mid\mho_m^k)$ such that 
$$
\widetilde F_{\lambda}(x_1, \dots, x_m\mid\mho_m^k)=m_\lambda(x_1, \dots, x_m) + \sum_{\mu\preceq\lambda,\mu\neq\lambda}\alpha_{\mu\lambda}m_{\mu}(x_1, \dots, x_m)
$$
and $\widetilde F_\lambda$ is proportional to $F_\lambda$. Then the functions $\widetilde F_{\lambda}(x_1, \dots, x_m\mid\mho_m^k)$ satisfy the three conditions of Proposition \ref{Pnormalization}, hence $\widetilde F_{\lambda}=F_\lambda$, finishing the proof.
\end{proof}

\subsection{Pieri rule and explicit expressions for some $F_\lambda$} Our next goal is to prove Theorem \ref{perfectTheorem} when $n=3$. From here on our arguments are completely different from \cite{Ok97}, and we will need to compute polynomials $F_{\lambda}$ for perfect $2$-grids. To do it, we use the following analogue of the Pieri rule:

\begin{lem}\label{sadPieri} Let $\mho$ be a perfect $2$-grid. Then for any $k\in\mathbb Z_{\geq 1}$ we have
\begin{multline*}
(x_1+x_2-\mho(1;k)-\mho(2;0))F_{(k)}(x_1, x_2\mid\mho)\\
=F_{(k+1)}(x_1,x_2\mid \mho)+\kappa_k(\mho) (x_1-\mho(2;0))(x_2-\mho(2;0))F_{(k-1)}(x_1, x_2\mid\mho^1),
\end{multline*}
where
$$
\kappa_k(\mho)=\frac{(\mho(1;k-1)+\mho(2;1)-\mho(1;k)-\mho(2;0))F_{(k)}(\mho(1;k-1), \mho(2;1)\mid\mho)}{(\mho(2;1)-\mho(2;0))(\mho(1;k-1)-\mho(2;0))\prod_{i=0}^{k-2}(\mho(1;k-1)-\mho(1;i))}.
$$
\end{lem} 
\begin{proof}
Fix $k\geq 1$ and set 
$$
f(x_1,x_2)=(x_1+x_2-\mho(1;k)-\mho(2;0))F_{(k)}(x_1, x_2\mid\mho).
$$
Note that $f\in \Bbbk[x_1,x_2]^{S_2}_{\leq k+1}$, and moreover $f(\mho(\mu))=0$ for all $\mu$ such that $|\mu|\leq k$: for $\mu\neq \lambda$ we have $F_{(k)}(\mho(\mu)\mid\mho)=0$, while for $\mu=\lambda$ we set $x_1=\mho(1;k), x_2=\mho(2;0)$ and $x_1+x_2-\mho(1;k)-\mho(2;0)$ vanishes. Hence the expression 
$$
f(x_1, x_2)-\alpha_k F_{(k+1)}(x_1, x_2\mid\mho)- \beta_k F_{(k,1)}(x_1, x_2\mid\mho)
$$
with 
$$
\alpha_k=\frac{f(\mho((k+1)))}{F_{(k+1)}(\mho((k+1))\mid\mho)}, \qquad \beta_k=\frac{f(\mho((k,1)))}{F_{(k,1)}(\mho((k,1))\mid\mho)}
$$
vanishes for any $(x_1, x_2)=\mho(\mu)$ with $|\mu|\leq k+1$, so by Lemma \ref{basicnondegLemma} we have
\begin{equation}
\label{tmpproofPieri}
f(x_1, x_2)=\alpha_k F_{(k+1)}(x_1, x_2\mid\mho)+\beta_k F_{(k,1)}(x_1, x_2\mid\mho).
\end{equation}

Note that by Lemma \ref{leftshiftPerfect} the top $x_1$-degree term on the sides of \eqref{tmpproofPieri} are $x_1^{k+1}$ and $\alpha_kx_1^{k+1}$, hence $\alpha_k=1$. By \eqref{hookFromula}, we have 
$$
F_{(k,1)}(\mho((k,1))\mid\mho)=(\mho(1;k-1)-\mho(2,0))(\mho(2;1)-\mho(2;0))\prod_{i=0}^{k-2}(\mho(1;k-1)-\mho(1;i)),
$$
hence $\beta_k=\kappa_k(\mho)$ from the statement. Finally, by Proposition \ref{uppershiftFirstStep} we have 
$$
F_{(k,1)}(x_1, x_2\mid \mho)=(x_1-\mho(2;0))(x_2-\mho(2;0))F_{(k-1)}(x_1,x_2\mid\mho^1),
$$
finishing the proof.
\end{proof}

Using Lemma \ref{sadPieri} we can get explicit expressions for $F_{(k)}(x_1,x_2\mid\mho)$ when $k=1,2,3$. Let $\mho$ be a perfect $2$-grid. To make expressions below manageable we use the notation $[i;j]:=\mho(i;j)$. By our choice in Proposition \ref{Pnormalization}, we clearly have $F_{\varnothing}(x_1,x_2\mid\mho)=1$. For $F_{(1)}(x_1,x_2\mid\mho)$ recall that this is a degree $1$ symmetric polynomial vanishing at $(x_1,x_2)=\mho(\varnothing)$ and with top term $x_1+x_2$ by Proposition \eqref{leftshiftPerfect}. Hence 
$$
F_{(1)}(x_1,x_2\mid\mho) =x_1+x_2-[1;0]-[2;0].
$$
From now we can use Lemma \ref{sadPieri}. Direct computations give
$$
\kappa_1(\mho)=\frac{[1;0]+[2;1]-[1;1]-[2;0]}{[1;0]-[2;0]};
$$
\begin{multline}
\label{P2}
F_{(2)}(x_1,x_2\mid\mho)=x_1^2+x_2^2+\frac{[1;0]+[1;1]-[2;0]-[2;1]}{[1;0]-[2;0]}x_1x_2-\frac{[1;0]^2+[1;0][1;1]-[2;0]^2-[2;0][2;1]}{[1;0]-[2;0]}(x_1+x_2)\\+\frac{[1;0]^2[1;1]+[1;0]^2[2;0]-[1;0][2;0]^2-[2;0]^2[2;1]}{[1;0]-[2;0]};
\end{multline}
$$
\kappa_2(\mho)=\frac{([1;1]+[2;1]-[1;2]-[2;0])([1;0]+[1;1]-[2;0]-[2;1])}{([1;0]-[2;0])([1;1]-[2;0])};
$$
\begin{multline}
\label{P3}
F_{(3)}(x_1,x_2\mid\mho)=(x_1+x_2-[1;2]-[2;0])F_{(2)}(x_1,x_2\mid\mho)-\kappa_2(\mho)(x_1-[2;0])(x_2-[2;0])(x_1+x_2-[1;0]-[2;1]).
\end{multline}

For Lemma \ref{31expression} below we also need an explicit expression for $F_{(2)}(x_1, x_2, x_3\mid\mho)$ when $\mho$ is a perfect $3$-grid such that $\mho(3;0)=0$. To get $F_{(2)}(x_1, x_2,x_3\mid\mho)$ we can use the expression for $F_{(2)}(x_1, x_2\mid\mho_2)$ above and follow the construction from Proposition \ref{nondegsuff} to get 
\begin{multline}
\label{P23var}
F_{(2)}(x_1,x_2,x_3\mid\mho)=\Sym F_{(2)}(x_1,x_2\mid\mho_2)=x_1^2+x_2^2+x_3^2+\frac{[1;0]+[1;1]-[2;0]-[2;1]}{[1;0]-[2;0]}(x_1x_2+x_1x_3+x_2x_3)\\
-\frac{[1;0]^2+[1;0][1;1]-[2;0]^2-[2;0][2;1]}{[1;0]-[2;0]}(x_1+x_2+x_3)+\frac{[1;0]^2[1;1]+[1;0]^2[2;0]-[1;0][2;0]^2-[2;0]^2[2;1]}{[1;0]-[2;0]}.
\end{multline}

\subsection{$n=3$ case} Now we can prove Theorem \ref{perfectTheorem} for perfect $3$-grids. The main idea is to use vanishings $F_{\lambda}(\mho(\mu)\mid\mho)=0$ for various $\lambda\not\subset\mu$ to obtain enough constraints on $\mho(i;j)$ for the desired classification.

\begin{lem}\label{12lemma} Let $\mho$ be a perfect $n$-grid. Then for any $i,j\in\mathbb Z_{\geq 1}$ such that $i\leq n-1$ we have
\begin{equation}
\label{12avarage}
\(\mho(i;1)-\mho(i+1;j-1)\)(\mho(i;1)-\mho(i+1;j+1))=(\mho(i;0)-\mho(i+1;j))(\mho(i;2)-\mho(i+1;j)).
\end{equation}
\end{lem}
\begin{proof}
Replacing $\mho$ by the $2$-grid $\prescript{}{i-1}{\mho}_{i+1}$, which is perfect by Propositions \ref{restrictionGrid}, \ref{leftshiftPerfect}, we can assume that $i=1$ and $\mho$ is a perfect $2$-grid. Moreover, replacing $\mho$ by $\mho^{j-1}$, which is perfect by Corollary \ref{uppershiftPerf}, we can assume that $j=1$. 

So, it is enough to prove that for a perfect $2$-grid $\mho$ we have
\begin{equation*}
\(\mho(1;1)-\mho(2;0)\)(\mho(1;1)-\mho(2;2))=(\mho(1;0)-\mho(2;1))(\mho(1;2)-\mho(2;1)).
\end{equation*}
This follows from the vanishing $F_{(3)}(\mho((2,2))\mid\mho)=0$: by explicit computation using \eqref{P3} we have
$$
F_{(3)}(\mho((2,2))\mid\mho)=F_{(3)}(\mho(1;0), \mho(2;2)\mid\mho)=\frac{(\mho(2;0)-\mho(2;2))(\mho(2;1)-\mho(2;2))}{\mho(1;1)-\mho(2;0)}G,
$$
where
$$
G=(\mho(1;0)-\mho(2;1))(\mho(1;2)-\mho(2;1))-\(\mho(1;1)-\mho(2;0)\)(\mho(1;1)-\mho(2;2)).
$$
Since $\mho$ is non-degenerate, $\mho(i;j)\neq\mho(i';j')$ as long as $(i,j)\neq(i',j')$, hence $F_{(3)}(\mho((2,2))\mid\mho)=0$ implies $G=0$ and the claim follows.
\end{proof}

From now on, let $\mho$ be a perfect $3$-grid.

\begin{lem}\label{31lemma} We have
\begin{equation}
\label{31expression}
[3;1]=\frac{[1;0][2;1]-[1;1][2;0]-[2;1][3;0]+[1;1][3;0]}{[1;0]-[2;0]}.
\end{equation}
\end{lem}
\begin{proof} We first note that if a grid $\mho$ is perfect then for any constant $c\in k$ the grid $\mho+c$, defined by $(\mho+c)(i;j)=\mho(i;j)+c$, is also perfect with 
$$
F_{\lambda}(x_1, x_2, x_3\mid\mho+c)=F_{\lambda}(x_1-c, x_2-c, x_3-c\mid\mho).
$$
On the other hand, the desired identity \eqref{31expression} is equivalent to
$$
([2;1]-[3;1])([1;0]-[2;0])-([1;1]-[2;1])([2;0]-[3;0])=0.
$$
The left-hand side above clearly stays intact when we change $\mho$ to $\mho+c$, so it is enough to prove the claim for the perfect $3$-grid $\mho-[3;0]$. In other words, it is enough to consider the case $[3;0]=0$.

Since $\mho$ is perfect, we must have $F_{(2)}(\mho((1,1,1))\mid\mho)=0$. Using \eqref{P23var}, we obtain by explicit computation
$$
F_{(2)}([1;0],[2;0],[3;1]\mid\mho)=\frac{[3;1]([1;0][3;1]-[2;0][3;1]+[1;1][2;0]-[1;0][2;1])}{[1;0]-[2;0]}.
$$
Since $\mho$ is non-degenerate and $[3;0]=0$, we have $[3;1]\neq 0$. Hence $F_{(2)}(\mho((1,1,1))\mid\mho)=0$ implies
$$
[1;0][3;1]-[2;0][3;1]+[1;1][2;0]-[1;0][2;1]=0,
$$
which is equivalent to the claim when $[3;0]=0$.
\end{proof}

\begin{lem}\label{finalconstraint}
We have
$$
([2;1]-[1;2])([2;0]-[1;0])=([2;1]-[1;1])([2;0]-[1;1]).
$$
\end{lem}
\begin{proof} Similarly to the proof of Lemma \ref{31lemma}, the claim for the grid $\mho$ is equivalent to the claim for a grid $\mho+c$ for arbitrary $c$. So, to simplify computations, we can assume that $[3;0]=0$ throughout the proof without loss of generality.

To prove the claim we consider two different expressions for $[3;2]$ in terms of $[1;0], [1;1], [1;2], [2;0], [2;1]$, and show that equating them implies the claim. For the first expression we use Lemma \ref{31lemma} for the grids $\mho$ and $\mho^1$, obtaining
$$
[3;1]=\frac{[1;0][2;1]-[1;1][2;0]}{[1;0]-[2;0]},
$$
$$
[3;2]=\frac{[1;0][2;1]-[1;1][2;0]-[2;1][3;1]+[1;1][3;1]}{[1;0]-[2;0]}.
$$
The other expression comes from Lemma \ref{12lemma} applied to $(i,j)=(1,1)$ and $(i,j)=(2,1)$, leading to
$$
[2;2]=[1;1]+\frac{([2;1]-[1;0])([2;1]-[1;2])}{[2;0]-[1;1]},
$$
$$
[3;2]=[2;1]-\frac{([3;1]-[2;0])([3;1]-[2;2])}{[2;1]}.
$$
Subtracting the two expressions for $[3;2]$ from each other we get
$$
\frac{[1;0][2;1]-[1;1][2;0]-[2;1][3;1]+[1;1][3;1]}{[1;0]-[2;0]}-[2;1]+\frac{([3;1]-[2;0])([3;1]-[2;2])}{[2;1]}=0,
$$
which, after algebraic manipulations, gives
$$
\frac{([3;1]-[2;0])([1;1][2;1]-[2;1]^2+[3;1]([1;0]-[2;0])-[2;2]([1;0]-[2;0]))}{[2;1]([1;0]-[2;0])}=0.
$$
Since $[3;1]\neq [2;0]$, we get
$$
[1;1][2;1]-[2;1]^2+[3;1]([1;0]-[2;0])-[2;2]([1;0]-[2;0])=0.
$$
Plugging the values for $[3;1],[2;2]$ above results in
$$
-([2;1]-[1;1])([2;1]-[1;0])+\frac{([2;1]-[1;0])([2;1]-[1;2])([1;0]-[2;0])}{[1;1]-[2;0]}=0,
$$
and since $[2;0]\neq[1;1]$ and $[2;1]\neq[1;0]$, we finally get
$$
-([2;1]-[1;1])([2;0]-[1;1])+([2;1]-[1;2])([2;0]-[1;0])=0.
$$
\end{proof}

The constraints proved above are almost sufficient to prove the classification for $n=3$. Namely, we can now prove the following:
\begin{lem}\label{3classification} One of the following cases holds for the perfect $3$-grid $\mho$:
\begin{enumerate}
\item For some constants $c,q,a_1, a_2, a_3\in\Bbbk$ we have 
$$
[1;0]=c+a_1, \qquad [1;1]= c+a_1q,\qquad [1;2]=c+a_1q^2,
$$
$$
[2;k]=c+a_2 q^k, \quad [3;k]= c+a_3q^k,\qquad k\in\mathbb Z_{\geq 0}.
$$

\item For some constants $d,c_1,c_2,c_3\in\Bbbk$ we have
$$
[1;0]=c_1, \qquad [1;1]= c_1+d,\qquad [1;2]=c_1+2d,
$$
$$
[2;k]=c_2+kd, \quad [3;k]= c_3+kd,\qquad k\in\mathbb Z_{\geq 0}.
$$
\end{enumerate}
\end{lem}
Note that in order to satisfy the non-degeneracy condition from Corollary \ref{nondegeneccesary}, we have to assume that $q\neq 0$ and $q$ is not a root of unity in the first case. 
\begin{proof}
Set
$$
q:=\frac{[2;1]-[1;1]}{[2;0]-[1;0]}.
$$
We have $2$ cases. 

{\bf Case 1: $q\neq 1$.}  Set 
$$
a_1:=\frac{[1;0]-[1;1]}{1-q},\qquad c:=[1;0]-a_1, \qquad a_2:=[2;0]-c, \qquad a_3:=[3;0]-c,
$$
so we have
$$
[1;0]=c+a_1,\qquad [1;1]=c+a_1q, \qquad [2;0]:=c+a_2,\qquad [2;1]=c+a_2q,\qquad [3;0]=c+a_3.
$$
Plugging these values into the constraint from Lemma $\ref{finalconstraint}$ we get
$$
(c+a_2q-[1;2])(a_2-a_1)=(a_2q-a_1q)(a_2-a_1q)=(a_2q-a_1q^2)(a_2-a_1).
$$
Note that since $[1;0]\neq[2;0]$, we have $a_1\neq a_2$, and hence the equation above implies $[1;2]=c+a_1q^2$.

To prove that $[2;k]=c+a_2q^k$ we use induction on $k$, with cases $k=0,1$ already covered. For the inductive step, assume that $[2;k-1]=c+a_2q^{k-1}, [2;k]=c+a_2q^k$, and apply Lemma \ref{12lemma} with $(i,j)=(1,k)$. This gives
$$
([2;k+1]-c-a_1q)(a_2q^{k-1}-a_1q)=(a_2q^k-a_1)(a_2q^k-a_1q^2)=(a_2q^{k+1}-a_1q)(a_2q^{k-1}-a_1q).
$$
Since $[2;k-1]\neq[1;1]$ implies $a_2q^{k-1}-a_1q\neq0$, we get $[2;k+1]=c+a_2q^{k+1}$.

Similarly, we can use induction to prove that $[3;k]=c+a_3q^k$. We already know the case $k=0$. Applying Lemma \ref{31lemma} to the perfect grid $\mho^k$, we get
$$
[3;k+1]=\frac{[1;0][2;1]-[1;1][2;0]-[2;1][3;k]+[1;1][3;k]}{[1;0]-[2;0]}=c(1-q)+q[3;k],
$$
which implies the inductive step and finishes the proof of this case.

{\bf Case 2: $q=1$.} The argument is similar to the previous case. Set
$$
c_1:=[1;0],\qquad c_2:=[2;0], \qquad c_3:=[3;0], \qquad d:=[1;1]-[1;0]=[2;1]-[2;0],
$$
where the last equality follows from $q=1$. Then we have
$$
[1;0]=c_1,\qquad [1;1]=c_1+d, \qquad [2;0]=c_2,\qquad [2;1]=c_2+d,\qquad [3;0]=c_3.
$$
Plugging these values into Lemma $\ref{finalconstraint}$ we get
$$
(c_2+d-[1;2])(c_2-c_1)=(c_2-c_1)(c_2-c_1-d).
$$
Since $[1;0]\neq[2;0]$, we have $c_1\neq c_2$ and hence the equation above implies $[1;2]=c_1+2d$.

To prove that $[2;k]=c_2+kd$ and $[3;k]=c_3+kd$ we can use Lemma \ref{12lemma} and Lemma \ref{31lemma} in exactly the same way as in Case 1.
\end{proof}

At this point, in order to reach the classification from Theorem \ref{perfectTheorem} for $n=3$, we only need to determine the values of $[1;k]$ for $k\geq 3$. To do so we focus on the $2$-grid $\mho_2$ and, considering each case separately, compute the functions $F_{(k)}(x_1,x_2\mid\mho_2)$. The result is summarized below.

\begin{lem} \label{finishmho1}Let $\tmho$ be a perfect $2$-grid.
\begin{enumerate}
\item Assume that for some $q,c,a_1,a_2\in\Bbbk$ we have
$$
\tmho(1;0)=c+a_1,\qquad \tmho(1;1)=c+a_1q,\qquad \tmho(2;k)=c+a_2q^k,\quad k\in\mathbb Z_{\geq 0}.
$$
Then we have $\tmho(1;k)=c+a_1q^k$ for all $k$ and
\begin{equation}
\label{qPkexplicit}
F_{(k)}(x_1+c, x_2+c\mid\tmho)=\sum_{i=0}^k x_1^ix_2^{k-i}\frac{(a_1/x_1;q)_i(a_2/x_2;q)_{k-i}(q;q)_k}{(q;q)_i(q;q)_{k-i}},
\end{equation}
where we have shifted the variables $x_1,x_2$ by $c$.
\item  Assume that for some $d,c_1,c_2\in\Bbbk$ we have
$$
\tmho(1;0)=c_1,\qquad \tmho(1;1)=c_1+d,\qquad \tmho(2;k)=c_2+kd,\quad k\in\mathbb Z_{\geq 0}.
$$
Then we have $\tmho(1;k)=c_1+kd$ for all $k$ and
\begin{equation}
\label{dPkexplicit}
F_{(k)}(x_1, x_2\mid\tmho)=\prod_{i=0}^{k-1}(x_1+x_2-c_1-c_2-i\, d).
\end{equation}
\end{enumerate}
\end{lem}
\begin{proof} For both parts we simultaneously prove the expressions for $\tmho(1;k)$ and $F_{(k+1)}(x_1,x_2\mid\tmho)$ by induction on $k$.

{\bf Part 1.} Note that without loss of generality we can replace $\tmho$ by $\tmho-c$ and assume that $c=0$. When $k=0$ we have $\tmho(1;0)=a_1$ and
$$
F_{(1)}(x_1,x_2\mid\tmho)=x_1+x_2-\tmho(1;0)-\tmho(2;0)=x_1-a_1+x_2-a_2.
$$
When $k=1$ we have $\tmho(1;1)=a_1q$ and from \eqref{P2} one can obtain
$$
F_{(2)}(x_1, x_2\mid\tmho)=(x_1-a_1)(x_1-a_1q)+(1+q)(x_1-a_1)(x_2-a_2)+(x_2-a_2)(x_2-a_2q).
$$

Now assume that $k\geq 2$ and we have proved the claim for $\tmho(1;i)$ and $F_{(i+1)}(x_1, x_2\mid\tmho)$ with $i<k$. Applying the Pieri rule from Lemma \ref{sadPieri}, we have
\begin{multline*}
(x_1+x_2-\tmho(1;k)-a_2)F_{(k)}(x_1, x_2\mid\tmho)\\
=F_{(k+1)}(x_1,x_2\mid \tmho)+\kappa_k(\tmho) (x_1-a_2)(x_2-a_2)F_{(k-1)}(x_1, x_2\mid\tmho^1),
\end{multline*}
where
\begin{equation}
\label{kappatempexpressions}
\kappa_k(\tmho)=\frac{(a_1q^{k-1}+a_2q-\tmho(1;k)-a_2)F_{(k)}(a_1q^{k-1}, a_2q\mid\tmho)}{(a_2q-a_2)(a_1q^{k-1}-a_2)\prod_{i=0}^{k-2}(a_1q^{k-1}-a_1q^i)}.
\end{equation}
Note that when for any $j=0, \dots, k$ we plug $x_1=a_1q^{k-j}$ and $x_2=a_2q^j$ in the right-hand side of \eqref{qPkexplicit}, there is only one non-zero term in the sum and this term corresponds to $i=j$. So, we have
\begin{equation}
\label{valueofPk}
F_{(k)}(\mho((k,j))\mid\tmho)=F_{(k)}(a_1q^{k-j},a_2q^j\mid\tmho)=a_1^{k-j}a_2^jq^{(k-j)^2+j^2}\frac{(q^{-k+j};q)_{k-j}(q^{-j};q)_j(q;q)_k}{(q;q)_{k-j}(q;q)_j}.
\end{equation}
Plugging this expression into \eqref{kappatempexpressions}, we get
$$
\kappa_k(\tmho)=\frac{(a_1q^{k-1}+a_2q-\tmho(1;k)-a_2)(1-q^k)}{(a_1q^{k-1}-a_2)(1-q)},
$$
and, consequently,
\begin{multline}
\label{Piereitemp11}
F_{(k+1)}(x_1,x_2\mid \tmho)=(x_1+x_2-\tmho(1;k)-a_2)F_{(k)}(x_1, x_2\mid\tmho)\\
-\frac{(a_1q^{k-1}+a_2q-\tmho(1;k)-a_2)(1-q^k)}{(a_1q^{k-1}-a_2)(1-q)} (x_1-a_2)(x_2-a_2)F_{(k-1)}(x_1, x_2\mid\tmho^1).
\end{multline}

Since $\tmho$ is perfect, we must have $F_{(k+1)}(\tmho((k,2))\mid \tmho)=F_{(k+1)}(a_1q^{k-2},a_2q^2\mid \tmho)=0$. Hence the expression above implies
\begin{multline*}
(a_1q^{k-2}+a_2q^{2}-\tmho(1;k)-a_2)F_{(k)}(\tmho((k,2))\mid\tmho)\\
=\frac{(a_1q^{k-1}+a_2q-\tmho(1;k)-a_2)(1-q^k)}{(a_1q^{k-1}-a_2)(1-q)} (a_1q^{k-2}-a_2)(a_2q^2-a_2)F_{(k-1)}(\tmho^1((k-1,1))\mid\tmho^1).
\end{multline*}
Applying \eqref{valueofPk} to $F_{(k)}$ and $F_{(k-1)}$, noting for the latter that $\mho^1$ also satisfies the assumption of the first part with $a_2$ multiplied by $q$, we get after elementary manipulations
$$
(a_1q^{k-2}+a_2q^{2}-\tmho(1;k)-a_2)(a_1q^{k-1}-a_2)(1-q)=(a_1q^{k-1}+a_2q-\tmho(1;k)-a_2) (a_1q^{k-2}-a_2)(1-q^2).
$$
This is a linear equation in $\tmho(1;k)$ of the form
$$
(q-1)(a_1q^{k-2}-a_2q)\mho(1;k)=\beta
$$
for some $\beta$ independent of $\tmho(1;k)$. Since $\tmho$ is non-degenerate, $q\neq 1$ and $a_1q^{k-2}=\tmho(1;k-2)\neq\tmho(2;1)=a_2q$, hence the system is non-degenerate and it can be readily checked that $\tmho(1;k)=a_1q^k$ is its unique solution. So we have proved $\tmho(1;k)=a_1q^k$.

Returning to \eqref{Piereitemp11}, plugging $\tmho(1;k)=a_1q^k$ we now have
$$
\kappa_k(\tmho)=1-q^k,
$$
\begin{multline*}
F_{(k+1)}(x_1,x_2\mid \tmho)=(x_1+x_2-a_1q^k-a_2)F_{(k)}(x_1, x_2\mid\tmho)\\
-(1-q^k)(x_1-a_2)(x_2-a_2)F_{(k-1)}(x_1, x_2\mid\tmho^1).
\end{multline*}
Plugging expressions for $F_{(k)}$ and $F_{(k-1)}$ and doing elementary manipulations, we get
\begin{multline*}
F_{(k+1)}(x_1,x_2\mid \tmho)=(x_1+x_2-a_1q^k-a_2)\sum_{i=0}^k x_1^ix_2^{k-i}\frac{(a_1/x_1;q)_i(a_2/x_2;q)_{k-i}(q;q)_k}{(q;q)_i(q;q)_{k-i}}\\
-(x_1-a_2)\sum_{i=0}^{k-1} x_1^ix_2^{k-i}\frac{(a_1/x_1;q)_i(a_2/x_2;q)_{k-i}(q;q)_{k}}{(q;q)_i(q;q)_{k-i-1}},
\end{multline*}
which can be rewritten as
\begin{multline*}
F_{(k+1)}(x_1,x_2\mid \tmho)=\sum_{i=0}^k (x_1q^{k-i}+x_2-a_1q^k-a_2q^{k-i}) x_1^ix_2^{k-i}\frac{(a_1/x_1;q)_i(a_2/x_2;q)_{k-i}(q;q)_k}{(q;q)_i(q;q)_{k-i}}\\
=\sum_{i=0}^k q^{k-i}x_1^{i+1}x_2^{k-i}\frac{(a_1/x_1;q)_{i+1}(a_2/x_2;q)_{k-i}(q;q)_k}{(q;q)_i(q;q)_{k-i}}   + x_1^ix_2^{k-i+1}\frac{(a_1/x_1;q)_i(a_2/x_2;q)_{k-i+1}(q;q)_k}{(q;q)_i(q;q)_{k-i}}\\
=\sum_{i=0}^{k+1} x_1^{i}x_2^{k-i}\frac{(a_1/x_1;q)_{i}(a_2/x_2;q)_{k-i}(q;q)_{k+1}}{(q;q)_i(q;q)_{k-i+1}},
\end{multline*}
finishing the proof of this part.

{\bf Part 2.} We again proceed by induction on $k$. When $k=0$ we have $\tmho(1;0)=a_1$ and
$$
F_{(1)}(x_1,x_2\mid\tmho)=x_1+x_2-\tmho(1;0)-\tmho(2;0)=x_1-c_1+x_2-c_2.
$$
Similarly, when $k=1$ we have $\tmho(1;0)=a_1+d$ and \eqref{P2} implies
$$
F_{(2)}(x_1,x_2\mid\tmho)=(x_1-c_1+x_2-c_2)(x_1-c_1+x_2-c_2-d).
$$

Now assume that $k\geq2$ and we have proved the claim for $\tmho(1;i)$ and $F_{(i+1)}(x_1, x_2\mid\tmho)$ with $i< k$. Applying Lemma \ref{sadPieri} again, we have
\begin{multline*}
(x_1+x_2-\tmho(1;k)-c_2)F_{(k)}(x_1, x_2\mid\tmho)\\
=F_{(k+1)}(x_1,x_2\mid \tmho)+\kappa_k(\tmho) (x_1-c_2)(x_2-c_2)P_{(k-1)}(x_1, x_2\mid\tmho^1),
\end{multline*}
where
$$
\kappa_k(\tmho)=\frac{(c_1+kd-\tmho(1;k))F_{(k)}(c_1+(k-1)d, c_2+d\mid\tmho)}{(c_1+(k-1)d-c_2)d^{k}(k-1)!}.
$$
Plugging the expression for $F_{(k)}(x_1,x_2\mid\tmho)$, we get
$$
\kappa_k(\tmho)=\frac{k(c_1+kd-\tmho(1;k))}{(c_1+(k-1)d-c_2)}.
$$
\begin{multline}
\label{Pieritemp22}
F_{(k+1)}(x_1,x_2\mid \tmho)=(x_1+x_2-\tmho(1;k)-c_2)F_{(k)}(x_1, x_2\mid\tmho)\\
-\frac{k(c_1+kd-\tmho(1;k))}{(c_1+(k-1)d-c_2)}(x_1-c_2)(x_2-c_2)F_{(k-1)}(x_1, x_2\mid\tmho^1).
\end{multline}
Since $F_{(k+1)}(\tmho((k,2))\mid \tmho)=F_{(k+1)}(c_1+(k-2)d,c_2+2d\mid \tmho)=0$, we get 
\begin{multline*}
(x_1+x_2-\tmho(1;k)-c_2)F_{(k)}(c_1+(k-2)d,c_2+2d\mid\tmho)\\
=\frac{k(c_1+kd-\tmho(1;k))}{(c_1+(k-1)d-c_2)}(x_1-c_2)(x_2-c_2)F_{(k-1)}(c_1+(k-2)d,c_2+2d\mid\tmho^1).
\end{multline*}
Using expressions for $F_{(k)}$ and $F_{(k-1)}$, we get after elementary manipulations
$$
(c_1+kd-\tmho(1;k))=2(c_1+(k-2)d-c_2)\frac{(c_1+kd-\tmho(1;k))}{(c_1+(k-1)d-c_2)}.
$$
which is equivalent to
$$
(c_1+kd-\tmho(1;k))(c_1+(k-3)d-c_2)=0.
$$
Note that $c_1+(k-3)t-c_2=\tmho(1;k-1)-\tmho(2;4)\neq0$ since $\tmho$ is non-degenerate. Hence the equation above implies $\tmho(1;k)=c_1+kd$. Plugging this value for $\tmho(1;k)$ into \eqref{Pieritemp22} we get
$$
F_{(k+1)}(x_1,x_2\mid \tmho)=(x_1+x_2-c_1-c_2-kd)F_{(k)}(x_1, x_2\mid\tmho),
$$
which implies the claim.
\end{proof}

Let us summarize Lemmas \ref{3classification} and \ref{finishmho1}:

\begin{cor}\label{n3casecor}
Let $\mho$ be a perfect $3$-grid. Then one of the following holds:
\begin{enumerate}
\item For some constants $c,q,a_1, a_2, a_3\in\Bbbk$ we have 
$$
\mho(1;k)=c+a_1q^k, \quad \mho(2;k)=c+a_2 q^k, \quad \mho(3;k)= c+a_3q^k,\qquad k\in\mathbb Z_{\geq 0}.
$$

\item For some constants $d,c_1,c_2,c_3\in\Bbbk$ we have
$$
\mho(1;k)=c_1+kd,\quad \mho(2;k)=c_2+kd, \quad \mho(3;k)= c_3+kd,\qquad k\in\mathbb Z_{\geq 0}.
$$
\end{enumerate}
\qed
\end{cor}

\subsection{Proof of Theorem \ref{perfectTheorem}} We first prove that two types of $n$-grids described in Theorem \ref{perfectTheorem} are indeed perfect. Propositions  \ref{degvanishing}, \ref{degchartheo} combined with Remark \ref{specremark} imply that the following $n$-grid $\mho$ is perfect:
$$
\mho(i;j)=a_iq^j
$$
for $q,a_1, \dots, a_n\in\Bbbk$ such that $a_i,q\neq0$, $q$ is not a root of unity, and $a_i/a_j\neq q^k$ for any $i\neq j$ and $k\in\Z$. Hence the grid $\mho+c$ is also perfect for any $c\in\Bbbk$, covering the first type from Theorem \ref{perfectTheorem}. Similarly, Propositions \ref{degdegvanishing}, \ref{degdegchartheo} and Remark \ref{specremark} imply that the grids of the form $\mho(i;j)=c_i+jd$ are perfect for $c_1, \dots, c_n, d\in\Bbbk$ such $d\neq 0$ and $c_i-c_j\neq kd$ for any $i\neq j$ and $k\in\Z$.

Hence, it is enough to prove that all perfect $n$-grids for $n\geq 3$ are described by Theorem \ref{perfectTheorem}. We do it by induction on $n$, where the base case $n=3$ follows from Corollary \ref{n3casecor}. Note that  the restrictions $a_i,q\neq 0$, $a_i/a_j\neq q^k$ and $c_i-c_j\neq kd$ follow automatically from Corollary \ref{nondegeneccesary}, so its enough to just prove that $\mho$ has the form $\mho(i;j)=c+a_iq^j$ or $\mho(i;j)=c_i+jd$.

For the inductive step assume that we have proved Theorem \ref{perfectTheorem} for all $(n-1)$-grids $\mho$. Let $\mho$ be a perfect $n$-grid. Then the $(n-1)$-grid $\mho_{n-1}$ is also perfect, and by inductive assumption we either have $\mho(i;j)=c+a_iq^j$ or $\mho(i;j)=c_i+jd$ for $i\leq n-1$. At the same time, the $3$-grid $\prescript{}{n-3}\mho$ is also perfect, and so we either have 
$$
\mho(n-2;j)=\widetilde{c}+\widetilde{a}_{n-2}\widetilde{q}^j,\qquad \mho(n-1;j)=\widetilde{c}+\widetilde{a}_{n-1}\widetilde{q}^j,\qquad \mho(n;j)=\widetilde{c}+\widetilde{a}_{n}\widetilde{q}^j
$$
for some $\widetilde{c},\widetilde{q}, \widetilde{a}_{n-2}, \widetilde{a}_{n-1}, \widetilde{a}_{n}\in\Bbbk$, or
$$
\mho(n-2;j)=\widetilde{c}_{n-2}+j\widetilde{d},\qquad \mho(n-1;j)=\widetilde{c}_{n-1}+j\widetilde{d},\qquad \mho(n;j)=\widetilde{c}_{n}+j\widetilde{d}
$$
for some $\widetilde{d}, \widetilde{c}_{n-2}, \widetilde{c}_{n-1}, \widetilde{c}_{n}\in\Bbbk$. But note that for $c,a, q,c',d\in\Bbbk$ we can simultaneously have $\mho(n-1; k)=c+a q^k=c'+kd$ for all $k\in\Z_{\geq 0}$ only if $d=0$ and all numbers $c+a q^k, c'+kd$ are equal. However, this is impossible since $\mho$ is non-degenerate and the numbers $\mho(n-1;k)$ are distinct. Hence, both grids $\mho_{n-1}$ and $\prescript{}{n-3}{\mho}$ are of the same type, and moreover $q=\widetilde{q}$ or $d=\widetilde{d}$. Thus, if $\mho(i;j)=c+a_iq^j$ for $i<n$ then $\mho(n;j)=c+a_nq^j$ as well, and similarly for $\mho(i;j)=c_i+jd$. This finishes the proof of Theorem \ref{perfectTheorem}.\qed

\end{document}